\documentclass[11pt,a4paper,reqno,oneside]{amsart}
\pdfoutput=1
\usepackage{times}
\usepackage
{geometry}

\usepackage[dvipsnames]{xcolor}
\usepackage[english]{babel} 

\definecolor{bred}{rgb}{0.8,0,0}

\usepackage{amsthm,amsmath,amssymb,bbm,geometry,epsfig,listings,paralist,enumerate,comment,nicefrac,verbatim,multirow,xcolor,wasysym,tikz}
\usepackage[pagewise]{lineno}[4.41]
\usepackage{mathrsfs}  
\usepackage{marginnote}\reversemarginpar

\usepackage[linktocpage,colorlinks,hidelinks,linkcolor=red,citecolor=green]{hyperref}

\usepackage[final,notref, notcite]{showkeys}
\usepackage[capitalise,compress]{cleveref} 

\crefname{equation}{}{}
\newtheorem{lemma}{Lemma}[section]

\newtheorem{proposition}[lemma]{Proposition}
\newtheorem{theorem}[lemma]{Theorem}

\newtheorem{setting}[lemma]{Setting}
\crefname{subsection}{Subsection}{Subsections}
\crefname{enumi}{item}{items}
\creflabelformat{enumi}{(#2\textup{#1}#3)}

\newcommand{\1}{\ensuremath{\mathbbm{1}}}
\providecommand{\N}{{\ensuremath{\mathbbm{N}}}}
\providecommand{\Z}{{\ensuremath{\mathbbm{Z}}}}
\providecommand{\R}{{\ensuremath{\mathbbm{R}}}}

\renewcommand{\P}{{\ensuremath{\mathbbm{P}}}}
\providecommand{\E}{{\ensuremath{\mathbbm{E}}}}
\providecommand{\F}{{\ensuremath{\mathbbm{F}}}}

\providecommand{\bfD}{{\ensuremath{\mathbf{D}}}}
\providecommand{\calR}{{\ensuremath{\mathcal{R}}}}
\providecommand{\calD}{{\ensuremath{\mathcal{D}}}}
\providecommand{\bfN}{{\ensuremath{\mathbf{N}}}}
\newcommand{\rdown}[1]{\lfloor #1\rfloor}
\newcommand{\calX}{\mathcal{X}}
\newcommand{\calP}{\mathcal{P}}
\newcommand{\xeqref}[1]{}
\newcommand{\supnorm}[1]{{\left\vert\kern-0.25ex\left\vert\kern-0.25ex\left\vert #1 
    \right\vert\kern-0.25ex\right\vert\kern-0.25ex\right\vert}}

\title[DNNs overcome curse of dimensionality when approximating semilinear PIDEs]{Deep ReLU neural networks overcome  the curse of dimensionality\\ when approximating \\semilinear partial integro-differential equations}

\author[A. Neufeld]{Ariel Neufeld$^{1}$}
\address{$^1$  Division of Mathematical Sciences, School of Physical and Mathematical Sciences, Nanyang Technological University, Singapore}
\email{ariel.neufeld@ntu.edu.sg}

\author[T.A. Nguyen]{Tuan Anh Nguyen$^{2}$}
\address{$^2$  Faculty of Mathematics, Bielefeld University, Bielefeld, Germany}
\email{tnguyen@math.uni-bielefeld.de}

\author[S. Wu]{Sizhou Wu$^{3}$}
\address{$^3$  Shanghai University of Finance and Economics, School of Mathematics, Shanghai, China}
\email{wusizhou@sufe.edu.cn}
\thanks{
Financial support by the MOE AcRF Tier 2 Grant \emph{MOE-T2EP20222-0013} and the Nanyang Assistant Professorship Grant (NAP Grant) \emph{Machine Learning based Algorithms in
Finance and Insurance} is gratefully acknowledged.}
\allowdisplaybreaks
\begin{document}
	\vspace*{-1.7cm}
	\maketitle
\vspace*{-0.8cm}	
\begin{abstract}
In this paper we  
consider nonlinear partial integro-differential equations (PIDEs) with gradient-independent Lipschitz continuous nonlinearities 
and
prove that deep neural networks with ReLU activation function can approximate solutions of
such semilinear 
PIDEs without curse of dimensionality
in the sense that the required number of parameters in the deep neural networks increases at most polynomially in both the dimension $ d $ of the corresponding PIDE and the reciprocal of the prescribed accuracy~$ \epsilon $. 
\end{abstract}
\medskip


\section{Introduction}

Nonlinear partial integro-differential equations (PIDEs) 
have many important applications in finance,
physics, biology, economics, and engineering; we refer to
\cite{CT2004,CV2005,Del2013,S2005} and the references therein, as
well as \cite{BBJ+2022} for an excellent survey on applications of nonlocal 
partial differential equations (PDEs) in various fields.
The development of numerical methods and their complexity analysis 
for PIDEs are only at their infancy and are still emerging.
\cite{GS2021,GS2021a} proposed both machine learning based 
and Monte Carlo based methods for solving \textit{linear} PIDEs 
and showed that their algorithms do not suffer from the curse of dimensionality. 
In \cite{al2019applications} the deep Galerkin algorithm of 
\cite{SirignanoSpiliopoulos18} has been extented to PIDEs. 
In \cite{yuan2022pinn} physics informed neural networks for solving nonlinear 
PIDEs have been constructed.
Deep neural network (DNN) based algorithms for solving nonlinear PIDEs have been 
presented in \cite{al2019applications,Cas2021,FK2022a,FK2022,neufeld2024full}.
Recently, \cite{BBJ+2022} introduced for non-local nonlinear PDEs with Neumann 
boundary conditions both a deep splitting algorithm as well as 
a multilevel Picard (MLP) algorithm to numerically solve the non-local PDEs 
under consideration. Furthermore, \cite{GPP2022} introduced a DNN based 
algorithm for backward stochastic differential equations  with jumps.

The impressive computational performance of deep learning methods has raised questions concerning their
theoretical foundations.
However, compared to the field of applications where these methods are successfully applied extensively, there exists only one theoretical result proving that DNNs do overcome the curse of dimensionality when approximating \emph{linear} PIDEs, see \cite{GS2021a}.
The main contribution of our paper is a  proof  that DNNs with ReLU activation function do overcome the curse of dimensionality when approximating \emph{semilinear} PIDEs whose nonlinear part,  linear part,  jump term, and terminal condition are  Lipschitz continuous. Especially, \cref{k41} in \cref{main result} proves that the number of parameters in the approximating DNN increases at most polynomially in both the reciprocal of the described approximation accuracy $\epsilon\in (0,1)$ and the  dimension $d\in \N$ of the corresponding PIDE.
Before presenting our settings and main result, \cref{k41}, let us introduce some notations frequently used in this paper.
\subsection{Notations}\label{n01}
In this subsection we introduce some notations 
used throughout this paper. We denote by $\Z$ and $\N$ the set of all integers,
and the set of all positive integers, respectively. 
Let $\left \lVert\cdot \right\rVert, \supnorm{\cdot} \colon (\cup_{d\in \N} \R^d) \to [0,\infty)$, 
$\dim \colon (\cup_{d\in \N}\R^d) \to \N$
satisfy for all $d\in \N$, $x=(x_1,\ldots,x_d)\in \R^d$ that $\|x\|=\sqrt{\sum_{i=1}^d(x_i)^2}$, $\supnorm{x}=\max_{i\in [1,d]\cap \N}|x_i|$, and
$\dim (x)=d$ and
let 
$\lVert\cdot \rVert_\mathrm{F}\colon \cup_{d\in\N}\R^{d\times d}\to [0,\infty) $ satisfy for all
$ d\in\N $, $x=(x_{ij})_{i,j\in [1,d]\cap\Z}\in\R^{d\times d}$ that
$\lVert x\rVert^2_\mathrm{F}=\sum_{i,j=1}^{d}\lvert x_{ij}\rvert^2 $.
Moreover, for every Borel function $\mathfrak{f}:\R\to\cup_{d\in\N}\R^d$ and $t\in\R$, 
we use the notation 
$\mathfrak{f}(t-):= \lim_{s\to t,s<t}\mathfrak{f}(s)$ 
whenever the limit exists.

\subsection{Existence and uniqueness of solutions of PIDEs}
In this subsection we introduce a setting, \cref{setting PIDE} below, 
which ensures existence and uniqueness 
of solutions of nonlinear PIDEs (cf. \cref{p01b}) 
which we then aim to approximate 
by neural networks. 

\begin{setting}
\label{setting PIDE}
Let $T\in (0,\infty)$ and $c\in [2,\infty)$.
For every $d\in \N$ let 
$\beta^d\in C(\R^d,\R^d)  $,
$\sigma^d\in C(\R^d,\R^{d\times d})$, $\gamma^d\in C(\R^{2d},\R^d)$,
$g^d\in C(\R^d,\R)$, $f\in C(\R,\R)$ 
and let $\nu^d\colon \mathcal{B}(\R^d\setminus\{0\})\to [0,\infty)$ 
be a L\'evy measure. 
Assume that for all $d\in \N$ there exists $C_d\in (0,\infty)$ 
such that for all $x,y,z\in \R^d$ we have that
\begin{align}\label{pointwise gamma}
\left\lVert\gamma^d(x,z)\right\rVert^2
\leq C_d \left(1\wedge \lVert z\rVert^2\right),\quad
\left\lVert
\gamma^d(x,z)-\gamma^d(y,z)\right\rVert^2
\leq C_d \lVert x-y\rVert^2\left(1\wedge \lVert z\rVert^2\right).
\end{align}
Assume that for all
$d\in \N$ and $x,z\in\R^d$ the 
Jacobian matrix $(D_x\gamma^d)(x,z)$ exists. Assume that for all
$d\in \N$ there exists $\lambda_d\in (0,\infty)$ such that
for all $x,z\in\R^d$, $\delta\in [0,1]$
we have that
\begin{align}\label{cond Jacob}
\lambda_d\leq \left\lvert
\mathrm{det}(I_d +\delta (D_x\gamma^d)(x,z) )
\right\rvert,
\end{align}
where $I_d$ denotes the $d\times d$ identity matrix. 
Assume for all $d\in \N$, $x,y\in \R^d$, $w_1,w_2\in \R$ that
\begin{align} \begin{split} 
\left\lVert\beta^d(x)-
\beta^d(y)\right\rVert^2+
\left\lVert\sigma^d(x)-
\sigma^d(y)\right\rVert^2_{\mathrm{F}}
+
\int_{\R^d\setminus \{0\}}
\left\lVert
\gamma^d(x,z)-
\gamma^d(y,z))\right\rVert^2\nu^d(dz)
\leq c\lVert x-y\rVert^2,\end{split}
\label{Lip coe}
\end{align}
\begin{align}
\lvert
f(w_1)-f(w_2)\rvert^2\leq c\lvert w_1-w_2\rvert^2,\quad 
\left\lvert g^d(x)-g^d(y)\right\rvert^2\leq cd^cT^{-1}\lVert x-y \rVert^2,
\label{Lip f g}
\end{align} 
\begin{align}\label{growth 0}
\left
\lVert\beta^d(0)\right\rVert^2+
\left
\lVert\sigma^d(0)\right\rVert^2_{\mathrm{F}}+
\int_{\R^d\setminus \{0\}}
\left\lVert
\gamma^d(0,z)\right\rVert^2\nu^d(dz)
+T^3(\lvert f(0)\rvert+1)^2
+T\lvert g^d(0)\rvert^2
\leq cd^c.
\end{align}
\end{setting}

Under \cref{setting PIDE} 
we present the following Proposition, 
proved in \cite[Propositions 5.4 and 5.16]{NW2022}, 
for the existence and uniqueness
of viscosity solutions of the nonlinear PIDE
which we aim to approximate by neural networks. 
\begin{proposition}
\label{p01b}
 Assume \cref{setting PIDE}. Then for every $d\in \N$
there exists a unique viscosity solution\footnote{For the definition of a viscosity solution see, e.g., \cite[Definition 2.7]{NW2022}.} $u^d\colon [0,T]\times \R^d\to\R$ to the PIDE
\begin{align} \begin{cases} 
&(\tfrac{\partial }{\partial t}u^d)(t,x)
+\left\langle \beta^d(x),(\nabla_x u^d)(t,x) \right\rangle\\&+\frac{1}{2}\mathrm{trace}\!\left(
\sigma^d(t,x)(\sigma^d(t,x))^\top\mathrm{Hess}_xu^d(t,x)\right)
+f(u^d(t,x))\\
&+\displaystyle\int_{\R^d}
\left(
u^d(x+\gamma^d(x,z))-u^d(t,x)
-\left\langle
(\nabla_x u^d)(t,x), \gamma^d(x,z)
\right\rangle\right) \nu^d(dz)=0
\quad \forall\,t\in [0,T),x\in\R^d,\\
& u^d(T,x)=g^d(x)\quad \forall\,x\in\R^d
\end{cases}\label{PIDE}
\end{align} 
satisfying that $\sup_{s\in [0,T],y\in \R^d}\frac{\lvert u^d(s,y)\rvert}
{(d^{c}+\lVert y\rVert^2)^{1/2}}<\infty$.
\end{proposition}


\subsection{Setting for approximating functions}
In this section, we present our setting for functions which will approximate
the linear part, the terminal condition, and the non-local part of the
PIDE \eqref{PIDE}. 
Later, we will assume that these functions are neural networks in order to
prove our main result presented in \cref{k41}.

\begin{setting}
\label{setting approx NN}
Let $T\in (0,\infty)$ and $c\in [2,\infty)$ 
be the constants introduced in \cref{setting PIDE}.
For every $d\in \N$, $\varepsilon\in (0,1)$
let
$\beta_\varepsilon^d\in C(\R^d,\R^d)$,
$\sigma_\varepsilon^d\in C(\R^d,\R^{d\times d})$,
$\gamma_{\varepsilon}^d\in C(\R^{2d},\R^{d\times d})$,
$g^d_\varepsilon\in C(\R^d,\R)$.
Assume that for all $d\in \N$, $\varepsilon\in(0,1)$ there exists 
$C_{d,\varepsilon}\in (0,\infty)$ 
such that for all $x,y,z\in \R^d$ we have that
\begin{align}\label{pointwise gamma epsilon}
\left\lVert\gamma^d_\varepsilon(x,z)\right\rVert^2
\leq C_{d,\varepsilon} \left(1\wedge \lVert z\rVert^2\right)
,\quad
\left\lVert
\gamma^d_\varepsilon(x,z)-\gamma^d_\varepsilon(y,z)\right\rVert^2
\leq C_{d,\varepsilon} \lVert x-y\rVert^2\left(1\wedge \lVert z\rVert^2\right).
\end{align}
Moreover, assume 
 for all $d\in \N$, $x,y\in \R^d$, 
$w_1,w_2\in \R$,
$\varepsilon\in (0,1)$ that
\begin{align} \begin{split} 
\left\lVert\beta^d_\varepsilon(x)-
\beta^d_\varepsilon(y)\right\rVert^2+
\left\lVert\sigma^d_\varepsilon(x)-
\sigma^d_\varepsilon(y)\right\rVert^2_{\mathrm{F}}
+
\int_{\R^d\setminus \{0\}}
\left\lVert
\gamma^d_\varepsilon(x,z)-
\gamma^d_\varepsilon(y,z)\right\rVert^2\nu^d(dz)
\leq c\lVert x-y\rVert^2,\end{split}
\label{Lip NN coe}
\end{align}
\begin{align} 
\left\lvert g^d_\varepsilon(x)-g^d_\varepsilon(y)\right\rvert^2\leq cd^cT^{-1}\lVert x-y \rVert^2,
\label{Lip g epsilon}
\end{align} 
\begin{align}
\left
\lVert\beta^d_\varepsilon(0)\right\rVert^2+
\left
\lVert\sigma^d_\varepsilon(0)\right\rVert^2_{\mathrm{F}}+
\int_{\R^d\setminus \{0\}}
\left\lVert
\gamma^d_\varepsilon(0,z)\right\rVert^2\nu^d(dz)
+T\lvert g^d_\varepsilon(0)\rvert^2
\leq cd^c,
\label{growth NN coe}
\end{align}
\begin{align} 
\label{approx NN coe}
\begin{split} 
&
\left\lVert
\beta^d_\varepsilon(x)-
\beta^d(x)\right\rVert^2
+
\left\lVert
\sigma^d_\varepsilon(x)-
\sigma^d(x)\right\rVert^2_{\mathrm{F}}
+\int_{\R^d\setminus \{0\}}
\left\lVert
\gamma^d_\varepsilon(x,z)-
\gamma^d(x,z)\right\rVert^2\,\nu^d(dz)
 + \left\lvert g^d_\varepsilon(x)-g^d(x)\right\rvert^2\\&\leq \varepsilon c d^c (d^c+\lVert x\rVert^2).\end{split}
\end{align}
\end{setting}

Note that condition \eqref{approx NN coe} ensures that the input functions
$\beta^d,\sigma^d,\gamma^d,g^d $ can be approximated by the functions
$\beta^d_\varepsilon,\sigma^d_\varepsilon,\gamma^d_\varepsilon,g^d_\varepsilon $. 
Later in the main theorem, \cref{k41} below, 
we will assume that 
$\beta^d_\varepsilon,\sigma^d_\varepsilon,\gamma^d_\varepsilon,g^d_\varepsilon $ 
are DNN functions.

\subsection{A mathematical framework for DNNs}

In \cref{m07} below we introduce a mathematical framework for DNNs with ReLU activation function.

\begin{setting}[A mathematical framework for DNNs]
\label{m07}
For every $d\in \N$ 
let
 $\mathbf{A}_{d}\colon \R^d\to\R^d $ satisfy for all  $x=(x_1,\ldots,x_d)\in \R^d$ that 
\begin{align}
\mathbf{A}_{d}(x)= \left(\max\{x_1,0\},\max\{x_2,0\},\ldots,\max\{x_d,0\}\right).
\end{align}
Let $\mathbf{D}=\cup_{H\in \N} \N^{H+2}$.
Let
\begin{align}
\mathbf{N}= \bigcup_{H\in  \N}\bigcup_{(k_0,k_1,\ldots,k_{H+1})\in \N^{H+2}}
\left[ \prod_{n=1}^{H+1} \left(\R^{k_{n}\times k_{n-1}} \times\R^{k_{n}}\right)\right].
\end{align} 
Let $\mathcal{D}\colon \mathbf{N}\to\mathbf{D}$, 
$\mathcal{P}\colon \mathbf{N}\to \N$,
$
\mathcal{R}\colon \mathbf{N}\to (\cup_{k,l\in \N} C(\R^k,\R^l))$
satisfy that
for all $H\in \N$, $k_0,k_1,\ldots,k_H,k_{H+1}\in \N$,
$
\Phi = ((W_1,B_1),\ldots,(W_{H+1},B_{H+1}))\in \prod_{n=1}^{H+1} \left(\R^{k_n\times k_{n-1}} \times\R^{k_n}\right), 
$
$x_0 \in \R^{k_0},\ldots,x_{H}\in \R^{k_{H}}$ with the property that
$\forall\, n\in \N\cap [1,H]\colon x_n = \mathbf{A}_{k_n}(W_n x_{n-1}+B_n )
$ we have that
\begin{align}
\mathcal{P}(\Phi)=\sum_{n=1}^{H+1}k_n(k_{n-1}+1),
\quad 
\mathcal{D}(\Phi)= (k_0,k_1,\ldots,k_{H}, k_{H+1}),
\end{align}
$
\mathcal{R}(\Phi )\in C(\R^{k_0},\R ^ {k_{H+1}}),
$
and
\begin{align}
 (\mathcal{R}(\Phi)) (x_0) = W_{H+1}x_{H}+B_{H+1}.
\end{align}
\end{setting}
Let us comment on the mathematical objects  in  \cref{m07}. For all $ d\in \N $, $ \mathbf{A}_d\colon\R^d\to\R^d$  refers to the componentwise  rectified linear unit (ReLU) activation function. 
By $ \mathbf{N} $ we denote the set of all
parameters characterizing artificial feed-forward DNNs.
For every $H\in \N$, $k_0,k_1,\ldots,k_H,k_{H+1}\in \N$,
$
\Phi = ((W_1,B_1),\ldots,(W_{H+1},B_{H+1}))\in \prod_{n=1}^{H+1} \left(\R^{k_n\times k_{n-1}} \times\R^{k_n}\right)\subseteq \bfN
$ the natural number $H$ can be interpreted as the depth of the parameters characterizing artificial feed-forward DNN $\Phi$ and
$
(W_1,B_1),\ldots,(W_{H+1},B_{H+1})
$
can be interpreted as the parameters of $\Phi$.
By $ \calR $   we denote the operator that maps each parameters characterizing a DNN to its corresponding function. By $ \calP $ we denote the function that counts the number of parameters of the corresponding DNN. By $ \calD $ we denote the function that maps the parameters characterizing a DNN to the vector of its layer dimensions.

\subsection{Main result}
\label{main result}

\begin{theorem}
\label{k41}
Assume Settings \ref{setting PIDE}, \ref{setting approx NN}, and \ref{m07}.
For every $d\in \N$, $\varepsilon\in (0,1)$
let $F_{\varepsilon}^d\in C(\R,\R^{d\times d})$ and 
$G^d\in C(\R^d,\R^d)$ satisfy for all 
$y,z\in \R^d$ that
$\gamma_{\varepsilon}^d(y,z)=F_{\varepsilon}^d(y)G^d(z) $.
For every $d\in \N$, $\varepsilon\in (0,1)$, $v\in\R^d$
let
$\Phi_{\beta_\varepsilon^d},\Phi_{\sigma_\varepsilon^d,v}, 
\Phi_{F^d_\varepsilon,v}, \Phi_{g^d_\varepsilon}
\in \bfN
$
satisfy that
$\beta_\varepsilon^d=\calR(\Phi_{\beta_\varepsilon^d}) $, 
$\sigma_\varepsilon^d (\cdot)v=\calR(\Phi_{\sigma_\varepsilon^d,v})$,
$F^d_\varepsilon(\cdot)v=\calR(\Phi_{F^d_\varepsilon,v} )$,
$g^d_\varepsilon=\calR(\Phi_{g^d_\varepsilon})$.
Assume for all $d\in \N$, $\varepsilon\in (0,1)$, $v\in\R^d$ that
$\calD(\Phi_{\sigma_\varepsilon^d,v})=\calD(\Phi_{\sigma_\varepsilon^d,0})$,
$\calD(\Phi_{F_\varepsilon^d,v})=\calD(\Phi_{F_\varepsilon^d,0})$, 
\begin{align}
\supnorm{\calD(\Phi_{\beta^d_\varepsilon})}
+
\supnorm{\calD(\Phi_{\sigma^d_\varepsilon,0})}
+\supnorm{\calD(\Phi_{F^d_\varepsilon,0})}
+\supnorm{\calD(\Phi_{g^d_\varepsilon})}\leq \frac{bd^c\varepsilon^{-c}}{4},\label{k59}
\end{align}
and
\begin{align}
\dim(\calD(\Phi_{\beta^d_\varepsilon}))+ \dim(\calD(\Phi_{\sigma^d_\varepsilon,0}))+
\dim(\calD(\Phi_{F^d_\varepsilon,0}))+
\dim(\calD(\Phi_{g^d_\varepsilon}))\leq \frac{bd^c\varepsilon^{-c}}{4}
\label{k60}
\end{align}
with 
$
b = 64\Bigl(1+ \left\lvert c^\frac{1}{2}(4 c^\frac{1}{2} 
+2 c^\frac{1}{2} T^{-\frac{3}{2}} )\right\rvert^\frac{1}{2}\Bigr),
$
where $c\in[2,\infty)$ is the constant introduced in \cref{setting PIDE}.

Moreover, for each $d\in\N$ let $u^d:[0,T]\times\R^d\to\R$ 
be the unique viscosity solution to PIDE \eqref{PIDE} (c.f. \cref{p01b}),
and let $\Gamma^d\colon \mathcal{B}(\R^d)\to [0,1]$ 
be a probability measure satisfying
\begin{align}
\left(\int_{\R^d}\lVert x\rVert^4\,\Gamma^d(dx)\right)^\frac{1}{2}\leq cd^c.
\label{x61}
\end{align}

Then there exists $(C_{\delta})_{\delta\in (0,1)}\subseteq(0,\infty)$, $\eta\in (0,\infty)$, $(\Psi_{d,\epsilon})_{d\in \N, \epsilon\in(0,1)}$ such that for all 
$d\in \N$, $\epsilon\in(0,1)$ we have that 
$\calR(\Psi_{d,\epsilon})\in C(\R^d,\R)$,  
\begin{align}
	\calP(\Psi_{d,\epsilon})
	\leq C_\delta\eta d^{3c+12c^2+2c(6+\delta)}\epsilon^{-6c-6-\delta},
	\qquad 
	\mbox{ and } 
	\qquad \left(
\int_{\R^d}
\left\lvert(\calR(\Psi_{d,\epsilon}))(x)
-u^d(t,x)\right\rvert^2\Gamma^d(dx)\right)^\frac{1}{2}\leq \epsilon.
\label{k92}
\end{align}

\end{theorem}

Let us make some comments on the mathematical objects  in  \cref{k41}.

The assumptions above \eqref{k59} ensure that the functions
$g^d_\varepsilon,\beta^d_\varepsilon,\sigma^d_\varepsilon,
f_\varepsilon, F^d_\varepsilon, \gamma^d_\varepsilon$
which approximate the terminal condition, the linear part, 
the nonlinear part (see \cref{lemma f epsilon}), and the non-local part
of the PIDE \eqref{PIDE} are neural networks.
The bound $bd^c\varepsilon^{-c}/4$ in \eqref{k59} and \eqref{k60}, 
which is a polynomial of $d$ and $\varepsilon^{-1}$,
ensures that the functions
$\beta^d_\varepsilon,\sigma^d_\varepsilon,\gamma^d_\varepsilon,
F^d_\varepsilon, g^d_\varepsilon $ 
are DNNs whose corresponding numbers of parameters grow 
without the curse of dimensionality. 

Under these assumptions  \cref{k41} states that, roughly speaking, 
if DNNs can approximate the terminal condition, the linear part, 
the nonlinear part, and the jump part of the PIDE in \eqref{PIDE} 
without the curse of dimensionality, 
then they can also approximate its solution without the curse of dimensionality. 
{\color{black} More precisely, we show in \eqref{k92} that for every dimension $d\in \mathbb{N}$ and for every accuracy $\epsilon \in (0,1)$ the $L^2(\Gamma^d(dx))$-expression error of the unique viscosity solution of the nonlinear PIDE \eqref{PIDE} is $\epsilon$ and the number of parameters of the DNNs is upper bounded polynominally in $d$ and $\epsilon^{-1}$. Therefore, the approximation rates are free from the curse of dimensionality.}
We refer to \cite{AJK+2023,CR2023,CHW2022,GHJvW2023,HJKN2020a,JSW2021,
neufeld2024rectified} 
for similar results obtained for  PDEs \emph{without any non-local/ jump term}.

\subsection{Outline of the proof and organization of the paper}
Although the result presented in \cref{k41} is purely deterministic, we use probabilistic arguments to prove its statement. More precisely, we 
employ the theory of full history recursive MLP approximations, which are numerical approximation methods
for which it is known that they overcome the curse of dimensionality.
We refer to \cite{NW2022} for the convergence analysis of MLP algorithms for semilinear PIDEs and  to \cite{
beck2020overcomingElliptic,
beck2020overcoming,
hutzenthaler2019multilevel,
giles2019generalised,
hutzenthaler2021multilevel,
HJK2021,
HJKN2020,
HJK+2020,
HJvW2020,
hutzenthalerkruse2020multilevel,
HN2022a, neufeld2023multilevel} 
for corresponding results proving that MLP algorithms overcome the curse of dimensionality for PDEs without any non-local/ jump term.

The main strategy of the proof, roughly speaking, 
is to demonstrate that these MLP approximations can be  represented by DNNs, 
if the coefficients determining the linear part, the jump term, 
the terminal condition, and the nonlinear part are corresponding DNNs 
(cf.\  \cref{k04}). Such ideas have been successfully applied to prove that
DNNs overcome the curse of dimensionality in the numerical approximations of
 \emph{semilinear} heat equations (see  \cite{HJKN2020a})
as well as \emph{semilinear} Kolmogorov
PDEs (see \cite{CHW2022}). We also refer to 
\cite{GHJvW2023,JSW2021} for results proving that
DNNs
overcome the curse of dimensionality when approximating \emph{linear} PDEs.

In order to introduce (the outline of) the proof of \cref{k41}, 
we need the following setting which we will often use throughout this paper.

\begin{setting}
\label{theta setting}
Assume \cref{setting PIDE} and \cref{setting approx NN}.
Let $(\Omega,\mathcal{F},\P,(\F_t)_{t\in [0,T]})$ be a filtered probability space 
satisfying the usual conditions.
Let $\Theta=\cup_{n\in\N}\Z^n$. 
Let 
$\mathfrak{t}^\theta\colon \Omega\to[0,1]$, $\theta\in \Theta$,
be identically independently distributed\ random variables 
which satisfy for all $t\in (0,1)$ that
$\P( \mathfrak{t}^0\leq t  )=t$.
For every
$\theta\in \Theta$, $t\in [0,T]$
let $\mathfrak{T}^\theta_t\colon \Omega\to \R$ 
satisfy for all $\theta\in \Theta$ that
$\mathfrak{T}^\theta_t=t+(T-t)\mathfrak{t}^\theta$.
For every $d\in \N$ let 
$W^{d,\theta}\colon \Omega\times[0,T]\to\R^d$, $\theta\in\Theta$,
be identically independently distributed\ standard 
$(\F_t)_{t\in [0,T]}$-Brownian motions.
For every $d\in \N$ let $N^{d,\theta}$,
$\theta\in\Theta$,
be independent $(\F_t)_{t\in[0,T]}$-Poisson random measures 
on $ [0,\infty)\times (\R^d\setminus\{0\})$ with compensator $\nu^d(dz)\,dt$
and let
\begin{align}
\tilde{N}^{d,\theta}(dz,dt)={N}^{d,\theta}(dz,dt)-\nu^d(dz)\,dt.
\end{align} 
Assume for all $d\in \N$ that $\mathcal{F}_0$,
$(\mathfrak{t}^\theta)_{\theta\in\Theta}$,
$(N^{d,\theta})_{\theta\in\Theta}$, and
$(W^{d,\theta})_{\theta\in\Theta}$
are independent.

Moreover, for every $d\in\N$, $\theta\in\Theta$, $(t,x)\in[0,T]\times\R^d$,
$\varepsilon\in(0,1)$ let
$X^{d,\theta,t,x}, X^{d,\theta,\varepsilon,t,x}:[t,T]\times\R^d\to\R^d$ 
be the c\`adl\`ag adapted processes defined by\footnote{
Note that \eqref{pointwise gamma}, \eqref{Lip coe}, \eqref{growth 0}, 
\eqref{pointwise gamma epsilon}, \eqref{Lip NN coe}, and \eqref{growth NN coe}
ensure the existence and uniqueness of solutions
of SDEs \eqref{SDE X} and \eqref{SDE X epsilon}
(see, e.g., \cite[Theorem 3.1]{Kunita} or \cite[Theorem 117]{Situ}). 
}
\begin{align} \begin{split} 
X^{d,\theta,t,x}_{s}&=x+
\int_{t}^{s}
\beta^d(
X^{d,\theta,t,x}_{u-}
)du
+
\int_{t}^{s}
\sigma^d(
X^{d,\theta,t,x}_{u-}
)dW_u^{d,\theta}
\\
&\quad 
+
\int_{t}^{s}\int_{\R^d\setminus \{0\}}
\gamma^d(
X^{d,\theta,t,x}_{u-},z)\tilde{N}^{d,\theta}(dz,du), \quad s\in[t,T],
\end{split}
\label{SDE X}
\end{align}
and
\begin{align} 
\begin{split} 
X^{d,\theta,\varepsilon,t,x}_{s}&
=x+
\int_{t}^{s}
\beta^d_\varepsilon(
X^{d,\theta,\varepsilon,t,x}_{u-}
)du
+
\int_{t}^{s}
\sigma^d_\varepsilon(
X^{d,\theta,\varepsilon,t,x}_{u-}
)dW_u^{d,\theta}
\\&\quad 
+\int_{t}^{s}
\int_{\R^d\setminus \{0\}}
\gamma^d_\varepsilon(
X^{d,\theta,\varepsilon,t,x}_{u-},z)\tilde{N}^{d,\theta}(dz,du), \quad s\in[t,T].
\end{split}
\label{SDE X epsilon}
\end{align}

In addition, for every
$K\in \N$ 
let $\rdown{\cdot}_K\colon \R\to\R$ satisfy for all $t\in \R$ that 
$\rdown{t}_K=\max( \{0,\frac{T}{K},\frac{2T}{K},\ldots,T\} 
\cap (    (-\infty,t)\cup\{0\} )  )$
and denote for every $d\in\N$, $\theta\in\Theta$, $\varepsilon\in(0,1)$,
$K\in\N$, $(t,x)\in[0,T]\times\R^d$
by $X^{d,\theta,K,\varepsilon,t,x}:[t,T]\times\R^d\to\R^d$, 
the Euler-Maruyama discretization of $X^{d,\theta,\varepsilon,t,x}$, i.e.
\begin{align} 
\begin{split} 
X^{d,\theta,K,\varepsilon,t,x}_{s}
&
=x+\int_{t}^{s}
\beta^d_\varepsilon(
X^{d,\theta,K,\varepsilon,t,x}_{\max\{t,\rdown{u-}_K\}}\,
)du
+
\int_{t}^{s}
\sigma^d_\varepsilon(
X^{d,\theta,K,\varepsilon,t,x}_{\max\{t,\rdown{u-}_K\}}\,
)dW_u^{d,\theta}\\&\quad+
\int_{t}^{s}
\int_{\R^d\setminus \{0\}}
\gamma^d_\varepsilon(
X^{d,\theta,K,\varepsilon,t,x}_{\max\{t,\rdown{u-}_K\}},z)
\,\tilde{N}^{d,\theta}(dz,du),
\quad s\in[t,T].
\end{split}
\label{SDE X epsilon K}
\end{align}
\end{setting} 

Now, let us outline the proof of \cref{k41}.
First, the unique viscosity solution of the PIDE \eqref{PIDE} 
can be represented by the following stochastic fixed point equation (SFPE) 
(cf. \cite[Proposition 5.16]{NW2022} 
which we recall as \cref{Prop SFPE} in \cref{section SFPE}):
\begin{align*}
u^d(t,x)
&
=\E\!\left [g^d(X^{d,0,t,x}_{T} )  \right] + 
\int_{t}^{T}\E\!\left [
f(u^d(X^{d,0,t,x}_{s} ) ) \right] ds.
\end{align*}

Next, let $f_\varepsilon$, $\varepsilon\in(0,1)$, 
be neural network functions which approximate $f$
(c.f. \cref{lemma f epsilon}). 
Then we approximate the input functions $\beta^d$, $\sigma^d$, $\gamma^d$,
$g^d$, and $f$ by $\beta^d_\varepsilon$, $\sigma^d_\varepsilon$, 
$\gamma^d_\varepsilon$, $g^d_\varepsilon$, and $f_\varepsilon$, respectively,
where the functions with index $\varepsilon$ are DNNs whose number of
parameters grows at most polynomially in $d$ and $\varepsilon^{-1}$
due to \eqref{k59} and \eqref{k60}.
The key observation is that under our assumptions the following
MLP approximation
\begin{align} \begin{split} 
&
{U}^{d,\theta,K,\varepsilon}_{n,m}(t,x)=\frac{\1_{\N}(n)}{m^n}\sum_{i=1}^{m^n}g^d_\varepsilon(X^{d,{(\theta,0,-i)},K,\varepsilon,t,x}_{T})\\&\quad
+\sum_{\ell=0}^{n-1}
\frac{T-t}{m^{n-\ell}}
\sum_{i=1}^{m^n}
(f_\varepsilon\circ {U}^{d,{(\theta,\ell,i),K,\varepsilon}}_{\ell,m}
-
f_\varepsilon\circ {U}^{d,{(\theta,-\ell,i),K,\varepsilon}}_{\ell-1,m}
)(\mathfrak{T}_t^{(\theta,\ell,i)}, 
X^{d,{(\theta,0,i)},K,\varepsilon,t,x}_{\mathfrak{T}_t^{(\theta,\ell,i)}}
)\end{split}
\end{align}
is also a DNN whose numbers of parameters 
grows at most polynomially in $d$ and $\varepsilon^{-1}$
(cf. \cref{k04}).
To control the error between the solution of the PIDE $u^d$ 
and the MLP approximation ${U}^{d,\theta,K,\varepsilon}_{n,m}$,
we introduce for every $d,K\in\N$, $\varepsilon\in(0,1)$, $(t,x)\in[0,T]\times\R^d$
the stochastic fixed-point equations (SFPEs):
\begin{align*}
u^{d,\varepsilon}(t,x)
=\E\!\left [g^d_\varepsilon(X^{d,0,\varepsilon,t,x}_{T} )  \right] + 
\int_{t}^{T}\E\!\left [
f_\varepsilon(u^{d,\varepsilon}(X^{d,0,\varepsilon,t,x}_{s} ) ) \right] ds,
\end{align*}
and
\begin{align*}
u^{d,K,\varepsilon}(t,x)=\E\!\left [g^d_\varepsilon(X^{d,0,K,\varepsilon,t,x}_{T} )  \right] + 
\int_{t}^{T}\E\!\left [
f_\varepsilon(u^{d,K,\varepsilon}(X^{d,0,K,\varepsilon,t,x}_{s} ) ) \right] ds.
\end{align*}
We then decompose the error
$
{U}^{d,\theta,K,\varepsilon}_{n,m}(t,x)
-u^d(t,x)
$
as follows:
\begin{align} \begin{split} 
&
{U}^{d,\theta,K,\varepsilon}_{n,m}(t,x)
-u^d(t,x)\\
&
=
\underbrace{{U}^{d,\theta,K,\varepsilon}_{n,m}(t,x)
-u^{d,K,\varepsilon}(t,x)}_{=:E_1(t,x)}+
\underbrace{u^{d,K,\varepsilon}(t,x)-u^{d,\varepsilon}(t,x)}_{=:E_2(t,x)}+
\underbrace{u^{d,\varepsilon}(t,x)
-u^d(t,x)
}_{=:E_3(t,x)}\end{split}\end{align}
where $(\E [\lvert E_1(t,x)\rvert^2])^{1/2}$ is estimated in 
\cref{a31}, $\lvert E_2(t,x)\rvert$ is estimated in \cref{k32}, and 
$\lvert E_3(t,x)\rvert$ is estimated in \cref{x01}.	


The remaining part of the paper is organized as follows. 
In \cref{section SFPE} we present a proposition (c.f. \cref{Prop SFPE}) on
some SFPEs used throughout this paper.  
In \cref{s02} we establish a stability result on SFPEs that demonstrates the error $E_3$. In \cref{s03} we recall basic facts on Euler-Maruyama and MLP approximations and establish the error 
$E_1$
of the MLP approximation 
as well as the discretization error $E_2$. \cref{s04} introduces a mathematical framework for DNNs and demonstrates the connection between DNNs and MLP approximations, see \cref{k04}. Finally, \cref{s05} combines the results of the previous sections to prove the main result, \cref{k41}.

\section{Stochastic fixed point equations}
\label{section SFPE}
In this section, we will present a proposition (see \cref{Prop SFPE} below)
which summarizes the known results on existence and uniqueness 
of solutions to some stochastic fixed point equations 
(c.f. \cite[Proposition 2.2]{HJKN2020} 
and \cite[Propositions 4.1 and 5.16]{NW2022}).
To this end we first show the following lemma 
(c.f. \cite[Corollary 3.13]{HJKN2020a}).

\begin{lemma}
\label{lemma f epsilon}
Assume Settings \ref{setting PIDE} and \ref{m07} and
let  
$
b:= 64\Bigl(1+ \left\lvert c^\frac{1}{2}(4 c^\frac{1}{2} 
+2 c^\frac{1}{2} T^{-\frac{3}{2}} )\right\rvert^\frac{1}{2}\Bigr)
$
with $c\in[2,\infty)$ and $T\in(0,\infty)$
being the constants introduced in \cref{setting PIDE}.
Then there exist
$f_\varepsilon\in C(\R,\R)$ and
$\Phi_{f_\varepsilon}\in\bfN$ satisfying for all
$\varepsilon\in(0,1)$, $w_1,w_2\in \R$ that
\begin{align}
\calR(\Phi_{f_\varepsilon})=f_\varepsilon,
\quad
\dim(\calD(\Phi_{f_\varepsilon}))=3\leq\frac{ bd^c\varepsilon^{-c}}{4},
\label{k65}
\end{align} 
\begin{align} \begin{split} 
\supnorm{\calD(\Phi_{f_\varepsilon})}
\leq \frac{b\varepsilon^{-2}}{4},\end{split}
\label{k66}
\end{align} 
\begin{align}
\quad\lvert f_\varepsilon(w_1)- f_\varepsilon(w_2)\rvert^2\leq c\lvert w_1-w_2\rvert^2,
\quad
\lvert f(w_1)-f_\varepsilon(w_1)\rvert^2\leq \varepsilon(1+\lvert w_1\rvert^4),
\label{k67}
\end{align}
and
\begin{align}
T^3\lvert f_\varepsilon(0)\rvert
\leq cd^c.
\label{growth f epsilon}
\end{align}
\end{lemma}

\begin{proof}[Proof of \cref{lemma f epsilon}]
Notice that \cite[Corollary 3.13]{HJKN2020a} 
(applied for all $\varepsilon\in(0,1)$ with 
$L\gets c^\frac{1}{2}$, $q\gets2$, $\epsilon\gets(\varepsilon/2)^\frac{1}{2}$ 
in the notation of \cite[Corollary 3.13]{HJKN2020a}) 
and \eqref{growth 0} ensures \eqref{k65}, the first inequality in \eqref{k67},
\begin{align*} \begin{split} 
\supnorm{\calD(\Phi_{f_\varepsilon})}
&\leq 16\Bigl(1+ \left\lvert c^\frac{1}{2}(4 c^\frac{1}{2} +2\lvert f(0)\rvert )\right\rvert^\frac{1}{2}\Bigr)\varepsilon^{-2}\\&\leq 
16\Bigl(1+ \left\lvert c^\frac{1}{2}(4 c^\frac{1}{2} +2 \xeqref{growth 0} c^\frac{1}{2} T^{-\frac{3}{2}} )\right\rvert^\frac{1}{2}\Bigr)\varepsilon^{-2}\leq \frac{b\varepsilon^{-2}}{4},\end{split}
\end{align*} 
and
\begin{align}
\lvert f(w_1)-f_\varepsilon(w_1)\rvert^2\leq \frac{\varepsilon^2}{2}(1+\lvert w_1\rvert^2)^2\leq \varepsilon(1+\lvert w_1\rvert^4).
\label{k67b}
\end{align}
Furthermore,
\eqref{k67b} and the triangle inequality show 
for all $d\in\N$, $\varepsilon\in (0,1)$ that
$\lvert f_\varepsilon(0)\rvert\leq \sqrt{\varepsilon}+\lvert f(0)\rvert
\leq 1+\lvert f(0)\rvert$.
This together with \eqref{growth 0} and \eqref{growth NN coe} imply 
\eqref{growth f epsilon}.
Therefore, we have completed the proof of this lemma.

\end{proof}

\begin{proposition}[Stochastic fixed point equations]
\label{Prop SFPE}
Assume Settings \ref{setting PIDE}, \ref{setting approx NN}, and \ref{theta setting}.
For every $\varepsilon\in(0,1)$, let $f_\varepsilon\in C(\R,\R)$ be a function
which satisfies \eqref{k67} and \eqref{growth f epsilon}.
Then for all $\varepsilon\in (0,1)$, $d,K\in \N$ 
there exist measurable functions
$u^{d,K,\varepsilon},u^{d,\varepsilon},u^{d}\colon [0,T]\times \R^d\to\R$ 
such that for all
$t\in[0,T]$, $x\in\R^d$ we have that
$
\E\!\left[\left\lvert g^d_\varepsilon(
X^{d,0,K,\varepsilon,t,x}_{T} )\right\rvert\right]
+\int_{t}^{T}
\E\!\left[\left\lvert f_\varepsilon(u^{d,K,\varepsilon}(r,
X^{d,0,K,\varepsilon,t,x}_{r} ))\right\rvert\right]dr
+
\E\!\left[\left\lvert g^d_\varepsilon(
X^{d,0,\varepsilon,t,x}_{T} )\right\rvert\right]
+\int_{t}^{T}
\E\!\left[\left\lvert f_\varepsilon(u^{d,\varepsilon}(r,
X^{d,0,\varepsilon,t,x}_{r} ))\right\rvert\right]dr
+
\E\!\left[\left\lvert g^d(
X^{d,0,t,x}_{T} )\right\rvert\right]
+\int_{t}^{T}
\E\!\left[\left\lvert f(u^{d}(r,
X^{d,0,t,x}_{r} ))\right\rvert\right]dr
<\infty
$,
\begin{align}
\sup_{s\in[0,T]}\sup_{y\in \R^d}
\frac{\lvert u^{d,K,\varepsilon}(s,y)\rvert+\lvert u^{d,\varepsilon}(s,y)\rvert+\lvert u^d(s,y)\rvert }{(d^c+\lVert y\rVert^2)^{1/2}}<\infty,
\label{growth u}
\end{align}
\begin{align}
u^{d,K,\varepsilon}(t,x)=
\E\!\left[g^d_\varepsilon(
X^{d,0,K,\varepsilon,t,x}_{T} )\right]
+\int_{t}^{T}
\E\!\left[f_\varepsilon(u^{d,K,\varepsilon}(r,
X^{d,0,K,\varepsilon,t,x}_{r} ))\right]dr,
\label{FK u K epsilon}
\end{align}
\begin{align}
u^{d,\varepsilon}(t,x)=
\E\!\left[g^d_\varepsilon(
X^{d,0,\varepsilon,t,x}_{T} )\right]
+\int_{t}^{T}
\E\!\left[f_\varepsilon(u^{d,\varepsilon}(r,
X^{d,0,\varepsilon,t,x}_{r} ))\right]dr,
\label{FK u epsilon}
\end{align} 
and
\begin{align}
u^{d}(t,x)=
\E\!\left[g^d(
X^{d,0,t,x}_{T} )\right]
+\int_{t}^{T}
\E\!\left[f(u^{d}(r,
X^{d,0,t,x}_{r} ))\right]dr.
\label{FK u}
\end{align}
\end{proposition}

\begin{proof}[Proof of \cref{Prop SFPE}]
First we notice that \eqref{Lip f g}, \eqref{growth 0}, 
\eqref{Lip NN coe}, \eqref{growth NN coe} imply for all
$d\in \N$, $\varepsilon\in(0,1)$, $x\in \R^d$ that 
\begin{align}
\left\lvert
g_\varepsilon^d(x)\right\rvert
\leq \left\lvert
g_\varepsilon^d(0)\right\rvert
+(cd^cT^{-1})^{\frac{1}{2}}\lVert x\rVert
\leq (cd^c T^{-1})^{\frac{1}{2}}
+(cd^c T^{-1})^{\frac{1}{2}}\lVert x\rVert
\leq 
2
(cd^c T^{-1})^{\frac{1}{2}} (d^c+\lVert x\rVert^2)^\frac{1}{2},
\label{growth g epsilon x}
\end{align}
and
\begin{align}
\left\lvert
g^d(x)\right\rvert
\leq 
2
(cd^c T^{-1})^{\frac{1}{2}} (d^c+\lVert x\rVert^2)^\frac{1}{2}.
\label{growth g x}
\end{align}
By \eqref{Lip coe}, \eqref{growth 0}, \eqref{Lip NN coe}, \eqref{growth NN coe}, 
and \cite[Lemmas 3.1 and 3.6]{NW2022}, for every $d\in\N$ 
there is a constant $c_d$
only depending on $T$ and $c$ such that 
for all $K\in \N$, $\varepsilon\in(0,1)$, 
$t\in[0,T]$, $s\in[t,T]$, $x\in \R^d$ 
\begin{align}
\max \left\{
\E \!\left[d^c+\left\lVert X^{d,0,K,\varepsilon,t,x}_{s}\right\rVert^2
\right],
\E \!\left[d^c+\left\lVert X^{d,0,\varepsilon,t,x}_{s}\right\rVert^2
\right]
,
\E \!\left[d^c+\left\lVert X^{d,0,t,x}_{s}\right\rVert^2
\right]\right\}
\leq  c_d(d^c+\lVert x\rVert^2).
\label{moment ests X}
\end{align}
Then by \eqref{Lip f g}, \eqref{growth 0}, \eqref{k67}, 
\eqref{growth f epsilon}, 
\eqref{growth g epsilon x}, \eqref{growth g x}, and \eqref{moment ests X},
the application of \cite[Proposition 4.1]{NW2022} proves this proposition.
\end{proof}

\section{Approximation of the coefficients}\label{s02}
In \cref{x01} below we approximate the solution to
the SFPE \eqref{FK u} through the solution to the SFPE \eqref{FK u epsilon}, 
which is defined with respect to 
the approximating functions introduced in \cref{setting approx NN}.

\begin{proposition}[A stability result]\label{x01}
Assume Settings \ref{setting PIDE}, \ref{setting approx NN}
and \ref{theta setting}.
For every $\varepsilon\in(0,1)$ let $f_\varepsilon\in C(\R,\R)$ be a function
which satisfies \eqref{k67} and \eqref{growth f epsilon}.
For every $ d\in\N $,
$\varepsilon\in (0,1)$
let $u^{d,\varepsilon},u^d\colon [0,T]\times \R^d\to\R$ 
be measurable functions introduced in \eqref{FK u epsilon} and \eqref{FK u},
respectively.
Moreover, for convenience we use the notations
$X^{d,t,x}:=X^{d,0,t,x}$,
$X^{d,\varepsilon,t,x}:=X^{d,0,\varepsilon,t,x}$,
$W^d:=W^{d,0}$, $\tilde{N}^d:=\tilde{N}^{d,0}$
for every $d\in\N $,
$t\in[0,T]$, $x\in\R^d$, $\varepsilon\in (0,1)$.
Then the following items hold.
\begin{enumerate}[(i)]
\item \label{i01}
For all
$d\in \N$, $\varepsilon\in(0,1)$,
$t\in[0,T]$, 
$x\in \R^d$ we have that 
\begin{align}
\max\!\left\{
\E \!\left[d^c+\left\lVert X^{d,\varepsilon,t,x}_{s}\right\rVert^2
\right]
,
\E \!\left[d^c+\left\lVert X^{d,t,x}_{s}\right\rVert^2
\right]
\right\}
\leq  (d^c+\lVert x\rVert^2)e^{7c(s-t)}.
\end{align}
\item \label{i03}For all
$d\in \N$, $\varepsilon\in(0,1)$,
$t\in[0,T]$, 
$x,y\in \R^d$ we have that 
\begin{align}
\left\lvert
u^{d,\varepsilon}(t,x)-
u^{d,\varepsilon}(t,y)
\right\rvert\leq 
2(cd^cT^{-1})^\frac{1}{2}
\lVert x-y\rVert e^{5cT + 2cT^2}.
\end{align}
\item\label{i04} For all
$d\in \N$, $\varepsilon\in(0,1)$,
$t\in[0,T]$, 
$x\in \R^d$ we have that 
\begin{align}
\left\lvert
u^{d,\varepsilon}(t,x)-
u^{d}(t,x)\right\rvert\leq 23cd^c\varepsilon^{\frac{1}{2}}
(d^c+\lVert x\rVert^2)
e^{24cT+5cT^2}.
\end{align}
\end{enumerate}
\end{proposition}

\begin{proof}[Proof of \cref{x01}]Throughout this proof 
let $\langle\cdot,\cdot \rangle \colon \cup_{d\in \R}\R^{d}\times\R^d\to \R $ satisfy
for every $d\in\N$, $x=(x_1,\ldots,x_d),y=
(y_1,\ldots,y_d)
\in \R^d$ that
$\langle x,y \rangle=\sum_{i=1}^{d}x_iy_i$.
First,
the fact that
$\forall\,d\in \N, x,y\in \R^d\colon \lVert x+y\rVert^2\leq 2\lVert x\rVert^2+2\lVert y\rVert^2$,
\eqref{Lip NN coe}, and \eqref{growth NN coe} show for all  $d\in \N$, 
$x\in \R^d$, $\varepsilon\in (0,1)$  that
\begin{align}
\lVert
\beta^d_\varepsilon(x)
\rVert^2\leq 
2\lVert\beta^d_\varepsilon(0)\rVert^2
+2\lVert\beta^d_\varepsilon(x)-\beta^d_\varepsilon(0)
\rVert^2\leq  2\xeqref{growth NN coe}cd^c+2\xeqref{Lip NN coe}c\lVert x\rVert^2 
= 2c(d^c+\lVert x\rVert^2).\label{m01}
\end{align}
Similarly, we have for all  $d\in \N$, 
$x\in \R^d$, $\varepsilon\in (0,1)$  that
\begin{align}\label{m03}
\left\lVert
\sigma^d_\varepsilon(x)
\right\rVert^2_{\mathrm{F}}\leq  2c(d^c+\lVert x\rVert^2).
\end{align}
Next,
the fact that
$\forall\,d\in \N, x,y\in \R^d\colon \lVert x+y\rVert^2
\leq 2\lVert x\rVert^2+2\lVert y\rVert^2$, \eqref{Lip NN coe}, 
and \eqref{growth NN coe} imply for all  $d\in \N$, 
$x\in \R^d$, $\varepsilon\in (0,1)$  that
\begin{align} \begin{split} 
\int_{\R^d\setminus \{0\}}
\left\lVert
\gamma^d_\varepsilon(x,z)\right\rVert^2\nu^d(dz)
&\leq 
\int_{\R^d\setminus \{0\}}
2
\left\lVert
\gamma^d_\varepsilon(0,z)\right\rVert^2
+2\left\lVert
\gamma^d_\varepsilon(x,z)-\gamma^d_\varepsilon(0,z)\right\rVert
^2\nu^d(dz)\\&\leq 2\xeqref{growth NN coe}cd^c+
2
\xeqref{Lip NN coe}c\lVert x^2\rVert=2c(d^c+\lVert x\rVert^2).
\end{split}\label{m04}
\end{align}
Furthermore, It\^o's formula (see, e.g., \cite[Theorem 3.1]{GW2021}) 
and \eqref{SDE X epsilon} show 
for all  $d\in \N$, $t\in [0,T]$, $s\in [t,T]$,
$x\in \R^d$, $\varepsilon\in (0,1)$  that
$\P$-a.s.\ we have that 
\begin{align} \begin{split} 
\left\lVert X^{d,\varepsilon,t,x}_{s}\right\rVert^2
&=\lVert x\rVert^2
+\int_{t}^{s} \left(2\left\langle
X^{d,\varepsilon,t,x}_{r},
\beta^d_\varepsilon(X^{d,\varepsilon,t,x}_{r})
\right\rangle
+\left\lVert\sigma_\varepsilon^d(X^{d,\varepsilon,t,x}_{r})\right\rVert^2_{\mathrm{F}}\right)dr\\&\quad +2
\int_{t}^{s}\sum_{i,j=1}^{d}
\left(X^{d,\varepsilon,t,x}_{r-}\right)_i
\left(\sigma_\varepsilon^d(X^{d,\varepsilon,t,x}_{r-})\right)_{ij}d(W^{d}_j)_r\\
&\quad +2\int_{t}^{s}\int_{\R^d\setminus\{0\}}
\left\langle
X^{d,\varepsilon,t,x}_{r-},
\gamma^d_\varepsilon(X^{d,\varepsilon,t,x}_{r-},z)
\right\rangle
\tilde{N}^d(dz,dr)\\&\quad +
\int_{t}^{s}\int_{\R^d\setminus\{0\}}
\left\lVert
\gamma^d_\varepsilon(X^{d,\varepsilon,t,x}_{r-},z)
\right\rVert^2N^d(dz,dr).
\end{split}\label{m50}
\end{align}
In addition, 
for every
  $d,n\in \N$, 
$x\in \R^d$, $\varepsilon\in (0,1)$ let $
\tau_n^{d,\varepsilon,x}\colon \Omega\to \R$ satisfy that
\begin{align} \begin{split} &
\tau_n^{d,\varepsilon,x}
=\inf \Biggl\{s\in [t,T] 
\colon \int_{t}^{s} \left(2\left\langle
X^{d,\varepsilon,t,x}_{r},
\beta^d_\varepsilon(X^{d,\varepsilon,t,x}_{r})
\right\rangle
+\left\lVert\sigma_\varepsilon^d(X^{d,\varepsilon,t,x}_{r})\right\rVert^2_{\mathrm{F}}\right)dr
\\&\qquad\qquad\qquad
+\int_{t}^{s}\sum_{i,j=1}^{d}
\left\lvert\left(X^{d,\varepsilon,t,x}_{r}\right)_i
\left(\sigma_\varepsilon^d(X^{d,\varepsilon,t,x}_{r})\right)_{ij}
\right\rvert^2
dr\\&\qquad\qquad\qquad +\int_{t}^{s}\int_{\R^d\setminus \{0\}}
\sum_{i=1}^{d}\left\lvert
\left(X^{d,\varepsilon,t,x}_{r}\right)_i
\left(\gamma^d_\varepsilon(X^{d,\varepsilon,t,x}_{r},z)\right)_i
\right\rvert^2\nu^d(dz)\,dr\\&\qquad\qquad\qquad
+\int_{t}^{s}\int_{\R^d\setminus\{0\}}
\left\lVert
\gamma^d_\varepsilon(X^{d,\varepsilon,t,x}_{r},z)
\right\rVert^2_{\mathrm{F}}\nu^d(dz)\,dr
\geq n
\Biggr\}\wedge T\end{split}\label{x00}
\end{align}
(with the convention that $\inf\emptyset=\infty$).
Then \eqref{m50}, the fact that
$\forall\,d\in \N, x,y\in \R^d\colon 2\langle x,y\rangle\leq \lVert x\rVert^2+\lVert y\rVert^2$,
\eqref{m01}, \eqref{m03}, \eqref{m04}, and the fact that $c\geq1$ imply
for all  $d,n\in \N$, $t\in [0,T]$, $s\in [t,T]$,
$x\in \R^d$, $\varepsilon\in (0,1)$  that
\begin{align} \begin{split} 
&\E \!\left[d^c+\left\lVert X^{d,\varepsilon,t,x}_{s\wedge \tau_n^{d,\varepsilon,x}}\right\rVert^2
\right]\\&
=\xeqref{m50}d^c+\lVert x\rVert^2+
\E \!\left[ \int_{t}^{s\wedge \tau_n^{d,\varepsilon,x}} \left(2\left\langle
X^{d,\varepsilon,t,x}_{r},
\beta^d_\varepsilon(X^{d,\varepsilon,t,x}_{r})
\right\rangle
+\left\lVert\sigma_\varepsilon^d(X^{d,\varepsilon,t,x}_{r})\right\rVert^2_{\mathrm{F}}\right)dr\right]\\
&\quad+
\E\!\left[
\int_{t}^{s\wedge \tau_n^{d,\varepsilon,x}}\int_{\R^d\setminus\{0\}}
\left\lVert
\gamma^d_\varepsilon(X^{d,\varepsilon,t,x}_{r},z)
\right\rVert^2\nu^d(dz)\,dr
\right]\\
&\leq d^c+ 
\lVert x\rVert^2+
\E \!\left[ \int_{t}^{s\wedge \tau_n^{d,\varepsilon,x}}
\left\lVert
X^{d,\varepsilon,t,x}_{r}\right\rVert^2dr\right]
+
\E \!\left[ \int_{t}^{s\wedge \tau_n^{d,\varepsilon,x}}
\left\lVert
\beta^d_\varepsilon(X^{d,\varepsilon,t,x}_{r})\right\rVert^2dr\right]\\
&\quad 
+
\E \!\left[ \int_{t}^{s\wedge \tau_n^{d,\varepsilon,x}} \left\lVert\sigma_\varepsilon^d(X^{d,\varepsilon,t,x}_{r})\right\rVert^2_{\mathrm{F}}dr\right]+
\E\!\left[
\int_{t}^{s\wedge \tau_n^{d,\varepsilon,x}}\int_{\R^d\setminus\{0\}}
\left\lVert
\gamma^d_\varepsilon(X^{d,\varepsilon,t,x}_{r},z)
\right\rVert^2\nu^d(dz)\,dr
\right]\\
\end{split}
\end{align} and
\begin{align}
 \begin{split} 
&
\E \!\left[d^c+\left\lVert X^{d,\varepsilon,t,x}_{s\wedge \tau_n^{d,\varepsilon,x}}\right\rVert^2
\right]\\
&\leq d^c+ 
\lVert x\rVert^2+
\E \!\left[ \int_{t}^{s\wedge \tau_n^{d,\varepsilon,x}}\left(
d^c+
\left\lVert
X^{d,\varepsilon,t,x}_{r}\right\rVert^2\right)dr\right]
+
\E \!\left[ \int_{t}^{s\wedge \tau_n^{d,\varepsilon,x}}\xeqref{m01}2c\left(d^c+
\left\lVert
X^{d,\varepsilon,t,x}_{r}\right\rVert^2\right) dr\right]\\
&\quad 
+
\E \!\left[ \int_{t}^{s\wedge \tau_n^{d,\varepsilon,x}}  \xeqref{m03}
2c\left(d^c+
\left\lVert X^{d,\varepsilon,t,x}_{r}\right\rVert^2_{\mathrm{F}}\right)dr\right]+
\E\!\left[
\int_{t}^{s\wedge \tau_n^{d,\varepsilon,x}}
\xeqref{m04}2c
\left(d^c+
\left\lVert X^{d,\varepsilon,t,x}_{r}
\right\rVert^2
\right)
dr
\right]\\
&\leq 
d^c+ 
\lVert x\rVert^2+7c\int_{t}^{s}
\E \!\left[ 
d^c+
\left\lVert
X^{d,\varepsilon,t,x}_{r\wedge \tau_n^{d,\varepsilon,x}}\right\rVert^2dr\right].
\end{split}
\end{align}
Hence, 
 the fact that
$\forall\,d\in \N, x\in \R^d, \varepsilon\in (0,1)\colon \P (\lim_{n\to\infty} \tau_n^{d,\varepsilon,x}=T)=1$
(cf. \eqref{x00}),
Fatou's lemma,
 and Gr\"onwall's inequality demonstrate
for all  $d\in \N$, $t\in [0,T]$, $s\in [t,T]$,
$x\in \R^d$, $\varepsilon\in (0,1)$  that
\begin{align}\label{k22b} \begin{split} 
&\E \!\left[d^c+\left\lVert X^{d,\varepsilon,t,x}_{s}\right\rVert^2
\right]
=
\E \!\left[
\lim_{n\to\infty}\left(
d^c+\left\lVert X^{d,\varepsilon,t,x}_{s\wedge \tau_n^{d,\varepsilon,x}}\right\rVert^2
\right)
\right]\\
&
\leq \liminf_{n\to\infty}
\E \!\left[d^c+\left\lVert X^{d,\varepsilon,t,x}_{s\wedge \tau_n^{d,\varepsilon,x}}\right\rVert^2
\right]\leq (d^c+\lVert x\rVert^2)e^{7c(s-t)}.\end{split}
\end{align}
Using a similar argument as that for \eqref{k22b},  
by \eqref{Lip coe} and \eqref{growth 0} 
we obtain for all  $d\in \N$, $t\in [0,T]$, $s\in [t,T]$,
$x\in \R^d$ that
\begin{align}\label{k22}
\E \!\left[d^c+\left\lVert X^{d,t,x}_{s}\right\rVert^2
\right]\leq (d^c+\lVert x\rVert^2)e^{7c(s-t)}.
\end{align}
This and \eqref{k22b} show \eqref{i01}.

Next, \eqref{approx NN coe},  \eqref{k22},
and the fact that for all
$d\in\N$, $t\in [0,T]$, $s\in [t,T]$,
$x\in \R^d$, $\varepsilon\in(0,1)\colon$ 
$\int_{s}^{t} \1_{ X^{d,t,x}_{r-}\neq X^{d,t,x}_{r} } \,dr=0$
$\P$-a.s.\
show 
for all  $d\in \N$, $t\in [0,T]$, $s\in [t,T]$,
$x\in \R^d$, $\varepsilon\in (0,1)$  that
\begin{align} \begin{split} 
&
\E\! \left[
\int_{t}^{s}\left\lVert
\beta^d_\varepsilon(X^{d,t,x}_{r-})
-\beta^d(X^{d,t,x}_{r-})
\right\rVert^2 dr
\right]=
\E\! \left[
\int_{t}^{s}\left\lVert
\beta^d_\varepsilon(X^{d,t,x}_{r})
-\beta^d(X^{d,t,x}_{r})
\right\rVert^2 dr
\right]
\\&\leq \int_{t}^{s}\xeqref{approx NN coe}\varepsilon c d^c \left(d^c+ 
\E\! \left[
\left\lVert
X^{d,t,x}_{r}
\right\rVert^2\right]
\right)dr\\
&\leq 
\int_{t}^{s}\varepsilon c d^c \xeqref{k22}(d^c+\lVert x\rVert^2)e^{7c(r-t)}\,dr\leq
\varepsilon c d^c (d^c+\lVert x\rVert^2) (s-t) e^{7c T}
\end{split}\label{k24}
\end{align}
and similarly
\begin{align} \begin{split} &
\int_{t}^{s}
\E\!\left[
\left\lVert\sigma^d_\varepsilon(
X^{d,t,x}_{r-}
)-\sigma^d(
X^{d,t,x}_{r-}
)\right\rVert^2_{\mathrm{F}}\right]dr\leq 
\varepsilon c d^c (d^c+\lVert x\rVert^2) (s-t) e^{7c T}
.\end{split}\label{k25}
\end{align}
In addition, \eqref{approx NN coe} and \eqref{k22} 
ensure
for all  $d\in \N$, $t\in [0,T]$, $s\in [t,T]$,
$x\in \R^d$, $\varepsilon\in (0,1)$  that
\begin{align} \begin{split} 
&\int_{t}^{s}
\int_{\R^d\setminus \{0\}}
\E \left[
\left\lVert
\gamma^d_\varepsilon(
X^{d,t,x}_{r-},z)-
\gamma^d(
X^{d,t,x}_{r-},z)
\right\rVert^2\right]
\nu^d(dz)
\,dr\\&\leq \int_{t}^{s}\xeqref{approx NN coe}\varepsilon c d^c \left(d^c+ 
\E \left[
\left\lVert
X^{d,t,x}_{r}
\right\rVert^2\right]
\right)dr\\
&\leq \int_{t}^{s}\varepsilon c d^c \xeqref{k22}(d^c+\lVert x\rVert^2)e^{7c(r-t)}\,dr\leq
\varepsilon c d^c (d^c+\lVert x\rVert^2)(s-t)e^{7cT} .
 \end{split}\label{k26}
\end{align}
 Furthermore, H\"older's inequality,
the fact that
$\forall\,d\in \N, x,y\in \R^d\colon \lVert x+y\rVert^2\leq 2\lVert x\rVert^2+2\lVert y\rVert^2$, \eqref{Lip NN coe}, and \eqref{k24}
imply 
for all  $d\in \N$, $t\in [0,T]$, $s\in [t,T]$,
$x\in \R^d$, $\varepsilon\in (0,1)$  that
\begin{align} \begin{split} &
\E\! \left[\left\lVert
\int_{t}^{s}
(\beta^d_\varepsilon(
X^{d,\varepsilon,t,x}_{r-}
)-\beta^d(
X^{d,t,x}_{r-}
))\,dr
\right\rVert^2\right]\\
&\leq 
\E\! \left[\left(
\int_{t}^{s}\left\lVert
\beta^d_\varepsilon(
X^{d,\varepsilon,t,x}_{r-}
)-\beta^d(
X^{d,t,x}_{r-}
)\right\rVert dr
\right)^2\right]\\
&\leq 
\E\! \left[\left(\int_{t}^{s}dr\right)\left(
\int_{t}^{s}\left\lVert
\beta^d_\varepsilon(
X^{d,\varepsilon,t,x}_{r-}
)-\beta^d(
X^{d,t,x}_{r-}
)\right\rVert^2 dr
\right)\right]\\
&
\leq T
\E\! \left[
\int_{t}^{s}\left\lVert
\beta^d_\varepsilon(
X^{d,\varepsilon,t,x}_{r-}
)-\beta^d(
X^{d,t,x}_{r-}
)\right\rVert^2 dr
\right]
\\
&
\leq 2T
\E\! \left[
\int_{t}^{s}\left\lVert
\beta^d_\varepsilon(
X^{d,\varepsilon,t,x}_{r-}
)-\beta^d_\varepsilon(
X^{d,t,x}_{r-}
)\right\rVert^2 dr
\right]
+
 2T
\E\! \left[
\int_{t}^{s}\left\lVert
\beta^d_\varepsilon(X^{d,t,x}_{r-})
-\beta^d(X^{d,t,x}_{r-})
\right\rVert^2 dr
\right]\\
&\leq 2T\xeqref{Lip NN coe}c
\int_{t}^{s}\E\! \left[
\left\lVert
X^{d,\varepsilon,t,x}_{r}
-
X^{d,t,x}_{r}
\right\rVert^2 \right]dr
+2T\cdot \xeqref{k24}
\varepsilon c d^c (d^c+\lVert x\rVert^2) (s-t) e^{7 c T}\\
&= 2cT
\int_{t}^{s}\E\! \left[
\left\lVert
X^{d,\varepsilon,t,x}_{r}
-
X^{d,t,x}_{r}
\right\rVert^2 \right]dr
+2
\varepsilon c d^c (d^c+\lVert x\rVert^2)T (s-t) e^{7 c T}.
\end{split}\label{m06}
\end{align}
Moreover,
by It\^o's isometry, the fact that $\forall\,d\in \N, x,y\in \R^{d \times d}\colon \lVert x+y\rVert^2_\mathrm{F}\leq 2\lVert x\rVert^2_\mathrm{F}+2\lVert y\rVert^2_\mathrm{F}$, \eqref{Lip NN coe}, and \eqref{k25}
we obtain 
for all  $d\in \N$, $t\in [0,T]$, $s\in [t,T]$,
$x\in \R^d$, $\varepsilon\in (0,1)$  that
\begin{align} \begin{split} 
&
\E\!\left[
\left\lVert
\int_{t}^{s}(
\sigma^d_\varepsilon(
X^{d,\varepsilon,t,x}_{r-}
)-\sigma^d(
X^{d,t,x}_{r-}
))\,dW_r^{d}\right\rVert^2\right]\\
&=
\E\!\left[
\int_{t}^{s}
\left\lVert\sigma^d_\varepsilon(
X^{d,\varepsilon,t,x}_{r-}
)-\sigma^d(
X^{d,t,x}_{r-}
)\right\rVert^2_{\mathrm{F}} dr\right]\\
&\leq 2
\int_{t}^{s}
\E\!\left[
\left\lVert\sigma^d_\varepsilon(
X^{d,\varepsilon,t,x}_{r-}
)-\sigma^d_\varepsilon(
X^{d,t,x}_{r-}
)\right\rVert^2_{\mathrm{F}}\right]dr
+
2
\int_{t}^{s}
\E\!\left[
\left\lVert\sigma^d_\varepsilon(
X^{d,t,x}_{r-}
)-\sigma^d(
X^{d,t,x}_{r-}
)\right\rVert^2_{\mathrm{F}}\right]dr\\
&\leq 2\xeqref{Lip NN coe}c
\int_{t}^{s}
\E\!\left[
\left\lVert
X^{d,\varepsilon,t,x}_{r}
-
X^{d,t,x}_{r}
\right\rVert^2_{\mathrm{F}}\right]dr+2\xeqref{k25}\varepsilon c d^c (d^c+\lVert x\rVert^2) (s-t) e^{7 c T}.
\end{split}\label{m08}
\end{align}
Next,
It\^o's isometry (see, e.g., \cite[Proposition 8.8]{CT2004}), the fact that $\forall\,d\in \N, x,y\in \R^{d}\colon \lVert x+y\rVert^2\leq 2\lVert x\rVert^2+2\lVert y\rVert^2$, \eqref{Lip NN coe}, and \eqref{k26}
show 
for all  $d\in \N$, $t\in [0,T]$, $s\in [t,T]$,
$x\in \R^d$, $\varepsilon\in (0,1)$  that 
\begin{align} \begin{split} 
&
\E \left[
\left\lVert
\int_{t}^{s}
\int_{\R^d\setminus \{0\}}(
\gamma^d_\varepsilon(
X^{d,\varepsilon,t,x}_{r-},z)-\gamma^d(
X^{d,t,x}_{r-},z))\,\tilde{N}^{d}(dz,dr)\right\rVert^2\right]\\
&= \int_{t}^{s}
\int_{\R^d\setminus \{0\}}
\E \left[
\left\lVert
\gamma^d_\varepsilon(
X^{d,\varepsilon,t,x}_{r-},z)-\gamma^d(
X^{d,t,x}_{r-},z))\right\rVert^2\right]\nu^d(dz)\,dr\\
&\leq 2
\int_{t}^{s}
\int_{\R^d\setminus \{0\}}
\E \left[
\left\lVert
\gamma^d_\varepsilon(
X^{d,\varepsilon,t,x}_{r-},z)-
\gamma^d_\varepsilon(
X^{d,t,x}_{r-},z))\right\rVert^2\right]\nu^d(dz)
\,dr
\\&\quad+2\int_{t}^{s}
\int_{\R^d\setminus \{0\}}
\E \left[
\left\lVert
\gamma^d_\varepsilon(
X^{d,t,x}_{r-},z)-
\gamma^d(
X^{d,t,x}_{r-},z)
\right\rVert^2\right]
\nu^d(dz)
\,dr\\
&
\leq
2\xeqref{Lip NN coe}c
\int_{t}^{s}
\E \left[
\left\lVert
X^{d,\varepsilon,t,x}_{r}-
X^{d,t,x}_{r}\right\rVert^2\right]
dr+2 \xeqref{k26}\varepsilon c d^c (d^c+\lVert x\rVert^2)(s-t)e^{7 cT}.
\end{split}\label{m09}
\end{align}
In addition, \eqref{SDE X} and \eqref{SDE X epsilon}
ensure 
for all  $d\in \N$, $t\in [0,T]$, $s\in [t,T]$,
$x\in \R^d$, $\varepsilon\in (0,1)$  that
$\P$-a.s.\ 
\begin{align} \begin{split} 
&
X^{d,\varepsilon,t,x}_{s}-X^{d,t,x}_{s}\\&=\xeqref{SDE X}\xeqref{SDE X epsilon}
\int_{t}^{s}
(\beta^d_\varepsilon(
X^{d,\varepsilon,t,x}_{r-}
)-\beta^d(
X^{d,t,x}_{r-}
))\,dr
+
\int_{t}^{s}(
\sigma^d_\varepsilon(
X^{d,\varepsilon,t,x}_{r-}
)-\sigma^d(
X^{d,t,x}_{r-}
))\,dW_r^{d}\\
&\quad +
\int_{t}^{s}
\int_{\R^d\setminus \{0\}}(
\gamma^d_\varepsilon(
X^{d,\varepsilon,t,x}_{r-},z)-\gamma^d(
X^{d,t,x}_{r-},z))\,\tilde{N}^{d}(dz,dr).
\end{split}\label{k21}
\end{align}
This, the fact that
$\forall\,d\in \N,x,y,z\in \R^d\colon 
\lVert x+y+z \rVert^2\leq 3\lVert x\rVert^2+3\lVert y\rVert^2+3\lVert z\rVert^2
$,
\eqref{m06}, \eqref{m08}, and \eqref{m09}
imply
for all  $d\in \N$, $t\in [0,T]$, $s\in [t,T]$,
$x\in \R^d$, $\varepsilon\in (0,1)$   that 
\begin{align} \begin{split} 
&
\E\!\left[
\left\lVert
X^{d,\varepsilon,t,x}_{s}-X^{d,t,x}_{s}\right\rVert^2
\right]\\&\leq \xeqref{k21}
3\E\!\left[\left\lVert
\int_{t}^{s}
(\beta^d_\varepsilon(
X^{d,\varepsilon,t,x}_{r-}
)-\beta^d(
X^{d,t,x}_{r-}
))dr
\right\rVert^2\right]
+
3
\E\!\left[
\left\lVert
\int_{t}^{s}(
\sigma^d_\varepsilon(
X^{d,\varepsilon,t,x}_{r-}
)-\sigma^d(
X^{d,t,x}_{r-}
))\,dW_r^{d}\right\rVert^2\right]\\
&\quad  +
3
\E\!\left[
\left\lVert
\int_{t}^{s}
\int_{\R^d\setminus \{0\}}(
\gamma^d_\varepsilon(
X^{d,\varepsilon,t,x}_{r-},z)-\gamma^d(
X^{d,t,x}_{r-},z))\,\tilde{N}^{d}(dz,dr)\right\rVert^2\right]\\
&\leq 3\xeqref{m06}\left[
 2cT
\int_{t}^{s}\E\! \left[
\left\lVert
X^{d,\varepsilon,t,x}_{r}
-
X^{d,t,x}_{r}
\right\rVert^2 \right]dr
+2
\varepsilon c d^c (d^c+\lVert x\rVert^2)T (s-t) e^{7 c T}\right]\\
&\quad +3\xeqref{m08}
\left[2c
\int_{t}^{s}
\E\!\left[
\left\lVert
X^{d,\varepsilon,t,x}_{r}
-
X^{d,t,x}_{r}
\right\rVert^2_{\mathrm{F}}\right]dr+2\varepsilon c d^c (d^c+\lVert x\rVert^2) (s-t) e^{7 c T}\right]\\
&\quad +3\xeqref{m09}\left[2c
\int_{t}^{s}
\E\! \left[
\left\lVert
X^{d,\varepsilon,t,x}_{r}-
X^{d,t,x}_{r}\right\rVert^2\right]
dr+ 2\varepsilon c d^c (d^c+\lVert x\rVert^2) (s-t) e^{7 c T}\right]\\
&=(12c+6cT)
\int_{t}^{s}
\E\! \left[
\left\lVert
X^{d,\varepsilon,t,x}_{r}-
X^{d,t,x}_{r}\right\rVert^2\right]
dr
+3\left(2T+4\right)
\varepsilon c d^c (d^c+\lVert x\rVert^2) (s-t)e^{7 cT}.
\end{split}
\end{align}
Thus, Gr\"onwall's inequality, \eqref{k22b}, \eqref{k22}, and the fact that
$
3\left(2T+4\right)\leq 12(T+1)\leq 12 e^{cT}
$
show 
for all  $d\in \N$, $t\in [0,T]$, $s\in [t,T]$,
$x\in \R^d$, $\varepsilon\in (0,1)$   that
\begin{align} \begin{split} 
\E\!\left[
\left\lVert
X^{d,\varepsilon,t,x}_{s}-X^{d,t,x}_{s}\right\rVert^2
\right]&\leq 3\left(2T+4\right)
\varepsilon c d^c (d^c+\lVert x\rVert^2)
(s-t) e^{7cT}e^{(12c+6cT)T}\\&
\leq 12e^{cT}\varepsilon c d^c (d^c+\lVert x\rVert^2) e^{7cT}e^{(12c+6cT)T}(s-t)\\&=
12
\varepsilon  c d^c (d^c+\lVert x\rVert^2) e^{20 cT+6cT^2}(s-t).
\end{split}\label{k01}\end{align}
In addition,
by \eqref{SDE X epsilon}, 
the fact that
$\forall\,d\in \N,x_1,x_2,x_3,x_4\in\R^d\colon 
\lVert
\sum_{i=1}^{4}x_i\rVert^2\leq 4\sum_{i=1}^{4}\lVert
x_i\rVert^2
$,
Jensen's inequality, It\^o's isometry,
and \eqref{Lip NN coe} we have 
for all  $d\in \N$, $t\in [0,T]$, $s\in [t,T]$,
$x\in \R^d$, $\varepsilon\in (0,1)$   that
\begin{align} \begin{split} 
\E \!\left[\left\lVert X^{d,\varepsilon,t,x}_{s}
-
X^{d,\varepsilon,t,y}_{s}\right\rVert^2\right]
&\leq\xeqref{SDE X epsilon} 4\lVert x-y\rVert^2+
4\E\!\left[\left\lVert
\int_{t}^{s}(
\beta^d_\varepsilon(
X^{d,\varepsilon,t,x}_{r-}
)-\beta^d_\varepsilon(X^{d,\varepsilon,t,y}_{r-}))
dr\right\rVert^2\right]\\
&\quad 
+
4\E\!\left[\left\lVert\int_{t}^{s}
(\sigma^d_\varepsilon(
X^{d,\varepsilon,t,x}_{r-}
)-
\sigma^d_\varepsilon(X^{d,\varepsilon,t,y}_{r-}))
dW_r^{d}\right\rVert^2\right]\\
&\quad +
4\E\!\left[\left\lVert
\int_{t}^{s}
\int_{\R^d\setminus \{0\}}(
\gamma^d_\varepsilon(X^{d,\varepsilon,t,x}_{r-},z)
-\gamma^d_\varepsilon(X^{d,\varepsilon,t,y}_{r-},z))
\tilde{N}^{d}(dz,dr)\right\rVert^2\right]\\
\end{split}
\end{align}
and
\begin{align}
 \begin{split} 
&\E \!\left[\left\lVert X^{d,\varepsilon,t,x}_{s}
-
X^{d,\varepsilon,t,y}_{s}\right\rVert^2\right]\\
&\leq 
4\lVert x-y\rVert^2+
4T
\E\!\left[
\int_{t}^{s}\left\lVert
\beta^d_\varepsilon(
X^{d,\varepsilon,t,x}_{r-}
)-\beta^d_\varepsilon(X^{d,\varepsilon,t,y}_{r-})\right\rVert^2
dr\right]\\&\quad 
+
4\E\!\left[\int_{t}^{s}
\left\lVert\sigma^d_\varepsilon(
X^{d,\varepsilon,t,x}_{r-}
)-
\sigma^d_\varepsilon(X^{d,\varepsilon,t,y}_{r-})
\right\rVert^2_{\mathrm{F}}
dr\right]\\
&\quad +
4\E\!\left[
\int_{t}^{s}
\int_{\R^d\setminus \{0\}}
\left\lVert
\gamma^d_\varepsilon(X^{d,\varepsilon,t,x}_{r-},z)
-\gamma^d_\varepsilon(X^{d,\varepsilon,t,y}_{r-},z)
\right\rVert^2\nu^d(dz)\,dr\right]\\
&\leq \xeqref{Lip NN coe}
4\lVert x-y\rVert^2+
4T
\E\!\left[
\int_{t}^{s}c\left\lVert
X^{d,\varepsilon,t,x}_{r-}
-X^{d,\varepsilon,t,y}_{r-}\right\rVert^2
dr\right]
+
4\E\!\left[\int_{t}^{s}c
\left\lVert 
X^{d,\varepsilon,t,x}_{r-}
-
X^{d,\varepsilon,t,y}_{r-}
\right\rVert^2
dr\right]\\
&\quad +
4\E\!\left[\int_{t}^{s}c
\left\lVert
X^{d,\varepsilon,t,x}_{r-}
-
X^{d,\varepsilon,t,y}_{r-}
\right\rVert^2
dr\right]\\
&=
4\lVert x-y\rVert^2+(4 T c + 8c)
\int_{t}^{s}
\E\!\left[
\left\lVert
X^{d,\varepsilon,t,x}_{r}
-
X^{d,\varepsilon,t,y}_{r}
\right\rVert^2
\right]dr.
\end{split}
\end{align}
This, Gr\"onwall's lemma, and \eqref{k22b} 
show 
for all  $d\in \N$, $t\in [0,T]$, $s\in [t,T]$,
$x\in \R^d$, $\varepsilon\in (0,1)$   that
\begin{align} \begin{split} 
\E \!\left[\left\lVert X^{d,\varepsilon,t,x}_{s}
-
X^{d,\varepsilon,t,y}_{s}\right\rVert^2\right]\leq 4\lVert x-y\rVert^2 e^{(4Tc+8c)(s-t)}
\leq 4\lVert x-y\rVert^2 e^{(4Tc+8c)T}
=
 4\lVert x-y\rVert^2e^{8cT + 4cT^2}.\label{k40}
\end{split}
\end{align}
Next, \eqref{Lip g epsilon}, Jensen's inequality, and \eqref{k01}
imply for all
$d\in \N$, $\varepsilon\in(0,1)$,
$t\in[0,T]$, $x\in \R^d$ that 
\begin{align} \begin{split} 
\E \! \left[
\left\lvert
g^d_\varepsilon(X^{d,\varepsilon,t,x}_{T} ) -
g^d_\varepsilon(X^{d,t,x}_{T} ) \right\rvert
\right]
&\leq \xeqref{Lip g epsilon}
\E \! \left[
(cd^c)^\frac{1}{2}T^{-\frac{1}{2}}
\left\lVert
X^{d,\varepsilon,t,x}_{T}  -
X^{d,t,x}_{T}  \right\rVert
\right]\\
&\leq 
(cd^c)^\frac{1}{2}T^{-\frac{1}{2}}
\left(
\E \!\left[
\left\lVert
X^{d,\varepsilon,t,x}_{T}  -
X^{d,t,x}_{T}  \right\rVert^2\right]\right)^\frac{1}{2}
\\
&\leq 
(cd^c)^\frac{1}{2}T^{-\frac{1}{2}}
\left(\xeqref{k01}12 \varepsilon  d^c (d^c+\lVert x\rVert^2) e^{20 cT+6cT^2}T\right)^{\frac{1}{2}}
\\&\leq 4 (\varepsilon c d^{2c})^\frac{1}{2}
(d^c+\lVert x\rVert^2)^\frac{1}{2}e^{10cT+3cT^2}.\end{split}\label{k06}
\end{align}
 Moreover,
\eqref{approx NN coe}, Jensen's inequality, and  \eqref{k22}
ensure for all
$d\in \N$, $\varepsilon\in(0,1)$,
$t\in[0,T]$, $x\in \R^d$ that 
\begin{align} \begin{split} 
\E\!\left[
\left\lvert
g^d_\varepsilon(X^{d,t,x}_{T} ) 
-g^d(X^{d,t,x}_{T} )\right\rvert
\right]\leq \xeqref{approx NN coe}
\E \!\left[
\left\lvert
 \varepsilon c d^c \left(d^c+\left\lVert X^{d,t,x}_{T}\right\rVert^2\right)\right\rvert^\frac{1}{2}
\right]
&\leq (\varepsilon cd^c)^\frac{1}{2}
\left(
\E \!\left[
 d^c+\left\lVert X^{d,t,x}_{T}\right\rVert^2
\right]\right)^\frac{1}{2}\\
&\leq 
(\varepsilon cd^c)^\frac{1}{2}\left(\xeqref{k22}(d^c+\lVert x\rVert^2)e^{7cT}\right)^{\frac{1}{2}}\\&
=
(\varepsilon cd^c)^\frac{1}{2}(d^c+\lVert x\rVert^2)^{\frac{1}{2}}e^{3.5cT}.
\end{split}\label{k30}
\end{align}
This, the triangle inequality, and \eqref{k06}
show for all
$d\in \N$, $\varepsilon\in(0,1)$,
$t\in[0,T]$, $x\in \R^d$ that 
\begin{align} \begin{split} 
\E \!\left[\left\lvert g^d_\varepsilon(X^{d,\varepsilon,t,x}_{T} )  -g^d(X^{d,t,x}_{T} )\right\rvert\right]& \leq 
\E \!\left[\left\lvert
g^d_\varepsilon(X^{d,\varepsilon,t,x}_{T} ) -
g^d_\varepsilon(X^{d,t,x}_{T} ) 
\right\rvert\right]+
\E \!\left[\left\lvert
g^d_\varepsilon(X^{d,t,x}_{T} ) 
-g^d(X^{d,t,x}_{T} )
\right\rvert\right]\\
&\leq \xeqref{k06}4(\varepsilon c d^{2c})^\frac{1}{2}
(d^c+\lVert x\rVert^2)^\frac{1}{2}e^{10 cT+3cT^2}
+\xeqref{k30}
(\varepsilon cd^c)^\frac{1}{2}(d^c+\lVert x\rVert^2)^{\frac{1}{2}}e^{3.5cT}\\
&
\leq 5(\varepsilon c d^{2c})^\frac{1}{2}
(d^c+\lVert x\rVert^2)^\frac{1}{2}e^{10 cT+3cT^2}.
\end{split}\label{k88}
\end{align}
Thus, Jensen's inequality and \eqref{k22}
imply for all
$d\in \N$, $\varepsilon\in(0,1)$,
$s\in [0,T]$,
$t\in[s,T]$, $x\in \R^d$ that 
\begin{align} \begin{split} 
\E\!\left[
\E \!\left[\left\lvert g^d_\varepsilon(X^{d,\varepsilon,t,\tilde{x}}_{T} )  -g^d(X^{d,t,\tilde{x}}_{T} )\right\rvert\right]\Bigr|_{\tilde{x}=X^{d,s,x}_{t}}
\right]
&\leq 
5(\varepsilon c d^{2c})^\frac{1}{2}
\E\left[\left(d^c+\left\lVert X^{d,s,x}_{t}\right\rVert^2\right)^\frac{1}{2}
\right]
e^{10cT+3cT^2}\\
&\leq 
 5(\varepsilon c d^{2c})^\frac{1}{2}
\left(\E\!\left[d^c+\left\lVert X^{d,s,x}_{t}\right\rVert^2
\right]\right)^\frac{1}{2}
e^{10cT+3cT^2}
\\
&\leq 
 5(\varepsilon c d^{2c})^\frac{1}{2}
\left(\xeqref{k22}(d^c+\lVert x\rVert^2)e^{7 cT}\right)^\frac{1}{2}
e^{10cT+3cT^2}\\
&\leq 5(\varepsilon c d^{2c})^\frac{1}{2}
(d^c+\lVert x\rVert^2)^{\frac{1}{2}}
e^{14cT+3cT^2}.
\end{split}\label{k89}\end{align}
Next, 
by the triangle inequality, 
\eqref{Lip g epsilon}, \eqref{growth NN coe}, 
\eqref{k67}, and \eqref{growth f epsilon}
we obtain for all
$d\in \N$, $\varepsilon\in(0,1)$,
 $x\in \R^d$, $w\in\R$ that 
\begin{align} \begin{split} 
\left\lvert
g_\varepsilon^d(x)\right\rvert\leq 
\left\lvert
g_\varepsilon^d(0)\right\rvert
+
\left\lvert g_\varepsilon^d(x)-
g_\varepsilon^d(0)\right\rvert
&\leq \xeqref{growth NN coe}(cd^c T^{-1})^{\frac{1}{2}}
+\xeqref{Lip g epsilon}(cd^c T^{-1})^{\frac{1}{2}}\lVert x\rVert
\\&\leq 
2
(cd^c T^{-1})^{\frac{1}{2}} (d^c+\lVert x\rVert^2)^\frac{1}{2}\label{k90}
\end{split}
\end{align}
and 
\begin{align}
\lvert
f_\varepsilon(w)\rvert\leq
\lvert
f_\varepsilon(0)\rvert
+
\lvert
f_\varepsilon(w)-f_\varepsilon(0)\rvert
\leq \xeqref{growth f epsilon}(cd^c T^{-3})^{\frac{1}{2}}+
\xeqref{k67}c^\frac{1}{2}\lvert w\rvert.\label{m91}
\end{align}
This and \eqref{approx NN coe} show
 for all
$d\in \N$, 
$x\in \R^d$ that 
\begin{align}
\lvert g^d(x)\rvert=\lim_{\varepsilon\to 0}
\left\lvert g^d_\varepsilon(x)\right\rvert\leq 
2
(cd^c T^{-1})^{\frac{1}{2}} (d^c+\lVert x\rVert^2)^\frac{1}{2}.
\end{align}
Hence, \eqref{k22}
implies for all
$d\in \N$, $t\in [0,T]$,
 $x\in \R^d$ that 
\begin{align} \begin{split} 
\left\lVert g^d(X^{d,t,x}_{T})  \right\rVert_{L^2(\P)}&\leq \xeqref{k90}
2
(cd^c T^{-1})^{\frac{1}{2}} \left\lVert\left(d^c+\left\lVert X^{d,t,x}_{T}\right\rVert^2\right)^\frac{1}{2}\right\rVert_{L^2(\P)}\\
&\leq 
2
(cd^c T^{-1})^{\frac{1}{2}}
\left(
 \E \!\left[d^c+\left\lVert X^{d,t,x}_{T}\right\rVert^2\right]\right)^\frac{1}{2}\\
&\leq 2
(cd^c T^{-1})^{\frac{1}{2}}
\left(\xeqref{k22} (d^c+\lVert x\rVert^2)e^{7 cT} \right)^{\frac{1}{2}}\\
&
=2
(cd^c T^{-1})^{\frac{1}{2}}
\left( d^c+\lVert x\rVert^2 \right)^{\frac{1}{2}}e^{3.5 cT}.
\end{split}\label{m92}
\end{align}
This,
\eqref{FK u epsilon}, the triangle inequality, the disintegration theorem, the flow property, and \eqref{m91} 
demonstrate for all
$d\in \N$, 
$s\in[0,T]$, 
$t\in [s,T]$,
$x\in \R^d$ that 
\begin{align} \begin{split} 
&
\left\lVert
u^{d}(t,X_{t}^{d,s,x})\right\rVert_{L^2(\P)}\\
&
=\left\lVert
u^{d}(t,\tilde{x})
\big|_{\tilde{x}=X_{t}^{d,s,x}}
\right\rVert_{L^2(\P)}\\
&\leq\xeqref{FK u epsilon}
\left\lVert
\E\!\left[\left\lvert g^d(X^{d,t,\tilde{x}}_{T})  \right\rvert\right]
\bigr|_{\tilde{x}=X_{t}^{d,s,x}}
\right\rVert_{L^2(\P)} 
+ 
\int_{t}^{T}\left\lVert\E\!\left [\left\lvert
f(u^{d}(X^{d,t,\tilde{x}}_{r} ) ) 
 \right\rvert
\right] 
\bigr|_{\tilde{x}=X_{t}^{d,s,x}}
\right\rVert_{L^2(\P)}  dr\\
&\leq
\left\lVert
\left\lVert g^d(X^{d,t,\tilde{x}}_{T})  \right\rVert_{L^2(\P)}
\bigr|_{\tilde{x}=X_{t}^{d,s,x}}
\right\rVert_{L^2(\P)} 
+ 
\int_{t}^{T}\left\lVert\left\lVert
f(u^{d}(X^{d,t,\tilde{x}}_{r} ) ) 
 \right\rVert_{L^2(\P)}
\bigr|_{\tilde{x}=X_{t}^{d,s,x}}
\right\rVert_{L^2(\P)}  dr\\
&=
\left\lVert g^d(X^{d,s,x}_{T})  \right\rVert_{L^2(\P)} +
\int_{t}^{T}\left\lVert
f(u^{d}(X^{d,s,x}_{r} ) ) 
 \right\rVert_{L^2(\P)}  dr\\
&\leq\xeqref{m92} 2
(cd^c T^{-1})^{\frac{1}{2}}
\left( d^c+\lVert x\rVert^2 \right)^{\frac{1}{2}}e^{3.5 cT}+\int_{t}^{T}\xeqref{m91}
\left(
(cd^c T^{-3})^{\frac{1}{2}}+
c^\frac{1}{2}
\left
\lVert u^{d}(X^{d,s,x}_{r} )\right\rVert_{L^2(\P)}\right)dr\\
&\leq 3
(cd^c T^{-1})^{\frac{1}{2}}
\left( d^c+\lVert x\rVert^2 \right)^{\frac{1}{2}}e^{3.5 cT}+\int_{t}^{T}
c^\frac{1}{2}
\left
\lVert u^{d}(X^{d,s,x}_{r} )\right\rVert_{L^2(\P)}dr.
\end{split}\label{x65}\end{align}
Therefore, Gr\"onwall's lemma,
\eqref{growth u},
\eqref{k22}, and the fact that
$c\in [1,\infty)$
ensure for all
$d\in \N$,
$s\in[0,T]$, 
$t\in [s,T]$,
$x\in \R^d$ that 
\begin{align}
 \begin{split} 
\left\lVert
u^{d}(t,X_{t}^{d,s,x})\right\rVert_{L^2(\P)}&\leq  3
(cd^c T^{-1})^{\frac{1}{2}}
\left( d^c+\lVert x\rVert^2 \right)^{\frac{1}{2}}e^{3.5 cT}e^{c^{1/2}T}
\leq 
3
(cd^c T^{-1})^{\frac{1}{2}}
\left( d^c+\lVert x\rVert^2 \right)^{\frac{1}{2}}e^{4.5 cT}.
\end{split}\label{m93}
\end{align}
Hence, for all
$d\in \N$, 
$s\in[0,T]$, 
$t\in [s,T]$,
$x\in \R^d$ we have that 
\begin{align}
\lvert
u^{d}(t,x)\rvert\leq 3
(cd^c T^{-1})^{\frac{1}{2}}
\left( d^c+\lVert x\rVert^2 \right)^{\frac{1}{2}}e^{4.5 cT}.
\label{m93b}
\end{align}
Next,  \eqref{FK u epsilon} and the triangle inequality 
show for all
$d\in \N$, $\varepsilon\in(0,1)$,
$t\in[0,T]$, $x,y\in \R^d$ that 
\begin{align} \begin{split} 
\left\lvert u^{d,\varepsilon}(t,x)-
u^{d,\varepsilon}(t,y)\right\rvert
&\leq \E\!\left[\left\lvert g^d_\varepsilon(X^{d,\varepsilon,t,x}_{T})  
-g^d_\varepsilon(X^{d,\varepsilon,t,y}_{T} )\right\rvert
\right]\\&\quad  + 
\int_{t}^{T}\E\!\left [\left\lvert
f_\varepsilon(u^{d,\varepsilon}(X^{d,\varepsilon,t,x}_{r} ) ) 
-f_\varepsilon(u^{d,\varepsilon}(X^{d,\varepsilon,t,y}_{r} ) ) \right\rvert
\right] dr.\end{split}
\end{align}
This, the triangle inequality, the disintegration theorem, 
\eqref{Lip g epsilon}, \eqref{k67}, Jensen's inequality,
the fact that $c\in [1,\infty)$, and \eqref{k40}
imply for all
$d\in \N$, $\varepsilon\in(0,1)$,
$s\in[0,T]$, 
$t\in [s,T]$,
$x,y\in \R^d$ that 
\begin{align} \begin{split} 
&
\E\! \left[\left\lvert
u^{d,\varepsilon}(t,X^{d,\varepsilon,s,x}_{t})-
u^{d,\varepsilon}(t,X^{d,\varepsilon,s,y}_{t})
\right\rvert\right]\\&=
\E\! \left[\left\lvert
u^{d,\varepsilon}(t,\tilde{x})-
u^{d,\varepsilon}(t,\tilde{y})
\right\rvert
\Bigr|_{(\tilde{x},\tilde{y})=(X^{d,\varepsilon,s,x}_{t}, X^{d,\varepsilon,s,y}_{t})}
\right]\\
&\leq \E\!\left[
\E\!\left[\left\lvert g^d_\varepsilon(X^{d,\varepsilon,t,\tilde{x}}_{T})  
-g^d_\varepsilon(X^{d,\varepsilon,t,\tilde{y}}_{T} )\right\rvert
\right]\Bigr|_{(\tilde{x},\tilde{y})=(X^{d,\varepsilon,s,x}_{t}, X^{d,\varepsilon,s,y}_{t})}
\right]\\
&\quad
+
\int_{t}^{T}
\E\!\left[
\E\!\left [\left\lvert
f_\varepsilon(u^{d,\varepsilon}(X^{d,\varepsilon,t,\tilde{x}}_{r} ) ) 
-f_\varepsilon(u^{d,\varepsilon}(X^{d,\varepsilon,t,\tilde{y}}_{r} ) ) \right\rvert
\right] \Bigr|_{(\tilde{x},\tilde{y})=(X^{d,\varepsilon,s,x}_{t}, X^{d,\varepsilon,s,y}_{t})}
\right]
dr\\
&=
 \E\!\left[\left\lvert g^d_\varepsilon(X^{d,\varepsilon,s,x}_{T})  
-g^d_\varepsilon(X^{d,\varepsilon,s,y}_{T} )\right\rvert
\right] + 
\int_{t}^{T}\E\!\left [\left\lvert
f_\varepsilon(u^{d,\varepsilon}(X^{d,\varepsilon,s,x}_{r} ) ) 
-f_\varepsilon(u^{d,\varepsilon}(X^{d,\varepsilon,s,y}_{r} ) ) \right\rvert
\right] dr\\
&\leq 
\E\!\left[\xeqref{Lip g epsilon}(cd^cT^{-1})^\frac{1}{2}\left\lVert X^{d,\varepsilon,s,x}_{T}  
-X^{d,\varepsilon,s,y}_{T} \right\rVert
\right]+
\int_{t}^{T}\xeqref{k67}c^{\frac{1}{2}}\E\!\left [\left\lvert
u^{d,\varepsilon}(X^{d,\varepsilon,s,x}_{r} )  
-u^{d,\varepsilon}(X^{d,\varepsilon,s,y}_{r} )  \right\rvert
\right] dr\\
&\leq (cd^cT^{-1})^\frac{1}{2}
\left(
\E\!\left[\left\lVert X^{d,\varepsilon,s,x}_{T}  
-X^{d,\varepsilon,s,y}_{T} \right\rVert^2
\right]\right)^\frac{1}{2}+
\int_{t}^{T}c^{\frac{1}{2}}\E\!\left [\left\lvert
u^{d,\varepsilon}(X^{d,\varepsilon,s,x}_{r} )  
-u^{d,\varepsilon}(X^{d,\varepsilon,s,y}_{r} )  \right\rvert
\right]  dr\\
&\leq (cd^cT^{-1})^\frac{1}{2}
\left(\xeqref{k40}4\lVert x-y\rVert^2e^{8cT + 4cT^2}\right)^\frac{1}{2}+
\int_{t}^{T}c\E\!\left [\left\lvert
u^{d,\varepsilon}(X^{d,\varepsilon,s,x}_{r} )  
-u^{d,\varepsilon}(X^{d,\varepsilon,s,y}_{r} )  \right\rvert
\right]  dr\\
&\leq 2(cd^cT^{-1})^\frac{1}{2}
\lVert x-y\rVert e^{4cT + 2cT^2}+
\int_{t}^{T}c\E\!\left [\left\lvert
u^{d,\varepsilon}(X^{d,\varepsilon,s,x}_{r} )  
-u^{d,\varepsilon}(X^{d,\varepsilon,s,y}_{r} )  \right\rvert
\right]  dr.
\end{split}\label{x68}
\end{align}
Hence, Gr\"onwall's inequality, \eqref{growth u}, and \eqref{k22b}
show for all
$d\in \N$, $\varepsilon\in(0,1)$,
$s\in[0,T]$, 
$t\in [s,T]$,
$x,y\in \R^d$ that 
\begin{align} \begin{split} 
\E\! \left[\left\lvert
u^{d,\varepsilon}(t,X^{d,\varepsilon,s,x}_{t})-
u^{d,\varepsilon}(t,X^{d,\varepsilon,s,y}_{t})
\right\rvert\right]&\leq 
\xeqref{x68}2(cd^cT^{-1})^\frac{1}{2}
\lVert x-y\rVert e^{4cT + 2cT^2}\cdot e^{cT}\\
&=2(cd^cT^{-1})^\frac{1}{2}
\lVert x-y\rVert e^{5cT + 2cT^2}\end{split}
\end{align}
and hence
\begin{align}
\left\lvert
u^{d,\varepsilon}(t,x)-
u^{d,\varepsilon}(t,y)
\right\rvert\leq 
2(cd^cT^{-1})^\frac{1}{2}
\lVert x-y\rVert e^{5cT + 2cT^2}.\label{k97}
\end{align}
This proves \eqref{i03}.

Next, \eqref{k67}, \eqref{k97}, Jensen's inequality, and \eqref{k01}  ensure for all
$d\in \N$, $\varepsilon\in(0,1)$,
$t\in[0,T]$, 
$s\in [t,T]$,
$x\in \R^d$ that
\begin{align} \begin{split} 
&
\E\!\left[\left\lvert
f_\varepsilon(u^{d,\varepsilon}(s,X^{d,\varepsilon,t,x}_{s} ) )
-
f_\varepsilon(u^{d,\varepsilon}(s,X^{d,t,x}_{s} ) )\right\rvert
\right]\\&\leq \xeqref{k67}
c^{\frac{1}{2}}
\E\!\left[\left\lvert
u^{d,\varepsilon}(s,X^{d,\varepsilon,t,x}_{s} ) 
-u^{d,\varepsilon}(s,X^{d,t,x}_{s}  )\right\rvert
\right]\\
&
\leq 
c^\frac{1}{2}\xeqref{k97}2(cd^cT^{-1})^\frac{1}{2}
\E\!\left[
\left\lVert X^{d,\varepsilon,t,x}_{s} -
X^{d,t,x}_{s}
\right\rVert\right] e^{5cT + 2cT^2}
\\&\leq 
2c^\frac{1}{2}(cd^cT^{-1})^\frac{1}{2}
\left(
\E\!\left[
\left\lVert X^{d,\varepsilon,t,x}_{s} -
X^{d,t,x}_{s}
\right\rVert^2\right]\right)^\frac{1}{2} e^{5cT + 2cT^2}\\
&\leq 
2c^\frac{1}{2}(cd^cT^{-1})^\frac{1}{2}
\left(\xeqref{k01}
12\varepsilon c d^c (d^c+\lVert x\rVert^2) e^{20 cT+6cT^2}T\right)^\frac{1}{2} e^{5cT + 2cT^2}
\\&
\leq 8c^\frac{3}{2}d^c\varepsilon^\frac{1}{2}
( d^c+\lVert x\rVert^2)^\frac{1}{2}e^{15cT+5cT^2}.
\end{split}\label{m98}\end{align}
This, the triangle inequality, \eqref{k67},
\eqref{k67}, the fact that $c\in [1,\infty)$, the fact that $cT\leq e^{cT}$, and \eqref{m93}
imply for all
$d\in \N$, $\varepsilon\in(0,1)$,
$t\in[0,T]$, 
$x\in \R^d$ that
\begin{align}
 \begin{split}&\int_{t}^{T}
\E\!\left[\left\lvert
f_\varepsilon(u^{d,\varepsilon}(s,X^{d,\varepsilon,t,x}_{s} ) )-f(u^d(s,X^{d,t,x}_{s} ) ) 
\right\rvert
\right]ds\\&
\leq 
T\sup_{s\in[t,T]}\E\!\left[\left\lvert
f_\varepsilon(u^{d,\varepsilon}(s,X^{d,\varepsilon,t,x}_{s} ) )
-
f_\varepsilon(u^{d,\varepsilon}(s,X^{d,t,x}_{s} ) )\right\rvert
\right]\\
&\quad +\int_{t}^{T}
\E\!\left[\left\lvert
f_\varepsilon(u^{d,\varepsilon}(s,X^{d,t,x}_{s} ) )
-
f_\varepsilon(u^{d}(s,X^{d,t,x}_{s} ) )
\right\rvert\right]ds\\&\quad 
+
T\sup_{s\in[t,T]}\E\!\left[\left\lvert
f_\varepsilon(u^{d}(s,X^{d,t,x}_{s} ) )
-
f(u^{d}(s,X^{d,t,x}_{s} ) )\right\rvert\right]\\
&\leq T\cdot \xeqref{m98} 8c^\frac{3}{2}d^c\varepsilon^\frac{1}{2}
( d^c+\lVert x\rVert^2)^\frac{1}{2}e^{15 cT + 5cT^2}
+\int_{t}^{T}\xeqref{k67}c^{\frac{1}{2}}
\E\!\left[\left\lvert
u^{d,\varepsilon}(s,X^{d,t,x}_{s} ) 
-
u^{d}(s,X^{d,t,x}_{s} ) 
\right\rvert\right]ds\\
&\quad
+T\sup_{s\in[t,T]}\xeqref{k67}\left[\varepsilon^\frac{1}{2}\left(1+
\E\!\left[\left\lvert u^{d}(s,X^{d,t,x}_{s} ) \right\rvert^2\right]\right)\right] \end{split}
\end{align}
and
\begin{align} \begin{split} 
&\int_{t}^{T}\E\!\left[\left\lvert
f_\varepsilon(u^{d,\varepsilon}(s,X^{d,\varepsilon,t,x}_{s} ) )-f(u^d(s,X^{d,t,x}_{s} ) ) 
\right\rvert
\right]ds\\
&
\leq  8 cd^c\varepsilon^\frac{1}{2}
( d^c+\lVert x\rVert^2)^\frac{1}{2}e^{16cT+5cT^2}
+\int_{t}^{T}c^{\frac{1}{2}}
\E\!\left[\left\lvert
u^{d,\varepsilon}(s,X^{d,t,x}_{s} ) 
-
u^{d}(s,X^{d,t,x}_{s} ) 
\right\rvert\right]ds\\&\quad+
T\varepsilon^\frac{1}{2}
\left[1+
\left(\xeqref{m93}
3
(cd^c T^{-1})^{\frac{1}{2}}
\left( d^c+\lVert x\rVert^2 \right)^{\frac{1}{2}}e^{4.5 cT}\right)^2\right]\\
&= 8cd^c\varepsilon^\frac{1}{2}
( d^c+\lVert x\rVert^2)^\frac{1}{2}e^{16 cT + 5cT^2}
+\int_{t}^{T}c^{\frac{1}{2}}
\E\!\left[\left\lvert
u^{d,\varepsilon}(s,X^{d,t,x}_{s} ) 
-
u^{d}(s,X^{d,t,x}_{s} ) 
\right\rvert\right]ds\\&\quad  +
10cd^c
\varepsilon^\frac{1}{2}e^{9 cT} (d^c+\lVert x\rVert^2)\\
&\leq 18 cd^c\varepsilon^\frac{1}{2} (d^c+\lVert x\rVert^2)e^{16 cT+5cT^2}
+\int_{t}^{T}c
\E\!\left[\left\lvert
u^{d,\varepsilon}(s,X^{d,t,x}_{s} ) 
-
u^{d}(s,X^{d,t,x}_{s} ) 
\right\rvert\right]ds.
\end{split}\label{k99}
\end{align}
Hence, \eqref{k22},  the disintegration theorem, and the flow property  show for all
$d\in \N$, $\varepsilon\in(0,1)$,
$s\in[0,T]$, 
$t\in [s,T]$,
$x\in \R^d$ that
\begin{align}
 \begin{split} 
&
\int_{t}^{T}\E \!\left[
\E\!\left[\left\lvert
f_\varepsilon(u^{d,\varepsilon}(r,X^{d,\varepsilon,t,\tilde{x}}_{r} ) )-f(u^d(r,X^{d,t,\tilde{x}}_{r} ) ) 
\right\rvert
\right]\Bigr|_{\tilde{x}=X_{t}^{d,s,x}}\right]dr\\
&\leq 
 18 cd^c
\varepsilon^\frac{1}{2} \E\!\left[d^c+\left\lVert X_{t}^{d,s,x}\right\rVert^2\right]e^{16 cT + 5 cT^2}
+\int_{t}^{T}c
\E\!\left[
\E\!\left[\left\lvert
u^{d,\varepsilon}(r,X^{d,t,\tilde{x}}_{r} ) 
-
u^{d}(r,X^{d,t,\tilde{x}}_{r} ) 
\right\rvert
\right]\Bigr|_{\tilde{x}=X_{t}^{d,s,x}}
\right]dr\\
&\leq
18 cd^c
\varepsilon^\frac{1}{2} \left(\xeqref{k22}(d^c+\lVert x\rVert^2)e^{7cT}\right)
e^{16 cT+5cT^2}
+c\int_{t}^{T}
\E\!\left[\left\lvert
u^{d,\varepsilon}(r,X^{d,s,x}_{r} ) 
-
u^{d}(r,X^{d,s,x}_{r} ) 
\right\rvert
\right]dr
.
\end{split}\label{n02}\end{align}
This, the triangle inequality, \eqref{FK u}, \eqref{FK u epsilon}, and \eqref{k89} 
ensure for all
$d\in \N$, $\varepsilon\in(0,1)$,
$s\in[0,T]$, 
$t\in [s,T]$,
$x\in \R^d$ that
\begin{align} \begin{split} 
&
\E\!\left[\left\lvert
u^{d,\varepsilon}(t,X^{d,s,x}_{t} )-u^d(t,X^{d,s,x}_{t} )\right\rvert
\right]=
\E\!\left[\left\lvert
u^{d,\varepsilon}(t,\tilde{x} )-u^d(t,\tilde{x} )\right\rvert\Bigr|_{\tilde{x}=X^{d,s,x}_{t}}
\right]\\
&\leq\xeqref{FK u}\xeqref{FK u epsilon} \E\!\left[
\E \!\left[\left\lvert g^d_\varepsilon(X^{d,\varepsilon,t,\tilde{x}}_{T} )  -g^d(X^{d,t,\tilde{x}}_{T} )\right\rvert\right]\Bigr|_{\tilde{x}=X^{d,s,x}_{t}}
\right]\\&\quad 
+\int_{t}^{T}\E \!\left[
\E\!\left[\left\lvert
f_\varepsilon(u^{d,\varepsilon}(r,X^{d,\varepsilon,t,\tilde{x}}_{r} ) )-f(u^d(r,X^{d,t,\tilde{x}}_{r} ) ) 
\right\rvert
\right]\Bigr|_{\tilde{x}=X_{t}^{d,s,x}}\right]dr\\
&\leq \xeqref{k89}
5(\varepsilon c d^{2c})^\frac{1}{2}
(d^c+\lVert x\rVert^2)^{\frac{1}{2}}
e^{14 cT+3cT^2}\\
&\quad +
\xeqref{n02}
18 cd^c
\varepsilon^\frac{1}{2} \left(d^c+\lVert x\rVert^2)e^{7cT}\right)
e^{16 cT+5cT^2}
+c\int_{t}^{T}
\E\!\left[\left\lvert
u^{d,\varepsilon}(r,X^{d,s,x}_{r} ) 
-
u^{d}(r,X^{d,s,x}_{r} ) 
\right\rvert
\right]dr\\
&\leq 23cd^c\varepsilon^{\frac{1}{2}}
(d^c+\lVert x\rVert^2)
e^{23 cT+5cT^2}+c\int_{t}^{T}
\E\!\left[\left\lvert
u^{d,\varepsilon}(r,X^{d,s,x}_{r} ) 
-
u^{d}(r,X^{d,s,x}_{r} ) 
\right\rvert
\right]dr.
\end{split}\label{x75}
\end{align}
Hence, Gr\"onwall's lemma, 
\eqref{growth u}, and
\eqref{k22} imply for all
$d\in \N$, $\varepsilon\in(0,1)$,
$s\in[0,T]$, 
$t\in [s,T]$,
$x\in \R^d$ that
\begin{align} \begin{split} 
&
\E\!\left[\left\lvert
u^{d,\varepsilon}(t,X^{d,s,x}_{t} )-u^d(t,X^{d,s,x}_{t} )\right\rvert
\right]\\
&\leq \xeqref{x75}23cd^c\varepsilon^{\frac{1}{2}}
(d^c+\lVert x\rVert^2)
e^{23 cT+5cT^2}\cdot e^{cT}=23cd^c\varepsilon^{\frac{1}{2}}
(d^c+\lVert x\rVert^2)
e^{24cT+5cT^2}\end{split}
\end{align}
and hence
$
\left\lvert
u^{d,\varepsilon}(t,x)-
u^{d}(t,x)\right\rvert\leq 23cd^c\varepsilon^{\frac{1}{2}}
(d^c+\lVert x\rVert^2)
e^{24cT+5cT^2}$. This shows \eqref{i04}. 
The proof of \cref{x01} is thus completed.
\end{proof}

\section{Euler-Maruyama and MLP approximations revisited}\label{s03}
In \cref{k32} below we approximate the solution to
the SFPE \eqref{FK u epsilon}, associated to 
\eqref{SDE X epsilon}, through the solution to the SFPE \eqref{FK u K epsilon}, associated to the Euler-Maruyama approximation \eqref{SDE X epsilon K}.


\begin{proposition}[Discretization error] \label{k32}
Assume Settings \ref{setting PIDE}, \ref{setting approx NN}, 
and \ref{theta setting}.
For every $\varepsilon\in(0,1)$ let $f_\varepsilon\in C(\R,\R)$ be a function
which satisfies \eqref{k67} and \eqref{growth f epsilon}.
For every $ d,K\in\N $, $\varepsilon\in (0,1)$ 
let $u^{d,K,\varepsilon},u^{d,\varepsilon}\colon [0,T]\times \R^d\to\R$ 
be measurable functions introduced in 
\eqref{FK u K epsilon} and \eqref{FK u epsilon}, respectively.
Moreover, for convenience we use the notations
$X^{d,K,\varepsilon,t,x}:=X^{d,0,K,\varepsilon,t,x}$,
$X^{d,\varepsilon,t,x}:=X^{d,0,\varepsilon,t,x}$,
$W^d:=W^{d,0}$, and $\tilde{N}^d:=\tilde{N}^{d,0}$
for every $ d,K\in\N $,
$t\in [0,T]$,
$x\in\R^d$, $\varepsilon\in (0,1)$.
Then the following items hold.
\begin{enumerate}[(i)]
\item \label{k69b}For all  $d,K\in \N$, $t\in [0,T]$, $s\in [t,T]$,
$x\in \R^d$, $\varepsilon\in (0,1)$  we have that
\begin{align}
\E \!\left[d^c+\left\lVert X^{d,K,\varepsilon,t,x}_{s}\right\rVert^2
\right]\leq (d^c+\lVert x\rVert^2)e^{7c(s-t)}.
\end{align}
\item \label{k94}
For all
 $d,K\in \N$, $t\in [0,T]$, $x\in\R^d$,
 $\varepsilon\in (0,1)$, 
we have that
\begin{align}
\left\lvert u^{d,K,\varepsilon}(t,x)-
u^{d,\varepsilon}(t,x)\right\rvert\leq 
12c^{\frac{3}{2}}d^{\frac{c}{2}}(T+2)e^{21cT+5cT^2}
(d^c+\lVert x\rVert^2)^\frac{1}{2}
\frac{T^\frac{1}{2}}{K^\frac{1}{2}}.
\end{align}
\end{enumerate}
\end{proposition}

\begin{proof}
[Proof of \cref{k32}]
First,
for all
$d\in \N$, $\varepsilon\in(0,1)$,
$t\in[0,T]$, 
$x,y\in \R^d$ we have that
\begin{align}
\E \!\left[d^c+\left\lVert X^{d,\varepsilon,t,x}_{s}\right\rVert^2
\right]
\leq  (d^c+\lVert x\rVert^2)e^{7c(s-t)}\label{x63}
\end{align}
and
\begin{align}
\left\lvert
u^{d,\varepsilon}(t,x)-
u^{d,\varepsilon}(t,y)
\right\rvert\leq 
2(cd^cT^{-1})^\frac{1}{2}
\lVert x-y\rVert e^{5cT + 2cT^2}\label{k97b}
\end{align}
(cf. \cref{x01}).
Next,
the triangle inequality, the fact that
$\forall\,a_1,a_2\in \R\colon (a_1+a_2)^2
\leq 2\lvert a_1\rvert^2+2\lvert a_2\rvert^2$, 
\eqref{Lip NN coe}, and \eqref{growth NN coe} show for all  $d\in \N$, 
$x\in \R^d$, $\varepsilon\in (0,1)$  that
\begin{align}
\lVert
\beta^d_\varepsilon(x)
\rVert^2\leq 
2\lVert\beta^d_\varepsilon(0)\rVert^2
+2\lVert\beta^d_\varepsilon(x)-\beta^d_\varepsilon(0)
\rVert^2\leq  2\xeqref{growth NN coe}cd^c+\xeqref{Lip NN coe}2c\lVert x\rVert^2 
= 2c(d^c+\lVert x\rVert^2).\label{a01}
\end{align}
Similarly, we have for all  $d\in \N$, 
$x\in \R^d$, $\varepsilon\in (0,1)$  that
\begin{align}\label{a03}
\left\lVert
\sigma^d_\varepsilon(x)
\right\rVert^2_{\mathrm{F}}\leq  2c(d^c+\lVert x\rVert^2).
\end{align}
Moreover,
the triangle inequality, the fact that
$\forall\,a_1,a_2\in \R\colon (a_1+a_2)^2\leq 2\lvert a_1\rvert^2+2\lvert a_2\rvert^2$, \eqref{Lip NN coe}, and \eqref{growth NN coe} show for all  $d\in \N$, 
$x\in \R^d$, $\varepsilon\in (0,1)$  that
\begin{align} \begin{split} 
\int_{\R^d\setminus \{0\}}
\left\lVert
\gamma^d_\varepsilon(x,z)\right\rVert^2\nu^d(dz)
&\leq 
\int_{\R^d\setminus \{0\}}
2
\left\lVert
\gamma^d_\varepsilon(0,z)\right\rVert^2
+2\left\lVert
\gamma^d_\varepsilon(x,z)-\gamma^d_\varepsilon(0,z)\right\rVert
^2\nu^d(dz)\\&\leq \xeqref{growth NN coe}2cd^c+
\xeqref{Lip NN coe}2c\lVert x\rVert^2=2c(d^c+\lVert x\rVert^2).
\end{split}\label{a04}
\end{align}
Next, It\^o's formula (see, e.g., \cite[Theorem 3.1]{GW2021}) and \eqref{SDE X epsilon K} imply 
for all  $d,K\in \N$, $t\in [0,T]$, $s\in [t,T]$,
$x\in \R^d$, $\varepsilon\in (0,1)$  that
$\P$-a.s.\
\begin{align} \begin{split} 
\left\lVert X^{d,K,\varepsilon,t,x}_{s}\right\rVert^2
&=\lVert x\rVert^2
+\int_{t}^{s} \left(2\left\langle
X^{d,K,\varepsilon,t,x}_{\max\{t,\rdown{r-}_K\}},
\beta^d_\varepsilon(X^{d,K,\varepsilon,t,x}_{\max\{t,\rdown{r-}_K\}})
\right\rangle
+\left\lVert\sigma_\varepsilon^d(X^{d,K,\varepsilon,t,x}_{\max\{t,\rdown{r-}_K\}})\right\rVert^2_{\mathrm{F}}\right)dr
\\
&\quad+2
\int_{t}^{s}\sum_{i,j=1}^{d}
\left(X^{d,K,\varepsilon,t,x}_{\max\{t,\rdown{r-}_K\}}\right)_i
\left(\sigma_\varepsilon^d(X^{d,K,\varepsilon,t,x}_{\max\{t,\rdown{r-}_K\}})\right)_{ij}d(W^{d}_j)_r\\
&\quad +2\int_{t}^{s}\int_{\R^d\setminus\{0\}}
\left\langle
X^{d,K,\varepsilon,t,x}_{\max\{t,\rdown{r-}_K\}},
\gamma^d_\varepsilon(X^{d,K,\varepsilon,t,x}_{\max\{t,\rdown{r-}_K\}},z)
\right\rangle
\tilde{N}^d(dz,dr)\\&\quad +
\int_{t}^{s}\int_{\R^d\setminus\{0\}}
\left\lVert
\gamma^d_\varepsilon(X^{d,K,\varepsilon,t,x}_{\max\{t,\rdown{r-}_K\}},z)
\right\rVert^2N^d(dz,dr).
\end{split}\label{x93}
\end{align}
In addition, 
for every
  $d,n,K\in \N$, 
$x\in \R^d$, $\varepsilon\in (0,1)$ let $
\tau_n^{d,K,\varepsilon,x}\colon \Omega\to \R$ satisfy that
\begin{align} \begin{split} &
\tau_n^{d,K,\varepsilon,x}
=\inf \Biggl\{s\in [t,T] 
\colon \int_{t}^{s} \left(2\left\langle
X^{d,K,\varepsilon,t,x}_{\max\{t,\rdown{r-}_K\}},
\beta^d_\varepsilon(X^{d,K,\varepsilon,t,x}_{\max\{t,\rdown{r-}_K\}})
\right\rangle
+\left\lVert\sigma_\varepsilon^d(X^{d,K,\varepsilon,t,x}_{\max\{t,\rdown{r-}_K\}})\right\rVert^2_{\mathrm{F}}\right)dr
\\&\qquad\qquad\qquad
+\int_{t}^{s}\sum_{i,j=1}^{d}
\left\lvert\left(X^{d,K,\varepsilon,t,x}_{\max\{t,\rdown{r-}_K\}}\right)_i
\left(\sigma_\varepsilon^d(X^{d,K,\varepsilon,t,x}_{\max\{t,\rdown{r-}_K\}})\right)_{ij}
\right\rvert^2
dr\\&\qquad\qquad\qquad +\int_{t}^{s}\int_{\R^d\setminus \{0\}}
\sum_{i=1}^{d}\left\lvert
\left(X^{d,K,\varepsilon,t,x}_{\max\{t,\rdown{r-}_K\}}\right)_i
\left(\gamma^d_\varepsilon(X^{d,K,\varepsilon,t,x}_{\max\{t,\rdown{r-}_K\}},z)\right)_i
\right\rvert^2\nu^d(dz)\,dr\\&\qquad\qquad\qquad
+\int_{t}^{s}\int_{\R^d\setminus\{0\}}
\left\lVert
\gamma^d_\varepsilon(X^{d,K,\varepsilon,t,x}_{\max\{t,\rdown{r-}_K\}},z)
\right\rVert^2_{\mathrm{F}}\nu^d(dz)\,dr
\geq n
\Biggr\}\wedge T\end{split}
\end{align}
(with the convention that $\inf\emptyset=\infty$).
Then \eqref{x93}, the fact that
$\forall\,d\in \N, x,y\in \R^d\colon 2\langle x,y\rangle\leq \lVert x\rVert^2+\lVert y\rVert^2$,
\eqref{a01}, \eqref{a03}, and \eqref{a04} show
for all  $d,n,K\in \N$, $t\in [0,T]$, $s\in [t,T]$,
$x\in \R^d$, $\varepsilon\in (0,1)$  that
\begin{align} \begin{split} 
&\max \!\left\{\E \!\left[d^c +\left\lVert X^{d,K,\varepsilon,t,x}_{
\max\{t,\rdown{
s\wedge \tau_n^{d,K,\varepsilon,x}}_K \}
}\right\rVert^2
\right]
,
\E \!\left[d^c +\left\lVert X^{d,K,\varepsilon,t,x}_{
\max\{t,
s\wedge \tau_n^{d,K,\varepsilon,x} \}
}\right\rVert^2
\right]\right\}
\\&\leq \xeqref{x93}d^c +\lVert x\rVert^2 +
\E \!\left[ \int_{t}^{
s\wedge \tau_n^{d,K,\varepsilon,x} } \left(2\left\langle
X^{d,K,\varepsilon,t,x}_{\max\{t,\rdown{r}_K\}},
\beta^d_\varepsilon(X^{d,K,\varepsilon,t,x}_{\max\{t,\rdown{r}_K\}})
\right\rangle
+\left\lVert\sigma_\varepsilon^d(X^{d,K,\varepsilon,t,x}_{\max\{t,\rdown{r}_K\}})\right\rVert^2_{\mathrm{F}}\right)dr\right]\\
&\quad+
\E\!\left[
\int_{t}^{s\wedge \tau_n^{d,K,\varepsilon,x}}\int_{\R^d\setminus\{0\}}
\left\lVert
\gamma^d_\varepsilon(X^{d,K,\varepsilon,t,x}_{\max\{t,\rdown{r}_K\}},z)
\right\rVert^2\nu^d(dz)\,dr
\right]\\
&\leq d^c + 
\lVert x\rVert^2+
\E \!\left[ \int_{t}^{s\wedge \tau_n^{d,K,\varepsilon,x}}
\left\lVert
X^{d,K,\varepsilon,t,x}_{\max\{t,\rdown{r}_K\}}\right\rVert^2dr\right]
+
\E \!\left[ \int_{t}^{s\wedge \tau_n^{d,K,\varepsilon,x}}
\left\lVert
\beta^d_\varepsilon(X^{d,K,\varepsilon,t,x}_{\max\{t,\rdown{r}_K\}})\right\rVert^2dr\right]\\
&\quad 
+
\E \!\left[ \int_{t}^{s\wedge \tau_n^{d,K,\varepsilon,x}} \left\lVert\sigma_\varepsilon^d(X^{d,K,\varepsilon,t,x}_{\max\{t,\rdown{r}_K\}})\right\rVert^2_{\mathrm{F}}dr\right]\\
&\quad+
\E\!\left[
\int_{t}^{s\wedge \tau_n^{d,K,\varepsilon,x}}\int_{\R^d\setminus\{0\}}
\left\lVert
\gamma^d_\varepsilon(X^{d,K,\varepsilon,t,x}_{\max\{t,\rdown{r}_K\}},z)
\right\rVert^2\nu^d(dz)\,dr
\right].\end{split}
\end{align}
This,
\eqref{m01}--\eqref{m04}, and the fact that $c\geq1$ show
for all  $d,K,n\in \N$, $t\in [0,T]$, $s\in [t,T]$,
$x\in \R^d$, $\varepsilon\in (0,1)$  that
\begin{align} \begin{split} 
&\max \!\left\{\E \!\left[d^c +\left\lVert X^{d,K,\varepsilon,t,x}_{
\max\{t,\rdown{
s\wedge \tau_n^{d,K,\varepsilon,x}}_K \}
}\right\rVert^2
\right]
,
\E \!\left[d^c +\left\lVert X^{d,K,\varepsilon,t,x}_{
s\wedge \tau_n^{d,K,\varepsilon,x} 
}\right\rVert^2
\right]\right\}
\\
&\leq d^c + 
\lVert x\rVert^2+
\E \!\left[ \int_{t}^{s\wedge \tau_n^{d,K,\varepsilon,x}}\left(
d^c +
\left\lVert
X^{d,K,\varepsilon,t,x}_{\max\{t,\rdown{r}_K\}}\right\rVert^2\right)dr\right]\\
&\quad 
+
\E \!\left[ \int_{t}^{s\wedge \tau_n^{d,K,\varepsilon,x}}\xeqref{m01}
\xeqref{m03}
\xeqref{m04}
6c\left(d^c +
\left\lVert
X^{d,K,\varepsilon,t,x}_{\max\{t,\rdown{r}_K\}}\right\rVert^2\right) dr\right]\\
&\leq d^c + 
\lVert x\rVert^2+7c
\E \!\left[ \int_{t}^{s\wedge \tau_n^{d,K,\varepsilon,x}}\left(
d^c +
\left\lVert
X^{d,K,\varepsilon,t,x}_{\max\{t,\rdown{r}_K\}}\right\rVert^2\right)dr\right]\\
&\leq 
d^c + 
\lVert x\rVert^2+7c\int_{t}^{s}
\E \!\left[ 
d^c +
\left\lVert
X^{d,K,\varepsilon,t,x}_{\max\{t,\rdown{r 
\wedge \tau_{n}^{d,K,\varepsilon,x}}_K 
  \}}\right\rVert^2dr\right]
\end{split}\label{x59}
\end{align}
Hence, Fatou's lemma and Gr\"onwall's inequality  imply
for all  $d,K\in \N$, $t\in [0,T]$, $s\in [t,T]$,
$x\in \R^d$, $\varepsilon\in (0,1)$  that
\begin{align} \begin{split} 
\E \!\left[d^c +\left\lVert X^{d,K,\varepsilon,t,x}_{
\max\{t,\rdown{
s}_K \}
}\right\rVert^2
\right]
&\leq \liminf_{n\to\infty}
\E \!\left[d^c +\left\lVert X^{d,K,\varepsilon,t,x}_{\max\{t,\rdown{
s\wedge \tau_n^{d,K,\varepsilon,x}}_K \}}\right\rVert^2
\right]\\
&\leq (d^c +\lVert x\rVert^2)e^{7c(s-t)}.\end{split}\label{k22c}
\end{align}
Thus, Fatou's lemma,
\eqref{x59}, and the fact that
$\forall\,t\in [0,T], s\in [t,T]\colon 1+7c\int_{t}^{s} e^{7c(r-t)}\,dr=1+ e^{7c(r-t)}|_{r=s}^t=e^{7c(s-t)}$
ensure
for all  $d,K\in \N$, $t\in [0,T]$, $s\in [t,T]$,
$x\in \R^d$, $\varepsilon\in (0,1)$  that
\begin{align} \begin{split} 
\E \!\left[d^c +\left\lVert X^{d,K,\varepsilon,t,x}_{
s
}\right\rVert^2
\right]&\leq \liminf_{n\to\infty}
\E \!\left[d^c +\left\lVert X^{d,K,\varepsilon,t,x}_{
s\wedge \tau_n^{d,K,\varepsilon,x} 
}\right\rVert^2
\right]\\
&
\leq\xeqref{x59} d^c + 
\lVert x\rVert^2+7c
\E \!\left[ \int_{t}^{s}\left(
d^c +
\left\lVert
X^{d,K,\varepsilon,t,x}_{\max\{t,\rdown{r}_K\}}\right\rVert^2\right)dr\right]\\
&\leq 
d^c + 
\lVert x\rVert^2+7c\int_{t}^{s}\xeqref{k22c}(d^c +\lVert x\rVert^2)e^{7c(r-t)}\,dr\\
&=(
d^c + 
\lVert x\rVert^2)(1+7c\int_{t}^{s} e^{7c(r-t)}\,dr)\\
&=(
d^c + 
\lVert x\rVert^2)e^{7c(s-t)}.
\end{split}\label{k69}
\end{align}
This shows \eqref{k69b}.

Next, H\"older's inequality,  \eqref{a01}, and \eqref{k69} imply for all
 $d,K\in \N$, $t\in [0,T]$, $s,s'\in [t,T]$,
$x\in \R^d$, $\varepsilon\in (0,1)$  that
\begin{align} \begin{split} 
\E\!\left[\left\lVert
\int_{s}^{s'}
\beta^d_\varepsilon(
X^{d,K,\varepsilon,t,x}_{\max\{t,\rdown{r-}_K\}}
)\,dr\right\rVert^2\right]
&
\leq 
\E\!\left[\left(
\int_{s}^{s'}\left\lVert
\beta^d_\varepsilon(
X^{d,K,\varepsilon,t,x}_{\max\{t,\rdown{r-}_K\}}
)\right\rVert dr\right)^2\right]\\
&\leq
\E\!\left[\lvert s'-s\rvert
\int_{s}^{s'}\left\lVert
\beta^d_\varepsilon(
X^{d,K,\varepsilon,t,x}_{\max\{t,\rdown{r-}_K\}}
)\right\rVert^2dr\right]\\
&
\leq T\lvert s'-s\rvert
\sup_{r\in [s,s']}\E\!\left[\xeqref{a01}
2c\left(d^c+\left\lVert
X^{d,K,\varepsilon,t,x}_{\max\{t,\rdown{r-}_K\}}
\right\rVert^2\right)\right]\\
&\leq 
T\lvert s'-s\rvert\cdot \xeqref{k69}2c(d^c+\lVert x\rVert^2)e^{7cT}.
\end{split}\label{m71}\end{align}
In addition, It\^o's isometry, \eqref{a03}, and \eqref{k69}
show for all
 $d,K\in \N$, $t\in [0,T]$, $s,s'\in [t,T]$,
$x\in \R^d$, $\varepsilon\in (0,1)$  that
\begin{align} \begin{split} 
\E\!\left[\left\lVert
\int_{s}^{s'}
\sigma^d_\varepsilon(
X^{d,K,\varepsilon,t,x}_{\max\{t,\rdown{r-}_K\}}
)\,dW_r^{d}\right\rVert^2\right]
&=
\E\!\left[
\int_{s}^{s'}
\left\lVert
\sigma^d_\varepsilon(
X^{d,K,\varepsilon,t,x}_{\max\{t,\rdown{r-}_K\}}
)\right\rVert^2_{\mathrm{F}}dr\right]\\
&\leq \lvert s'-s\rvert
\sup_{r\in [s,s']}\E\!\left[\xeqref{a03}
2c\left(d^c+\left\lVert
X^{d,K,\varepsilon,t,x}_{\max\{t,\rdown{r-}_K\}}
\right\rVert^2\right)\right]\\&\leq \lvert s'-s\rvert\cdot \xeqref{k69}2c(d^c+\lVert x\rVert^2)e^{7cT}.
\end{split}\label{m72}
\end{align}
Next, It\^o's isometry, \eqref{a04}, and \eqref{k69}
imply for all
 $d,K\in \N$, $t\in [0,T]$, $s,s'\in [t,T]$,
$x\in \R^d$, $\varepsilon\in (0,1)$  that
\begin{align} \begin{split} 
&\E\!\left[\left\lVert
\int_{s}^{s'}
\int_{\R^d\setminus \{0\}}
\gamma^d_\varepsilon(
X^{d,K,\varepsilon,t,x}_{\max\{t,\rdown{r-}_K\}},z)\,\tilde{N}^{d}(dz,dr)
\right\rVert^2\right]\\&=
\E\!\left[
\int_{s}^{s'}
\int_{\R^d\setminus \{0\}}
\left\lVert
\gamma^d_\varepsilon(
X^{d,K,\varepsilon,t,x}_{\max\{t,\rdown{r-}_K\}},z)
\right\rVert^2
\nu^{d}(dz)\,dr
\right]\\
&\leq \lvert s'-s\rvert
\sup_{r\in [s,s']}\E\!\left[\xeqref{a04}
2c\left(d^c+\left\lVert
X^{d,K,\varepsilon,t,x}_{\max\{t,\rdown{r-}_K\}}
\right\rVert^2\right)\right]\\&\leq \lvert s'-s\rvert\cdot \xeqref{k69}2c(d^c+\lVert x\rVert^2)e^{7cT}.\end{split}\label{m73}
\end{align}
This, \eqref{SDE X epsilon K}, the fact that
$\forall\,d\in \N,x,y,z\in \R^d\colon 
\lVert x+y+z \rVert^2\leq 3\lVert x\rVert^2+3\lVert y\rVert^2+3\lVert z\rVert^2
$, \eqref{m71}, and \eqref{m72} demonstrate for all
 $d,K\in \N$, $t\in [0,T]$, $s,s'\in [t,T]$,
$x\in \R^d$, $\varepsilon\in (0,1)$  that
\begin{align} \begin{split} 
&\E\!\left[\left\lVert
X^{d,K,\varepsilon,t,x}_{s'}-
X^{d,K,\varepsilon,t,x}_{s}\right\rVert^2\right]
\\&\leq\xeqref{SDE X epsilon K} 3
\E\!\left[\left\lVert
\int_{s}^{s'}
\beta^d_\varepsilon(
X^{d,K,\varepsilon,t,x}_{\max\{t,\rdown{r-}_K\}}
)\,dr\right\rVert^2\right]
+3 \E\!\left[\left\lVert
\int_{s}^{s'}
\sigma^d_\varepsilon(
X^{d,K,\varepsilon,t,x}_{\max\{t,\rdown{r-}_K\}}
)\,dW_r^{d}\right\rVert^2\right]\\&\quad +
3\E\!\left[\left\lVert
\int_{s}^{s'}
\int_{\R^d\setminus \{0\}}
\gamma^d_\varepsilon(
X^{d,K,\varepsilon,t,x}_{\max\{t,\rdown{r-}_K\}},z)\,\tilde{N}^{d}(dz,dr)\right\rVert^2\right]\\
&\leq 3\cdot \xeqref{m71}T\lvert s'-s\rvert\cdot 2c(d^c+\lVert x\rVert^2)e^{7cT}
+3\cdot \xeqref{m72}
\lvert s'-s\rvert\cdot 2c(d^c+\lVert x\rVert^2)e^{7cT}
\\&\quad +3\cdot \xeqref{m73}
\lvert s'-s\rvert\cdot 2c(d^c+\lVert x\rVert^2)e^{7cT}
\\
&=6c(T+2)e^{7cT}(d^c+\lVert x\rVert^2)\lvert s'-s\rvert
.
\end{split}\label{t74}
\end{align}
Hence,
H\"older's inequality, \eqref{Lip NN coe}, the fact that
$\forall\,d\in\N,x,y\in \R^d\colon \lVert x+y\rVert^2\leq 2\lVert x\rVert^2+2\lVert y\rVert^2$,
and
\eqref{t74}
ensure for all
 $d,K\in \N$, $t\in [0,T]$, $s\in [t,T]$,
$x\in \R^d$, $\varepsilon\in (0,1)$  that
\begin{align} \begin{split} 
&\E \!\left[
\left\lVert
\int_{t}^{s}
\left(\beta^d_\varepsilon(
X^{d,K,\varepsilon,t,x}_{\max\{t,\rdown{r-}_K\}}
)-\beta^d_\varepsilon(
X^{d,\varepsilon,t,x}_{r-}
)\right)
dr\right\rVert^2
\right]\\
&
\leq 
\E \!\left[
\left(
\int_{t}^{s}
\left\lVert
\beta^d_\varepsilon(
X^{d,K,\varepsilon,t,x}_{\max\{t,\rdown{r-}_K\}}
)-\beta^d_\varepsilon(
X^{d,\varepsilon,t,x}_{r-}
)
\right\rVert
dr\right)^2
\right]\\
&\leq T
\E \!\left[
\int_{t}^{s}
\left\lVert
\beta^d_\varepsilon(
X^{d,K,\varepsilon,t,x}_{\max\{t,\rdown{r-}_K\}}
)-\beta^d_\varepsilon(
X^{d,\varepsilon,t,x}_{r-}
)
\right\rVert^2
dr
\right]\\&\leq T
\E \!\left[
\int_{t}^{s}\xeqref{Lip NN coe}
c
\left\lVert
X^{d,K,\varepsilon,t,x}_{\max\{t,\rdown{r-}_K\}}
-X^{d,\varepsilon,t,x}_{r-}
\right\rVert^2
dr
\right]\\
&\leq 
2cT\int_{t}^{s}
\E \!\left[
\left\lVert
X^{d,K,\varepsilon,t,x}_{\max\{t,\rdown{r}_K\}}
-
X^{d,K,\varepsilon,t,x}_{r}
\right\rVert^2\right]
+
2cT\int_{t}^{s}
\E \!\left[
\left\lVert
X^{d,K,\varepsilon,t,x}_{r}-
X^{d,\varepsilon,t,x}_{r}
\right\rVert^2\right]
dr\\
&\leq 2cT (s-t)\cdot\xeqref{t74} 6c(T+2)e^{7cT}(d^c+\lVert x\rVert^2)\frac{T}{K}
+
2cT\int_{t}^{s}
\E \!\left[
\left\lVert
X^{d,K,\varepsilon,t,x}_{r}-
X^{d,\varepsilon,t,x}_{r}
\right\rVert^2\right]
dr\\
&\leq   12c^2T^2(T+2)e^{7cT}(d^c+\lVert x\rVert^2)\frac{T}{K}
+
2cT\int_{t}^{s}
\E \!\left[
\left\lVert
X^{d,K,\varepsilon,t,x}_{r}-
X^{d,\varepsilon,t,x}_{r}
\right\rVert^2\right]
dr.
\end{split}\label{a75}
\end{align}
Next, It\^o's isometry, \eqref{Lip NN coe},
the fact that
$\forall\,d\in\N,x,y\in \R^d\colon \lVert x+y\rVert^2\leq 2\lVert x\rVert^2+2\lVert y\rVert^2$, and \eqref{t74}
show for all
 $d,K\in \N$, $t\in [0,T]$, $s\in [t,T]$,
$x\in \R^d$, $\varepsilon\in (0,1)$  that
\begin{align}
 \begin{split}
&
\E\! \left[\left\lVert\int_{t}^{s}
\left(
\sigma^d_\varepsilon(
X^{d,K,\varepsilon,t,x}_{\max\{t,\rdown{r-}_K\}}
)-\sigma^d_\varepsilon(
X^{d,\varepsilon,t,x}_{r-})\right)
dW_r^{d}\right\rVert^2
\right]\\&
=
\E\! \left[\int_{t}^{s}
\left\lVert
\sigma^d_\varepsilon(
X^{d,K,\varepsilon,t,x}_{\max\{t,\rdown{r-}_K\}}
)-\sigma^d_\varepsilon(
X^{d,\varepsilon,t,x}_{r-})
\right\rVert^2
dr
\right]\\
&\leq 
\E\! \left[\int_{t}^{s}
c
\left\lVert
X^{d,K,\varepsilon,t,x}_{\max\{t,\rdown{r-}_K\}}
-
X^{d,\varepsilon,t,x}_{r-}
\right\rVert^2
dr
\right]\\
&\leq 2c
\int_{t}^{s}
\E\! \left[
\left\lVert
X^{d,K,\varepsilon,t,x}_{\max\{t,\rdown{r}_K\}}
-
X^{d,K,\varepsilon,t,x}_{r}
\right\rVert^2
\right]dr+
2c
\int_{t}^{s}
\E\! \left[
\left\lVert
X^{d,K,\varepsilon,t,x}_{r}-
X^{d,\varepsilon,t,x}_{r}
\right\rVert^2
\right]dr\\
&\leq 2c (s-t)\xeqref{t74}\cdot 6c(T+2)e^{7cT}
(d^c+\lVert x\rVert^2)
\frac{T}{K}+
2c
\int_{t}^{s}
\E\! \left[
\left\lVert
X^{d,K,\varepsilon,t,x}_{r}-
X^{d,\varepsilon,t,x}_{r}
\right\rVert^2
\right]dr
\\
&
=12c^2T(T+2)e^{7cT}(d^c+\lVert x\rVert^2)\frac{T}{K}+
2c
\int_{t}^{s}
\E\! \left[
\left\lVert
X^{d,K,\varepsilon,t,x}_{r}-
X^{d,\varepsilon,t,x}_{r}
\right\rVert^2
\right]dr .
\end{split}\label{a76}
\end{align}
Moreover, It\^o's isometry, \eqref{Lip NN coe},
the fact that
$\forall\,d\in\N,x,y\in \R^d\colon \lVert x+y\rVert^2\leq 2\lVert x\rVert^2+2\lVert y\rVert^2$, and \eqref{t74}
imply for all
 $d,K\in \N$, $t\in [0,T]$, $s\in [t,T]$,
$x\in \R^d$, $\varepsilon\in (0,1)$  that
\begin{align} \begin{split} 
&
\E\!\left[\left\lVert
\int_{t}^{s}
\int_{\R^d\setminus \{0\}}\left(
\gamma^d_\varepsilon(
X^{d,K,\varepsilon,t,x}_{\max\{t,\rdown{r-}_K\}},z)
-\gamma^d_\varepsilon(
X^{d,\varepsilon,t,x}_{r-},z)\right)
\tilde{N}^{d}(dz,dr)\right\rVert^2\right]\\
&
=
\E\!\left[
\int_{t}^{s}
\int_{\R^d\setminus \{0\}}
\left\lVert
\gamma^d_\varepsilon(
X^{d,K,\varepsilon,t,x}_{\max\{t,\rdown{r-}_K\}},z)
-\gamma^d_\varepsilon(
X^{d,\varepsilon,t,x}_{r-},z)
\right\rVert^2
\nu^{d}(dz)dr
\right]\\
&\leq 
\E\!\left[
\int_{t}^{s}c
\left\lVert
X^{d,K,\varepsilon,t,x}_{\max\{t,\rdown{r-}_K\}}
-
X^{d,\varepsilon,t,x}_{r-}
\right\rVert^2dr
\right]
\\
&\leq 2c
\int_{t}^{s}
\E\! \left[
\left\lVert
X^{d,K,\varepsilon,t,x}_{\max\{t,\rdown{r}_K\}}
-
X^{d,K,\varepsilon,t,x}_{r}
\right\rVert^2
\right]dr+
2c
\int_{t}^{s}
\E\! \left[
\left\lVert
X^{d,K,\varepsilon,t,x}_{r}-
X^{d,\varepsilon,t,x}_{r}
\right\rVert^2
\right]dr\\
&\leq 2c (s-t)\xeqref{t74}\cdot 6c(T+2)e^{7cT}(d^c+\lVert x\rVert^2)\frac{T}{K}+
2c
\int_{t}^{s}
\E\! \left[
\left\lVert
X^{d,K,\varepsilon,t,x}_{r}-
X^{d,\varepsilon,t,x}_{r}
\right\rVert^2
\right]dr
\\
&
\leq 12c^2T(T+2)e^{7cT}(d^c+\lVert x\rVert^2)\frac{T}{K}+
2c
\int_{t}^{s}
\E\! \left[
\left\lVert
X^{d,K,\varepsilon,t,x}_{r}-
X^{d,\varepsilon,t,x}_{r}
\right\rVert^2
\right]dr 
.
\end{split}\label{a77}\end{align}
This, the fact that
$\forall\,d\in \N,x,y,z\in \R^d\colon 
\lVert x+y+z \rVert^2\leq 3\lVert x\rVert^2+3\lVert y\rVert^2+3\lVert z\rVert^2
$,
\eqref{a75}, and \eqref{a76}
imply for all
 $d,K\in \N$, $t\in [0,T]$, $s\in [t,T]$,
$x\in \R^d$, $\varepsilon\in (0,1)$  that
\begin{align} \begin{split} 
&\E\!\left[
\left\lVert X^{d,K,\varepsilon,t,x}_{s}
-
X^{d,\varepsilon,t,x}_{s}\right\rVert^2\right]
\\&\leq 3
\E\! \left[ \left\lVert
\int_{t}^{s}
\left(\beta^d_\varepsilon(
X^{d,K,\varepsilon,t,x}_{\max\{t,\rdown{r-}_K\}}
)-\beta^d_\varepsilon(
X^{d,\varepsilon,t,x}_{r-}
)\right)dr\right\rVert^2\right]
\\&\quad +3 \E\! \left[ \left\lVert
\int_{t}^{s}
\left(
\sigma^d_\varepsilon(
X^{d,K,\varepsilon,t,x}_{\max\{t,\rdown{r-}_K\}}
)-\sigma^d_\varepsilon(
X^{d,\varepsilon,t,x}_{r-})\right)
dW_r^{d}\right\rVert^2\right]\\&\quad +
3\E\! \left[ \left\lVert
\int_{t}^{s}
\int_{\R^d\setminus \{0\}}\left(
\gamma^d_\varepsilon(
X^{d,K,\varepsilon,t,x}_{\max\{t,\rdown{r-}_K\}},z)
-\gamma^d_\varepsilon(
X^{d,\varepsilon,t,x}_{r-},z)\right)
\tilde{N}^{d}(dz,dr)\right\rVert^2\right]\\
\end{split}
\end{align}
and
\begin{align}
 \begin{split} 
&\E\!\left[
\left\lVert X^{d,K,\varepsilon,t,x}_{s}
-
X^{d,\varepsilon,t,x}_{s}\right\rVert^2\right]\\&\leq 3\left[ \xeqref{a75}
 12c^2T^2(T+2)e^{7cT}(d^c+\lVert x\rVert^2)\frac{T}{K}
+
2cT\int_{t}^{s}
\E \!\left[
\left\lVert
X^{d,K,\varepsilon,t,x}_{r}-
X^{d,\varepsilon,t,x}_{r}
\right\rVert^2\right]
dr\right]\\
&\quad+3\left[\xeqref{a76}12c^2T(T+2)e^{7cT}(d^c+\lVert x\rVert^2)\frac{T}{K}+
2c
\int_{t}^{s}
\E\! \left[
\left\lVert
X^{d,K,\varepsilon,t,x}_{r}-
X^{d,\varepsilon,t,x}_{r}
\right\rVert^2
\right]dr \right]\\
&\quad +3\left[\xeqref{a77}
12c^2T(T+2)e^{7cT}(d^c+\lVert x\rVert^2)\frac{T}{K}+
2c
\int_{t}^{s}
\E\! \left[
\left\lVert
X^{d,K,\varepsilon,t,x}_{r}-
X^{d,\varepsilon,t,x}_{r}
\right\rVert^2
\right]dr \right]\\
&=36 c^2T(T+2)^2e^{7cT}(d^c+\lVert x\rVert^2)\frac{T}{K}+6c(T+2)
\int_{t}^{s}
\E\! \left[
\left\lVert
X^{d,K,\varepsilon,t,x}_{r}-
X^{d,\varepsilon,t,x}_{r}
\right\rVert^2
\right]dr .
\end{split}\label{x06}
\end{align}
Hence, \eqref{x63}, \eqref{k69}, and Gr\"onwall's inequality 
demonstrate for all
 $d,K\in \N$, $t\in [0,T]$, $s\in [t,T]$,
$x\in \R^d$, $\varepsilon\in (0,1)$  that
\begin{align}
\E\!\left[
\left\lVert X^{d,K,\varepsilon,t,x}_{s}
-
X^{d,\varepsilon,t,x}_{s}\right\rVert^2\right]\leq \xeqref{x06}
36c^2T(T+2)^2e^{7cT}(d^c+\lVert x\rVert^2)\frac{T}{K}
e^{6c(T+2)T}.\label{i79}
\end{align}
This,
\eqref{Lip g epsilon}, and Jensen's inequality
show for all
 $d,K\in \N$, $t\in [0,T]$, 
$x\in \R^d$, $\varepsilon\in (0,1)$  that
\begin{align} \begin{split} 
\E\!\left[\left\lvert g^d_\varepsilon(
X^{d,K,\varepsilon,t,x}_{T} )-g^d_\varepsilon(
X^{d,\varepsilon,t,x}_{T} )\right\rvert\right]
&\leq (cd^c)^\frac{1}{2}
T^{-\frac{1}{2}}
\E\!\left[\left\lVert 
X^{d,K,\varepsilon,t,x}_{T} -
X^{d,\varepsilon,t,x}_{T} \right\rVert\right]\\
&\leq (cd^c)^\frac{1}{2}
T^{-\frac{1}{2}}
\left(
\E\!\left[\left\lVert 
X^{d,K,\varepsilon,t,x}_{T} -
X^{d,\varepsilon,t,x}_{T} \right\rVert^2\right]
\right)^\frac{1}{2}\\&\leq
(cd^c)^\frac{1}{2}
T^{-\frac{1}{2}}
\left(\xeqref{i79}36 c^2T(T+2)^2e^{7cT}(d^c+\lVert x\rVert^2)\frac{T}{K}
e^{6c(T+2)T}
\right)^\frac{1}{2} \\
&\leq 6 c^\frac{3}{2}d^\frac{c}{2}(T+2)e^{10cT+3cT^2}
(d^c+\lVert x\rVert^2)^\frac{1}{2}
\frac{T^\frac{1}{2}}{K^\frac{1}{2}}.
\end{split}\label{k84}
\end{align}
 Next, \eqref{k97b}, Jensen's inequality, and \eqref{i79}
imply for all
 $d,K\in \N$, $t\in [0,T]$, $r\in[t,T]$,
$x\in \R^d$, $\varepsilon\in (0,1)$  that
\begin{align} \begin{split} 
&
\E\!\left[\left\lvert u^{d,\varepsilon}(r,
X^{d,K,\varepsilon,t,x}_{r} )
-u^{d,\varepsilon}(r,
X^{d,\varepsilon,t,x}_{r} )\right\rvert\right]
\\&\leq \xeqref{k97b} 
2(cd^cT^{-1})^\frac{1}{2}
\E \!\left[
\left\lVert X^{d,K,\varepsilon,t,x}_{r}-
X^{d,\varepsilon,t,x}_{r}
\right\rVert 
\right]
e^{5cT + 2cT^2}\\
&
\leq 2(cd^cT^{-1})^\frac{1}{2}
\left(
\E \!\left[
\left\lVert X^{d,K,\varepsilon,t,x}_{r}-
X^{d,\varepsilon,t,x}_{r}
\right\rVert ^2
\right]\right)^\frac{1}{2}
e^{5cT + 2cT^2}\\
&\leq 
2(cd^cT^{-1})^\frac{1}{2}
\left(\xeqref{i79}36 c^2T(T+2)^2e^{7cT}(d^c+\lVert x\rVert^2)\frac{T}{K}
e^{6c(T+2)T}\right)^\frac{1}{2}
e^{5cT + 2cT^2}\\
&\leq 12c^{\frac{3}{2}}d^{\frac{c}{2}}(T+2)e^{15cT+5cT^2}(d^c+\lVert x\rVert^2)^\frac{1}{2}\frac{T^\frac{1}{2}}{K^\frac{1}{2}}.
\end{split}\label{k85}\end{align}
 Furthermore,
Jensen's inequality and \eqref{k69}
show for all
 $d,K\in \N$, $t\in [0,T]$, $r\in[t,T]$,
$x\in \R^d$, $\varepsilon\in (0,1)$  that
\begin{align} \begin{split} 
&
\E\!\left[\left\lvert u^{d,K,\varepsilon}(r,
X^{d,K,\varepsilon,t,x}_{r} )-
u^{d,\varepsilon}(r,
X^{d,K,\varepsilon,t,x}_{r} )\right\rvert\right]
\\
&
\leq  e^{3.5c(T-r)}
\E \!\left[ \left(d^c+\left\lVert
X^{d,K,\varepsilon,t,x}_{r}
\right\rVert^2\right)^\frac{1}{2} \right]\sup_{y\in\R^d}
\frac{\left\lvert u^{d,K,\varepsilon}(r,y)-
u^{d,\varepsilon}(r,y)\right\rvert}{e^{3.5c(T-r)}(d^c+\lVert y\rVert^2)^{\frac{1}{2}}}\\
&\leq e^{3.5c(T-r)}
\left(
\E \!\left[ d^c+\left\lVert
X^{d,K,\varepsilon,t,x}_{r}
\right\rVert^2 \right]\right)^\frac{1}{2}
\sup_{y\in\R^d}
\frac{\left\lvert u^{d,K,\varepsilon}(r,y)-
u^{d,\varepsilon}(r,y)\right\rvert}{e^{3.5c(T-r)}(d^c+\lVert y\rVert^2)^{\frac{1}{2}}}
\\
&\leq e^{3.5c(T-r)} \xeqref{k69}(d^c+\lVert x\rVert^2)^\frac{1}{2}e^{3.5c(r-t)}
\sup_{y\in\R^d}
\frac{\left\lvert u^{d,K,\varepsilon}(r,y)-
u^{d,\varepsilon}(r,y)\right\rvert}{e^{3.5c(T-r)}(d^c+\lVert y\rVert^2)^{\frac{1}{2}}}\\
&=(d^c+\lVert x\rVert^2)^\frac{1}{2}e^{3.5c(T-t)}
\sup_{y\in\R^d}
\frac{\left\lvert u^{d,K,\varepsilon}(r,y)-
u^{d,\varepsilon}(r,y)\right\rvert}{e^{3.5c(T-r)}(d^c+\lVert y\rVert^2)^{\frac{1}{2}}}.\end{split}\label{k86}
\end{align}
 This, the triangle inequality,
\eqref{Lip g epsilon}, \eqref{k67}, the fact that
$c\geq 1$, 
\eqref{k84}, \eqref{k85}, and
the fact that
$1+cT\leq e^{cT}$ imply for all
 $d,K\in \N$, $t\in [0,T]$, 
$x\in \R^d$, $\varepsilon\in (0,1)$  that
\begin{align} \begin{split} 
&
\left\lvert
u^{d,K,\varepsilon}(t,x)-u^{d,\varepsilon}(t,x)\right\rvert\\
&\leq 
\E\!\left[\left\lvert g^d_\varepsilon(
X^{d,K,\varepsilon,t,x}_{T} )-g^d_\varepsilon(
X^{d,\varepsilon,t,x}_{T} )\right\rvert\right]
+\int_{t}^{T}
\E\!\left[\left\lvert f_\varepsilon(u^{d,K,\varepsilon}(r,
X^{d,K,\varepsilon,t,x}_{r} ))
-f_\varepsilon(u^{d,\varepsilon}(r,
X^{d,\varepsilon,t,x}_{r} ))\right\rvert
\right]dr\\
&\leq \xeqref{Lip g epsilon}\E\!\left[\left\lvert g^d_\varepsilon(
X^{d,K,\varepsilon,t,x}_{T} )-g^d_\varepsilon(
X^{d,\varepsilon,t,x}_{T} )\right\rvert\right]
+\int_{t}^{T}c
\E\!\left[\left\lvert u^{d,K,\varepsilon}(r,
X^{d,K,\varepsilon,t,x}_{r} )-
u^{d,\varepsilon}(r,
X^{d,K,\varepsilon,t,x}_{r} )\right\rvert\right]\\
&\quad+
\int_{t}^{T}c
\E\!\left[\left\lvert u^{d,\varepsilon}(r,
X^{d,K,\varepsilon,t,x}_{r} )
-u^{d,\varepsilon}(r,
X^{d,\varepsilon,t,x}_{r} )\right\rvert\right]dr\\
\end{split}
\end{align}
and 
\begin{align}
 \begin{split}
\left\lvert
u^{d,K,\varepsilon}(t,x)-u^{d,\varepsilon}(t,x)\right\rvert&\leq\xeqref{k84} 6 c^\frac{3}{2}d^\frac{c}{2}(T+2)e^{10cT+3cT^2}
(d^c+\lVert x\rVert^2)^\frac{1}{2}
\frac{T^\frac{1}{2}}{K^\frac{1}{2}}\\
&\quad +\int_{t}^{T}c
\xeqref{k86}(d^c+\lVert x\rVert^2)^\frac{1}{2}e^{3.5c(T-t)}
\sup_{y\in\R^d}
\frac{\left\lvert u^{d,K,\varepsilon}(r,y)-
u^{d,\varepsilon}(r,y)\right\rvert}{e^{3.5c(T-r)}(d^c+\lVert y\rVert^2)^{\frac{1}{2}}}
dr\\
&\quad+ Tc\cdot \xeqref{k85}
12 c^{\frac{3}{2}}d^{\frac{c}{2}}(T+2)e^{15cT+5cT^2}(d^c+\lVert x\rVert^2)^\frac{1}{2}\frac{T^\frac{1}{2}}{K^\frac{1}{2}}\\
&\leq 
12c^{\frac{3}{2}}d^{\frac{c}{2}}(T+2)e^{16cT+5cT^2}(d^c+\lVert x\rVert^2)^\frac{1}{2}\frac{T^\frac{1}{2}}{K^\frac{1}{2}}\\&\quad +\int_{t}^{T}c
 (d^c+\lVert x\rVert^2)^\frac{1}{2}e^{3.5c(T-t)}
\sup_{y\in\R^d}
\frac{\left\lvert u^{d,K,\varepsilon}(r,y)-
u^{d,\varepsilon}(r,y)\right\rvert}{e^{3.5c(T-r)} (d^c+\lVert y\rVert^2)^{\frac{1}{2}}} \,
dr.
\end{split}\label{x91}
\end{align}
 Dividing by $(d^c+\lVert x\rVert^2)^\frac{1}{2}e^{3.5c(T-t)}$ shows
for all
 $d,K\in \N$, $t\in [0,T]$, 
 $\varepsilon\in (0,1)$  that
\begin{align} \begin{split} 
&\sup_{x\in\R^d}
\frac{\left\lvert u^{d,K,\varepsilon}(t,x)-
u^{d,\varepsilon}(t,x)\right\rvert}{e^{3.5c(T-t)}(d^c+\lVert y\rVert^2)^{\frac{1}{2}}}\\
&\leq \xeqref{x91}
12 c^{\frac{3}{2}}d^{\frac{c}{2}}(T+2)e^{16cT+5cT^2}\frac{T^\frac{1}{2}}{K^\frac{1}{2}}+\int_{t}^{T}c
 \sup_{y\in\R^d}
\frac{\left\lvert u^{d,K,\varepsilon}(r,y)-
u^{d,\varepsilon}(r,y)\right\rvert}{ e^{3.5c(T-r)} (d^c+\lVert y\rVert^2)^{\frac{1}{2}}}\,
dr.\end{split}\label{x13}
\end{align}
Hence, \eqref{growth u}, \eqref{x63}, \eqref{k69}, and Gr\"onwall's inequality demonstrate for all
 $d,K\in \N$, $t\in [0,T]$, 
 $\varepsilon\in (0,1)$  that
\begin{align} \begin{split} 
\sup_{y\in\R^d}
\frac{\left\lvert u^{d,K,\varepsilon}(t,y)-
u^{d,\varepsilon}(t,y)\right\rvert}{ e^{3.5c(T-t)} (d^c+\lVert y\rVert^2)^{\frac{1}{2}}}&\leq \xeqref{x13}
12c^{\frac{3}{2}}d^{\frac{c}{2}}(T+2)e^{16cT+5cT^2}\frac{T^\frac{1}{2}}{K^\frac{1}{2}}\cdot e^{cT}\\
&=12c^{\frac{3}{2}}d^{\frac{c}{2}}(T+2)e^{17cT+5cT^2}\frac{T^\frac{1}{2}}{K^\frac{1}{2}}
.\end{split}
\end{align}
Therefore, for all
 $d,K\in \N$, $t\in [0,T]$, $x\in\R^d$,
 $\varepsilon\in (0,1)$ 
we have that
\begin{align}
\left\lvert u^{d,K,\varepsilon}(t,x)-
u^{d,\varepsilon}(t,x)\right\rvert\leq 
12c^{\frac{3}{2}}d^{\frac{c}{2}}(T+2)e^{21cT+5cT^2}
(d^c+\lVert x\rVert^2)^\frac{1}{2}
\frac{T^\frac{1}{2}}{K^\frac{1}{2}}.
\end{align}
This shows \eqref{k94}. The proof of \cref{k32} is thus completed.
\end{proof}

In \cref{a31} below we approximate the solution to SFPE \eqref{FK u K epsilon},
 associated to \eqref{SDE X epsilon K}, by the MLP approximation \eqref{k34}.

\begin{proposition}\label{a31}
Assume Settings \ref{setting PIDE}, \ref{setting approx NN}, \ref{theta setting}.
For every $\varepsilon\in(0,1)$ let $f_\varepsilon\in C(\R,\R)$ be a function
which satisfies \eqref{k67} and \eqref{growth f epsilon}.
For every $ d,K\in\N $, $\varepsilon\in (0,1)$
let $u^{d,K,\varepsilon}\colon [0,T]\times \R^d\to\R$ 
be the measurable function introduced in \eqref{FK u K epsilon}
and let 
$U^{d,\theta,K,\varepsilon}_{n,m}\colon [0,T]\times\R^d\times \Omega\to \R$, 
$\theta\in \Theta$,  
$n,m\in \Z$, 
satisfy for all 
$\theta\in \Theta$,  
$n\in \N_0$, $m\in \N$, $t\in [0,T]$, $x\in \R^d$ that
\begin{align} \begin{split} &U^{d,\theta,K,\varepsilon}_{n,m}(t,x)=
\frac{\1_{\N}(n)}{m^n}\sum_{i=1}^{m^n}g^d_\varepsilon\Bigl( X^{ d,(\theta,0,-i), K,\varepsilon,t,x}_{T} \Bigr)\\
&\quad
+\sum_{\ell=0}^{n-1}
\frac{(T-t)}{m^{n-\ell}}
\sum_{i=1}^{m^{n-\ell}}\Bigl(
f_\varepsilon\circ
U^{d,(\theta,\ell,i),K,\varepsilon}_{\ell,m}
-\1_{\N}(\ell)
f_\varepsilon\circ
U^{d,(\theta,-\ell,i),K,\varepsilon}_{\ell-1,m}\Bigr)
\Bigl(\mathfrak{T}_t^{(\theta,\ell,i)}, 
 X^{ d,(\theta,\ell,i), K,\varepsilon,t,x}_{\mathfrak{T}_t^{(\theta,\ell,i)}} 
\Bigr).\end{split}\label{k34}
\end{align}
Then
for all
$d,K,n,m\in\N $, 
$\theta\in\Theta$,
$\varepsilon\in (0,1)$, $t\in[0,T]$, $x\in\R^d$ we have
that $U^{d,\theta,K,\varepsilon}_{n,m}$ is measurable and
$
\E\!\left[\left(
\left\lvert
U^{d,\theta,K,\varepsilon}_{n,m}(t,x)-u^{d,K,\varepsilon}(t,x)\right\rvert^2\right]\right)^\frac{1}{2}\leq 6
e^\frac{m}{2}m^{-\frac{n}{2}}e^{12 c Tn}
(cd^c T^{-1})^{\frac{1}{2}}\left(d^c+\lVert x\rVert^2\right)^\frac{1}{2}.
$
\end{proposition}

\begin{proof}[Proof of \cref{a31}]
For measurability see \cite[Lemma 3.2]{HJKN2020}.
Next, 
\eqref{Lip g epsilon}, \eqref{growth NN coe}, \eqref{k67},
\eqref{growth f epsilon},
and the triangle inequality 
show for all
$d\in \N$, $\varepsilon\in(0,1)$,
 $x\in \R^d$, $w\in\R$ that 
\begin{align} \begin{split} 
\left\lvert
g_\varepsilon^d(x)\right\rvert
\leq \left\lvert
g_\varepsilon^d(0)\right\rvert
+\xeqref{Lip g epsilon}(cd^cT^{-1})^{\frac{1}{2}}\lVert x\rVert
\leq \xeqref{growth NN coe}(cd^c T^{-1})^{\frac{1}{2}}
+(cd^c T^{-1})^{\frac{1}{2}}\lVert x\rVert
\leq 
2
(cd^c T^{-1})^{\frac{1}{2}} (d^c+\lVert x\rVert^2)^\frac{1}{2}\label{k90b}\end{split}
\end{align}
and 
\begin{align}
\lvert
f_\varepsilon(w)\rvert\leq\xeqref{k67} \lvert f_\varepsilon(0)\rvert +c^\frac{1}{2}\lvert w\rvert\leq \xeqref{growth f epsilon}(cd^c T^{-3})^{\frac{1}{2}}+
c^\frac{1}{2}\lvert w\rvert.\label{m99}
\end{align}
First, 
for all  $d,K\in \N$, $t\in [0,T]$, $s\in [t,T]$,
$x\in \R^d$, $\varepsilon\in (0,1)$  we have that
\begin{align}
\E \!\left[d^c+\left\lVert X^{d,0,K,\varepsilon,x}_{t,s}\right\rVert^2
\right]\leq (d^c+\lVert x\rVert^2)e^{7c(s-t)}
\label{k98}
\end{align}
(cf. \cref{k32}). This, \eqref{k90b}, and Jensen's inequality demonstrate for all 
$d,K\in \N$, $t\in [0,T]$
$x\in \R^d$, $\varepsilon\in (0,1)$ 
that
\begin{align} \begin{split} 
\E\!\left[g^d_\varepsilon(
X^{d,0,K,\varepsilon,t,x}_{T} )\right]
&\leq\xeqref{k90b} 2
(cd^c T^{-1})^{\frac{1}{2}}\E\!\left[ \left(d^c+\left\lVert X^{d,0,K,\varepsilon,t,x}_{T}\right\rVert^2\right)^\frac{1}{2}\right]\\&\leq 
2
(cd^c T^{-1})^{\frac{1}{2}}\left(\E\!\left[ d^c+\left\lVert X^{d,0,K,\varepsilon,t,x}_{T}\right\rVert^2\right]\right)^\frac{1}{2}\\
&\leq 
2
(cd^c T^{-1})^{\frac{1}{2}}\left( \xeqref{k98}(d^c+\lVert x\rVert^2)e^{7cT} \right)^\frac{1}{2}\\&=
2
(cd^c T^{-1})^{\frac{1}{2}}(d^c+\lVert x\rVert^2)^\frac{1}{2}e^{3.5cT}.
\end{split}\label{b01}\end{align}
Thus, \eqref{FK u K epsilon}, \eqref{m99}, the fact that
$c\geq 1$, Jensen's inequality, \eqref{k98}
 show for all 
$d,K\in \N$, $t\in [0,T]$
$x\in \R^d$, $\varepsilon\in (0,1)$ 
that
\begin{align} \begin{split} 
&
\left\lvert u^{d,K,\varepsilon}(t,x)\right\rvert\\&\leq \xeqref{FK u K epsilon}
\E\!\left[\left\lvert g^d_\varepsilon(
X^{d,0,K,\varepsilon,t,x}_{T} )\right\rvert\right]
+\int_{t}^{T}
\E\!\left[\left\lvert f_\varepsilon(u^{d,K,\varepsilon}(r,
X^{d,0,K,\varepsilon,t,x}_{r} ))\right\rvert\right]dr\\
&\leq \xeqref{b01}
2
(cd^c T^{-1})^{\frac{1}{2}}(d^c+\lVert x\rVert^2)^\frac{1}{2}e^{3.5cT}+\int_{t}^{T}\left(\xeqref{m99}(cd^cT^{-3})^\frac{1}{2}+c^\frac{1}{2}
\E\!\left[\left\lvert
u^{d,K,\varepsilon}(r,
X^{d,0,K,\varepsilon,t,x}_{r} )\right\rvert\right]\right)dr\\
&\leq 
3
(cd^c T^{-1})^{\frac{1}{2}}(d^c+\lVert x\rVert^2)^\frac{1}{2}e^{3.5cT}\\&\quad +\int_{t}^{T}c
\left[ \sup_{y\in\R^d}\frac{\left\lvert u^{d,K,\varepsilon}(r,y)\right\rvert}{e^{3.5c(T-r)}(d^c+\lVert y\rVert^2)^\frac{1}{2}}\right]
e^{3.5c(T-r)}
\E\!\left[\left(d^c+\left\lVert 
X^{d,0,K,\varepsilon,t,x}_{r}\right\rVert^2\right)^\frac{1}{2}
\right]dr\\
&\leq 
3
(cd^c T^{-1})^{\frac{1}{2}}(d^c+\lVert x\rVert^2)^\frac{1}{2}e^{3.5cT}\\&\quad+\int_{t}^{T}c
\left[ \sup_{y\in\R^d}\frac{\left\lvert u^{d,K,\varepsilon}(r,y)\right\rvert}{e^{3.5c(T-r)}(d^c+\lVert y\rVert^2)^\frac{1}{2}}\right]
e^{3.5c(T-r)}
\left(
\E\!\left[d^c+\left\lVert 
X^{d,0,K,\varepsilon,t,x}_{r}\right\rVert^2
\right]\right)^\frac{1}{2}dr\\
&\leq 
3
(cd^c T^{-1})^{\frac{1}{2}}(d^c+\lVert x\rVert^2)^\frac{1}{2}e^{3.5cT}\\&\quad+\int_{t}^{T}c
\left[ \sup_{y\in\R^d}\frac{\left\lvert u^{d,K,\varepsilon}(r,y)\right\rvert}{e^{3.5c(T-r)}(d^c+\lVert y\rVert^2)^\frac{1}{2}}\right]
e^{3.5c(T-r)}
\left(\xeqref{k98} (d^c+\lVert x\rVert^2)e^{7c(r-t)} \right)^\frac{1}{2}dr\\
&=3
(cd^c T^{-1})^{\frac{1}{2}}(d^c+\lVert x\rVert^2)^\frac{1}{2}e^{3.5cT}
\\&\quad+
\int_{t}^{T}c
\left[ \sup_{y\in\R^d}\frac{\left\lvert u^{d,K,\varepsilon}(r,y)\right\rvert}{e^{3.5c(T-r)}(d^c+\lVert y\rVert^2)^\frac{1}{2}}\right]
e^{3.5c(T-t)}
 (d^c+\lVert x\rVert^2)^\frac{1}{2}\,dr.
\end{split}\end{align}
 Dividing by $e^{3.5c(T-t)}
 (d^c+\lVert x\rVert^2)^\frac{1}{2}$ we then obtain that
  for all 
$d,K\in \N$, $t\in [0,T]$,
$\varepsilon\in (0,1)$ we have
that
\begin{align}
 \sup_{y\in\R^d}\frac{\left\lvert u^{d,K,\varepsilon}(t,y)\right\rvert}{e^{3.5c(T-r)}(d^c+\lVert y\rVert^2)^\frac{1}{2}}\leq 
3
(cd^c T^{-1})^{\frac{1}{2}}e^{3.5cT}
+
\int_{t}^{T}c \sup_{y\in\R^d}\frac{\left\lvert u^{d,K,\varepsilon}(r,y)\right\rvert}{e^{3.5c(T-r)}(d^c+\lVert y\rVert^2)^\frac{1}{2}}\,dr.
\label{x24}
\end{align}
Hence, \eqref{growth u} and Gr\"onwall's inequality imply
for all 
$d,K\in \N$, $t\in [0,T]$
$\varepsilon\in (0,1)$ that
\begin{align}
 \sup_{y\in\R^d}\frac{\left\lvert u^{d,K,\varepsilon}(t,y)\right\rvert}{e^{3.5c(T-r)}(d^c+\lVert y\rVert^2)^\frac{1}{2}}\leq \xeqref{x24}3
(cd^c T^{-1})^{\frac{1}{2}}e^{3.5cT} e^{cT}.
\end{align}
This shows 
for all 
$d,K\in \N$, $t\in [0,T]$, $x\in\R^d$,
$\varepsilon\in (0,1)$
that
\begin{align}\left\lvert u^{d,K,\varepsilon}(t,x)\right\rvert\leq 3
(cd^c T^{-1})^{\frac{1}{2}}e^{8cT}(d^c+\lVert x\rVert^2)^\frac{1}{2}.\label{k04b}
\end{align}
Next, 
\eqref{growth f epsilon} and \eqref{k90b} demonstrate 
for all $d\in \N$,  $x\in\R^d$,
$\varepsilon\in (0,1)$
that
\begin{align}
\frac{T\lvert f_\varepsilon(0)\rvert+ \lvert g^d_\varepsilon(x) \rvert}{(d^c+\lVert x\rVert^2)^\frac{1}{2}}
\leq 
\frac{T\cdot\xeqref{growth f epsilon} T^{-\frac{3}{2}}c^{\frac{1}{2}}d^\frac{c}{2}+ \xeqref{k90b}2
(cd^c T^{-1})^{\frac{1}{2}} (d^c+\lVert x\rVert^2)^\frac{1}{2}}{(d^c+\lVert x\rVert^2)^\frac{1}{2}}\leq 3T^{-\frac{1}{2}}
(cd^c)^\frac{1}{2}\label{x06b}
\end{align}
Hence,  \cite[Corollary 3.12]{HJKN2020}
(applied for 
all $d,K\in\N $, $\varepsilon\in (0,1)$
with 
$f\gets f_\varepsilon$,
$g\gets g^d_\varepsilon$,
$\varphi\gets ( \R^d \ni x\mapsto (d^c+\lVert x\rVert^2)^\frac{1}{2})\in (0,\infty)$,
$
(
Y^\theta_{\cdot,\cdot}(\cdot))_{\theta\in\Theta}\gets
( X^{d,\theta,K,\varepsilon,\cdot}_{\cdot,\cdot})_{\theta\in\Theta}
$, $(U^\theta_{n,m})_{\theta\in\Theta,n,m\in\Z}\gets (U^{d,\theta, K,\varepsilon}_{n,m})_{\theta\in\Theta,n,m\in\Z}  $ in the notation of \cite[Corollary 3.12]{HJKN2020}), \eqref{k98}, and \eqref{k04b} 
show
for all
$d,K,n,m\in\N $, $\varepsilon\in (0,1)$, $\theta\in\Theta$
that
\begin{align} \begin{split} 
&
\sup_{t\in [0,T],x\in \R^d}
\frac{\E\!\left[\left(
\left\lvert
U_{n,m}^{d,\theta,K,\varepsilon}(t,x)-u^{d,K,\varepsilon}(t,x)\right\rvert^2\right]\right)^\frac{1}{2}}{\left(d^c+\lVert x\rVert^2\right)^\frac{1}{2}}
\\
&\leq 2
e^\frac{m}{2}m^{-\frac{n}{2}}(1+2Tc)^{N-1}e^{3.5c T}\left(
\sup_{x\in\R^d}
\frac{T\lvert f_\varepsilon(0)\rvert+ \lvert g^d_\varepsilon(x) \rvert}{(d^c+\lVert x\rVert^2)^\frac{1}{2}}+Tc \sup_{t\in[0,T],x\in\R^d} \frac{\left\lvert u^{d,K,\varepsilon}(t,x)\right\rvert}{(d^c+\lVert x\rVert^2)^\frac{1}{2}}
\right)\\
&\leq 2
e^\frac{m}{2}m^{-\frac{n}{2}}(e^{2cT})^{N-1}e^{3.5c T}\left( \xeqref{x06b}3T^{-\frac{1}{2}}
(cd^c)^\frac{1}{2}+Tc\cdot \xeqref{k04b}3
(cd^c T^{-1})^{\frac{1}{2}}e^{8cT}
\right)\\
&\leq 
2
e^\frac{m}{2}m^{-\frac{n}{2}}(e^{2cT})^{N-1}e^{3.5c T}(1+Tc)\cdot 3
(cd^c T^{-1})^{\frac{1}{2}}e^{8cT}\\
&\leq 
6
e^\frac{m}{2}m^{-\frac{n}{2}}(e^{2cT})^{N-1}e^{12 c T}
(cd^c T^{-1})^{\frac{1}{2}}\\&\leq 
6
e^\frac{m}{2}m^{-\frac{n}{2}}e^{12 c Tn}
(cd^c T^{-1})^{\frac{1}{2}}.
\end{split}\end{align}
This implies 
for all $\theta\in\Theta$,
$d,K,n,m\in\N $, $\varepsilon\in (0,1)$, $t\in[0,T]$, $x\in\R^d$
that
\begin{align}
\E\!\left[\left(
\left\lvert
U_{n,m}^{d,\theta,K,\varepsilon}(t,x)-u^{d,K,\varepsilon}(t,x)\right\rvert^2\right]\right)^\frac{1}{2}\leq 6
e^\frac{m}{2}m^{-\frac{n}{2}}e^{12 c Tn}
(cd^c T^{-1})^{\frac{1}{2}}\left(d^c+\lVert x\rVert^2\right)^\frac{1}{2}
.
\end{align}
This completes the proof of \cref{a31}.
\end{proof}

\section{DNNs}\label{s04}

\subsection{Properties of operations associated to DNNs} In \cref{m07b} below we introduce  operations which are important for constructing the random DNN that represents the MLP approximations
in the proof of \cref{k04}.
\begin{setting}\label{m07b}
Assume \cref{m07}.
Let $\mathfrak{n}_n^d
\in \mathbf{D} $, $n\in [3,\infty)\cap\Z$, $d\in \N$,  satisfy for all $n\in [3,\infty)\cap\N$, $d\in \N$ that
\begin{align}
\label{k07}
\mathfrak{n}_n^d= (d,\underbrace{2d,\ldots,2d}_{(n-2)\text{ times}},d)\in \N^n
.
\end{align} 
 Let $\mathfrak{n}_n\in \mathbf{D}$, $n\in [3,\infty)$, satisfy for all
$n\in [3,\infty)$ that $\mathfrak{n}_n=\mathfrak{n}^1_n$.
Let $\boxplus \colon \mathbf{D}\times \mathbf{D} \to\mathbf{D}  $ satisfy
for all $H\in \N$, 
$\alpha= (\alpha_0,\alpha_1,\ldots,\alpha_{H},\alpha_{H+1})\in \N^{H+2}$,
$\beta= (\beta_0,\beta_1,\beta_2,\ldots,\beta_{H},\beta_{H+1})\in \N^{H+2}$
that
$
\alpha \boxplus \beta =(\alpha_0,\alpha_1+\beta_1,\ldots,\alpha_{H}+\beta_{H},\beta_{H+1})\in \N^{H+2}.
$
Let $\odot \colon \mathbf{D}\times \mathbf{D} \to\mathbf{D} $ satisfy 
for all $H_1,H_2\in \N$, $ \alpha=(\alpha_0,\alpha_1,\ldots,\alpha_{H_1},\alpha_{H_1+1})\in\N^{H_1+2}$, $\beta=(\beta_0,\beta_1,\ldots,\beta_{H_2},\beta_{H_2+1})\in\N^{H_2+2}$
that
$
\alpha\odot \beta= (\beta_{0},\beta_{1},\ldots,\beta_{H_2},\beta_{H_2+1}+\alpha_{0},\alpha_{1},\alpha_{2},\ldots,\alpha_{H_1+1})\in \N^{H_1+H_2+3}.
$
\end{setting}

To prove our main result in this section presented in \cref{k04} we employ several results presented in Lemmas \ref{k43}--\ref{b04}, which are basic facts on DNNs. The proof of Lemmas \ref{k43}--\ref{b02} can be found in 
\cite{CHW2022,HJKN2020a} and therefore omitted.

\begin{lemma}[$\odot$ is associative--{\cite[Lemma 3.3]{HJKN2020a}}]\label{k43}Assume \cref{m07b} and let $\alpha,\beta,\gamma\in \bfD$. Then we have that
$(\alpha\odot\beta)\odot \gamma
= \alpha\odot(\beta\odot \gamma)$.
\end{lemma}

\begin{lemma}[$\boxplus$ and associativity--{\cite[Lemma 3.4]{HJKN2020a}}]Assume \cref{m07b},
let $H,k,l \in \N$, and let $\alpha,\beta,\gamma\in \left( \{k\}\times \N^{H} \times \{l\}\right)$.
Then
\begin{enumerate}[(i)]
\item we have that $\alpha\boxplus\beta\in \left(\{k\}\times \N^{H} \times \{l\}\right)$,
\item we have that $\beta\boxplus \gamma\in \left(\{k\}\times \N^{H} \times \{l\}\right)$, and 
\item we have that $(\alpha\boxplus\beta)\boxplus \gamma
= \alpha\boxplus(\beta\boxplus \gamma)$.
\end{enumerate}
\end{lemma}
 \begin{lemma}[Triangle inequality--{\cite[Lemma 3.5]{HJKN2020a}}]\label{b15}
Assume \cref{m07b},
let $k,l,H \in \N$, $\alpha,\beta\in \{k\}\times \N^{H} \times \{l\}$.
Then we have that
$\supnorm{\alpha\boxplus\beta}\leq\supnorm{\alpha}+
\supnorm{\beta} $.
\end{lemma}
\begin{lemma}[DNNs for affine transformations--{\cite[Lemma 3.7]{HJKN2020a}}]\label{p01}
Assume \cref{m07} and let $d,m\in \N$, 
$\lambda\in \R$,
$b\in\R^d$, $a\in\R^m$, $\Psi\in\mathbf{N}$ satisfy that $\mathcal{R}(\Psi)\in C(\R^d,\R^m)$. Then we have that
$
\lambda\left((\mathcal{R}(\Psi))(\cdot +b)+a\right)\in \mathcal{R}(\{\Phi\in\mathbf{N}\colon \mathcal{D}(\Phi)=\mathcal{D}(\Psi)\}).
$
\end{lemma}

\begin{lemma}[Composition of functions generated by DNNs--{\cite[Lemma 3.8]{HJKN2020a}}]\label{m11b}
Assume \cref{m07b} and let $d_1,d_2,d_3\in\N$, $f_1\in C(\R^{d_2},\R^{d_3})$, $f_2\in C(  \R^{d_1}, \R^{d_2}) $, 
$\alpha,\beta\in \mathbf{D}$ satisfy both that
$f_1\in \mathcal{R}(\{\Phi\in \mathbf{N}\colon \mathcal{D}(\Phi)=\alpha\})$
as well as
$f_2\in \mathcal{R}(\{\Phi\in \mathbf{N}\colon \mathcal{D}(\Phi)=\beta\})$.
Then we have
that $(f_1\circ f_2)\in \mathcal{R}(\{\Phi\in \mathbf{N}\colon \mathcal{D}(\Phi)=\alpha\odot\beta\})$.
\end{lemma}

\begin{lemma}[Sum of DNNs of the same length--{\cite[Lemma 3.9]{HJKN2020a}}]
\label{b01b}
Assume \cref{m07b} and let $p,q,M,H\in \N$,  $\alpha_1,\alpha_2,\ldots,\alpha_M\in\R$,
 $k_i\in \mathbf{D} $,
$g_i\in C(\R^{p},\R^{q})$,
$i\in [1,M]\cap\N$, satisfy 
for all $i\in [1,M]\cap\N$
that $ \dim(k_i)=H+2$ and
$g_i\in 
\mathcal{R}(\{\Phi\in\mathbf{N}\colon \mathcal{D}(\Phi)=k_i\}).
$
Then
we have that 
$
\sum_{i=1}^{M}\alpha_i g_i
\in\mathcal{R}\left(\left\{ \Phi\in\mathbf{N}\colon
\mathcal{D}(\Phi)=\boxplus_{i=1}^Mk_i\right\}\right).
$
\end{lemma}

\begin{lemma}[Existence of DNNs with $H$ hidden layers for $\mathrm{Id}_{\R^d}$--{\cite[Lemma 3.6]{CHW2022}}]\label{b03}
	Assume \cref{m07b} and let $d,H\in \N$.
	Then we have that
	$\mathrm{Id}_{\R^d}\in \mathcal{R}(\{\Phi\in\mathbf{N}\colon\mathcal{D}(\Phi)=\mathfrak{n}^d_{H+2} \}) $.
\end{lemma}

\begin{lemma}[{\cite[Lemma 3.7]{CHW2022}}]\label{b02}
Assume \cref{m07}, let $H,p,q\in \N$, and let $g\in C(\R^p,\R^q)$ satisfy that
$g\in \calR(\{\Phi\in \bfN\colon \dim(\calD(\Phi))=H+2\})$. Then for all
$n\in \N_0$ we have that
$g\in \calR(\{\Phi\in \bfN\colon \dim(\calD(\Phi))=H+2+n\})$.
\end{lemma}
\begin{lemma}\label{b04}Assume \cref{m07b}. Then for all $n\in \N $, $d_0,d_1,\ldots, d_n\in \N$, 
$f_1\in C(\R^{d_1},\R^{d_{0}}), f_2\in C(\R^{d_2},\R^{d_{1}}), \ldots ,f_n\in C(\R^{d_n},\R^{d_{n-1}})$, 
 $\phi_1,\phi_2,\ldots,\phi_n\in \bfN$ 
with 
 $\forall i\in [1,n]\cap\Z\colon f_i=\calR(\phi_i)$ we have that
\begin{align}
\supnorm{\operatorname*{\odot}_{i=1}^n\mathcal{D}(\phi_i)}\leq \max\left\{
\supnorm{\calD(\phi_1)},
\supnorm{\calD(\phi_2)},\ldots,\supnorm{\calD(\phi_n)}, 2d_1,2d_2,\ldots,2d_{n-1}
\right\}
\end{align}
\end{lemma}
\begin{proof}[Proof of \cref{b04}]
We will prove by induction on $n\in \N$. The base case $n=1$ is clear. 
For the induction step
$\N\in n-1\mapsto n\in \N$
 let $n\in \N\cap[2,\infty) $ satisfy that for all 
$d_0,d_1,\ldots, d_{n-1}\in \N$, 
$f_1\in C(\R^{d_1},\R^{d_{0}}), f_2\in C(\R^{d_2},\R^{d_{1}}), \ldots ,f_{n-1}\in C(\R^{d_{n-1}},\R^{d_{n-2}})$, 
 $\phi_1,\phi_2,\ldots,\phi_{n-1}\in \bfN$ 
with 
 $\forall i\in [1,n-1]\cap\Z\colon f_i=\calR(\phi_i)$ we have that
\begin{align}
\supnorm{\operatorname*{\odot}_{i=1}^{n-1}\mathcal{D}(\phi_i)}\leq\max\bigl\{
\supnorm{\calD(\phi_1)},
\supnorm{\calD(\phi_2)},\ldots,\supnorm{\calD(\phi_{n-1})}, 2d_1,2d_2,\ldots,2d_{n-2}
\bigr\}.\label{t01}
\end{align}
This implies that for all 
$d_0,d_1,\ldots, d_{n}\in \N$, 
$f_1\in C(\R^{d_1},\R^{d_{0}}), f_2\in C(\R^{d_2},\R^{d_{1}}), \ldots ,f_{n}\in C(\R^{d_{n}},\R^{d_{n-1}})$, 
 $\phi_1,\phi_2,\ldots,\phi_{n}\in \bfN$ 
with 
 $\forall i\in [1,n-1]\cap\Z\colon f_i=\calR(\phi_i)$ there exist 
$H_1,H_2\in \N$,
$\mathbf{a}_1\in \R^{H_1}$,
$\mathbf{a}_2\in \R^{H_2}
$ such that 
\begin{align}
C(\R^{d_{n-1}}, \R^{d_0}) \ni f_1\circ f_2\circ \ldots\circ f_{n-1}
\in 
\calR\left(\left\{\Phi\in \bfN\colon\calD(\Phi)=\operatorname*{\odot}_{i=1}^{n-1}\mathcal{D}(\phi_i)\right\}\right)  ,
\end{align}
$\operatorname*{\odot}_{i=1}^{n-1}\mathcal{D}(\phi_i)=(d_{n-1}, \mathbf{a}_1,d_{0})$, 
$\mathcal{D}(\phi_n)=( d_n,\mathbf{a}_2 ,d_{n-1} ) $, 
$
(\operatorname*{\odot}_{i=1}^{n-1}\mathcal{D}(\phi_i))\odot \calD(\phi_n)=(d_n,\mathbf{a}_2,2d_{n-1},\mathbf{a}_1,d_0)
 $, and
\begin{align} \begin{split} 
&
\supnorm{\operatorname*{\odot}_{i=1}^n\mathcal{D}(\phi_i)}=
\supnorm{\left[\operatorname*{\odot}_{i=1}^{n-1}\mathcal{D}(\phi_i)  
\right]\odot \calD(\phi_n)}\leq 
\max\left\{\supnorm{\operatorname*{\odot}_{i=1}^{n-1}\mathcal{D}(\phi_i)  },
\supnorm{\calD(\phi_n)}, 2d_{n-1}
\right\}\\
&\leq 
\max\Bigl\{\xeqref{t01}\max\bigl\{
\supnorm{\calD(\phi_1)},
\supnorm{\calD(\phi_2)},\ldots,\supnorm{\calD(\phi_{n-1})}, 2d_1,2d_2,\ldots,2d_{n-2}
\bigr\},
\supnorm{\calD(\phi_n)}, 2d_{n-1}
\Bigr\}\\
&
\leq \max\left\{
\supnorm{\calD(\phi_1)},
\supnorm{\calD(\phi_2)},\ldots,\supnorm{\calD(\phi_n)}, 2d_1,2d_2,\ldots,2d_{n-1}
\right\}.\end{split}
\end{align}
This proves the induction step. Induction hence completes the proof of \cref{b04}.
\end{proof}

\subsection{DNN representation of our Euler-Maruyama approximations}
In \cref{l01} below we prove that Euler-Maruyama approximations
\eqref{SDE X epsilon K} are DNNs if their coefficients are DNNs.

\begin{lemma}\label{l01}
 Assume Settings \ref{setting approx NN}, \ref{m07}, and  \ref{theta setting}. 
 For every $d\in \N$, $\varepsilon\in (0,1)$
let
$F_{\varepsilon}^d\colon\R^d\to \R^{d\times d}$, 
$G^d\colon \R^d\to \R^d$ be measurable and satisfy for all 
$y,z\in \R^d$
that
$\gamma_{\varepsilon}^d(y,z)=F_{\varepsilon}^d(y)G^d(z) $. 
 For every $d\in \N$, $\varepsilon\in (0,1)$,
$v\in\R^d$
 let
$\Phi_{\beta_\varepsilon^d},\Phi_{\sigma_\varepsilon^d,v}, 
\Phi_{F^d_\varepsilon,v} \in \bfN
$
satisfy that
$\beta_\varepsilon^d=\calR(\Phi_{\beta_\varepsilon^d}) $, 
$\sigma_\varepsilon^d (\cdot)v=
\calR(\Phi_{\sigma_\varepsilon^d,v})
$,
$
F^d_\varepsilon(\cdot)v=
\calR(\Phi_{F^d_\varepsilon,v} )$. 
Moreover, assume for all
 $d\in \N$, $\varepsilon\in (0,1]$,
$v\in\R^d$ that
$\calD(\Phi_{\sigma_\varepsilon^d,v})=\calD(\Phi_{\sigma_\varepsilon^d,0})$ and
$\calD(\Phi_{F_\varepsilon^d,v})=\calD(\Phi_{F_\varepsilon^d,0})$. 

Let $\omega\in \Omega$.
Then there exists
$(\calX^{d,\theta,K,\varepsilon,t}_{s})_{
  d\in\N  ,
 \theta\in \Theta ,
 \varepsilon\in (0,1) ,
  t\in [0,T) ,
 s\in (t,T] 
}\subseteq \bfN$ such that the following items are true.
\begin{enumerate}[(i)]
\item For all $ d\in\N $,
$\theta\in \Theta$,
$\varepsilon\in (0,1)$,
 $t\in [0,T)$,
$s\in (t,T]$, $x\in \R^d$ we have that
$\calR(\calX^{d,\theta,K,\varepsilon,t}_{s})\in C(\R^d,\R^d)$
and
$(\calR(\calX^{d,\theta,K,\varepsilon,t}_{s}))(x)
=X^{d,\theta,K,\varepsilon,t,x}_{s}(\omega)$.
\item For all $ d\in\N $,
$\theta\in \Theta$,
$\varepsilon\in (0,1)$,
 $t_1\in [0,T)$,
$s_1\in (t_1,T]$, 
 $t_2\in [0,T)$,
$s_2\in (t_2,T]$, 
$x\in \R^d$ we have that
$\calD(\calX^{d,\theta_1,K,\varepsilon,t_1}_{s_1})=
\calD(
\calX^{d,\theta_2,K,\varepsilon,t_2}_{s_2})
$.
\item For all $ d\in\N $,
$\theta\in \Theta$,
$\varepsilon\in (0,1)$,
 $t\in [0,T)$,
$s\in (t,T]$ we have that
$
\dim(\calD(\calX^{d,\theta,K,\varepsilon,t}_{s}))=K(
\max\{\dim(\calD(\Phi_{\beta^d_\varepsilon})) ,\dim(\calD(\Phi_{\sigma^d_\varepsilon,0})),
\dim(\calD(\Phi_{F^d_\varepsilon,0}))\}
-1)+1
$. 
\item For all $ d\in\N $,
$\theta\in \Theta$,
$\varepsilon\in (0,1)$,
 $t\in [0,T)$,
$s\in (t,T]$ we have that
$ \supnorm{\calD(\calX^{d,\theta,K,\varepsilon,t}_{s})}\leq 2d+\supnorm{\calD(\Phi_{\beta^d_\varepsilon})}
+
\supnorm{
\calD(\Phi_{\sigma^d_\varepsilon,0})}
+
\supnorm{
\calD(\Phi_{F^d_\varepsilon,0})}
$.
\end{enumerate}

\end{lemma}
\begin{proof}[Proof of \cref{l01}]Throughout this proof let the notation in \cref{m07b} be given.
First, observe that
for all 
$ d\in\N $,
$\theta\in \Theta$,
$x\in\R^d$,
$\varepsilon\in (0,1)$,
$k\in [1,K]\cap\Z$, $t\in [0,T)$,
$s\in [\frac{kT}{K},\frac{(k+1)T}{K}]$ we have that
\begin{align} \label{d01}\begin{split} 
X^{d,\theta,K,\varepsilon,t,x}_{s}(\omega)
&=X^{d,\theta,K,\varepsilon,t,x}_{\max\{t,\frac{kT}{K}\}}(\omega)
+\beta^d_\varepsilon\left(X^{d,\theta,K,\varepsilon,t,x}_{\max\{t,\frac{kT}{K}\}}(\omega)\right)\left(s-\max\{t,\tfrac{kT}{K}\}\right)\\&\quad+
\sigma^d_\varepsilon\left(X^{d,\theta,K,\varepsilon,t,x}_{\max\{t,\frac{kT}{K}\}}(\omega)\right) \left(W^{d,\theta}_s(\omega)-
W^{d,\theta}_{\max\{t,\frac{kT}{K}\}}(\omega)\right)\\
&\quad
+
F^d_\varepsilon\left(X^{d,\theta,K,\varepsilon,t,x}_{\max\{t,\frac{kT}{K}\}}(\omega)\right)\int_{\max\{t,\frac{kT}{K}\}}^{s}\int_{\R^d\setminus\{0\}}G^d(z)\tilde{N}^{d,\theta}(\omega)(dz,du).\end{split}
\end{align}
Next, for every 
$ d\in\N $,
$\theta\in \Theta$,
$x\in\R^d$,
$\varepsilon\in (0,1)$,
$k\in [1,K]\cap\Z$, $t\in [0,T)$,
$s\in (t,T]$
let $ J_k(s)\in \R$,
$
\phi^{d,\theta,K,\varepsilon}_{t,s,k}(x)\in \R^d$
 satisfy that
\begin{align} \begin{split} 
J_k(s)&=\max\{ t,\tfrac{(k-1)T}{K} \}\1_{ [0,\max\{ t,\frac{(k-1)T}{K} \}    ] }(s)\\&\quad 
+s\1_{ (\max\{ t,\frac{(k-1)T}{K} \},  \max\{ t,\frac{kT}{K} \} 
 ] }(s)+ \max\{ t,\tfrac{kT}{K} \}
\1_{ (  \max\{ t,\frac{kT}{K} \},T ]}(s),\end{split}
\end{align} 
and
\begin{align} \label{d02}\begin{split} 
\phi^{d,\theta,K,\varepsilon}_{t,s,k}(x)
&=x
+\beta^d_\varepsilon(x)\left(J_k(s)-\max\{t,\tfrac{(k-1)T}{K}\}\right)+
\sigma^d_\varepsilon(x)\left(W^{d,\theta}_{J_k(s)}(\omega)-
W^{d,\theta}_{\max\{t,\frac{(k-1)T}{K}\}}(\omega)\right)\\
&\quad
+
F^d_\varepsilon(x)\int_{\max\{t,\frac{(k-1)T}{K}\}}^{s}\int_{\R^d\setminus\{0\}}G^d(z)\tilde{N}^{d,\theta}(\omega)(dz,du).\end{split}
\end{align}
Furthermore, for every 
$ d\in\N $,
$\theta\in \Theta$,
$\varepsilon\in (0,1)$,
$k\in [1,K]\cap\Z$, $t\in [0,T)$,
$s\in (t,T]$ let 
\begin{align}\label{d03}
\psi^{d,\theta,K,\varepsilon}_{t,s,k}=
\phi^{d,\theta,K,\varepsilon}_{t,s,k}
\circ
\phi^{d,\theta,K,\varepsilon}_{t,s,k-1}\circ\ldots
\circ
\phi^{d,\theta,K,\varepsilon}_{t,s,1}
.
\end{align}
Note that for all 
$ d\in\N $,
$\theta\in \Theta$,
$\varepsilon\in (0,1)$, 
$k\in [1,K-1]\cap\Z$, $s\in [0,\max\{t,\frac{(k-1)T}{K}\}]$ we have that $\phi^{d,\theta,K,\varepsilon}_{t,s,k}=\mathrm{Id}_{\R^d} $. This ensures for all
$ d\in\N $,
$\theta\in \Theta$,
$\varepsilon\in (0,1)$, 
 $k\in [1,K-1]\cap\Z$, $n\in [k+1,K]\cap\Z$,
$s\in [0,\max\{t,\frac{kT}{K}\}]$ that 
$\psi^{d,\theta,K,\varepsilon}_{t,s,k}=
\psi^{d,\theta,K,\varepsilon}_{t,s,n}$ and in particular 
$\psi^{d,\theta,K,\varepsilon}_{t,s,k}=
\psi^{d,\theta,K,\varepsilon}_{t,s,K}$.
Observe that for all 
$ d\in\N $,
$\theta\in \Theta$,
$\varepsilon\in (0,1)$, 
$k\in [1,K]\cap\Z$, $s\in [0,\max\{t,\frac{kT}{K}\}]$, $x\in \R^d$ that 
$\psi^{d,\theta,K,\varepsilon}_{t,s,k}(x)=X^{d,\theta,K,\varepsilon,t,x}_{s}(\omega)$.
Therefore, for all
$ d\in\N $,
$\theta\in \Theta$,
$\varepsilon\in (0,1)$,
$s\in[0,T]$, $x\in \R^d$ we have that
$\psi^{d,\theta,K,\varepsilon}_{t,s,K}(x)=X^{d,\theta,K,\varepsilon,t,x}_{s}$, i.e.,
\begin{align}\label{k03}
X^{d,\theta,K,\varepsilon,t,x}_{s}(\omega)=
\phi^{d,\theta,K,\varepsilon}_{t,s,K}
\circ
\phi^{d,\theta,K,\varepsilon}_{t,s,K-1}\circ\ldots
\circ
\phi^{d,\theta,K,\varepsilon}_{t,s,1} (x).
\end{align}
Next, assume w.l.o.g.\ (cf. \cref{b02}) that \begin{align}\label{k02}
\dim(\calD(\Phi_{\beta^d_\varepsilon})) =\dim(\calD(\Phi_{\sigma^d_\varepsilon,0}))=
\dim(\calD(\Phi_{F^d_\varepsilon,0})).
\end{align}
\cref{b03} shows for all $d\in \N$, $\varepsilon\in (0,\infty)$ that
\begin{align}
\mathrm{Id}_{\R^d}\in \calR(\{\Phi\in \bfN\colon \calD(\Phi)=\mathfrak{n}^d_{\dim (\Phi_{\beta^d_\varepsilon}) }\}).
\end{align}
Therefore, \eqref{d02},
\eqref{k02}, and \cref{b01b}
imply for all $ d\in\N $,
$\theta\in \Theta$,
$\delta,\varepsilon\in (0,1)$,
$k\in [1,K]\cap\Z$, $t\in [0,T)$,
$s\in (t,T]$ that
\begin{align}
\phi^{d,\theta,K,\varepsilon}_{t,s,k}(\cdot)\in \calR\left(\left\{\Phi\in \bfN\colon 
\calD(\Phi)=
\mathfrak{n}^d_{\dim(\Phi_{\beta^d_\varepsilon})  }\boxplus \calD(\Phi_{\beta^d_\varepsilon})
\boxplus \calD(\Phi_{\sigma^d_\varepsilon,0})
\boxplus
\calD(\Phi_{F^d_\varepsilon,0})
\right\}\right).
\end{align}
This, \eqref{k03}, and \cref{m11b} ensure that there exists
$(\calX^{d,\theta,K,\varepsilon,t}_{s})_{
  d\in\N  ,
 \theta\in \Theta ,
 \varepsilon\in (0,1) ,
  t\in [0,T) ,
 s\in (t,T] 
}\subseteq \bfN$
such that
for all
$ d\in\N $,
$\theta\in \Theta$,
$\varepsilon\in (0,1)$, 
 $t\in [0,T)$,
$s\in (t,T]$, $x\in \R^d$ we have that
\begin{align} \begin{split} 
&
\calD(\calX^{d,\theta,K,\varepsilon,t}_{s})=
\operatorname*{\odot}_{k=1}^K\left[
\mathfrak{n}^d_{\dim(\Phi_{\beta^d_\varepsilon})  }\boxplus \calD(\Phi_{\beta^d_\varepsilon})
\boxplus \calD(\Phi_{\sigma^d_\varepsilon,0})
\boxplus
\calD(\Phi_{F^d_\varepsilon,0})\right],\\
& (\calR(\calX^{d,\theta,K,\varepsilon,t}_{s}))(x)
=X^{d,\theta,K,\varepsilon,t,x}_{s}(\omega).
\label{k08}\end{split}
\end{align}
Thus, the definition of $\odot$ and an induction argument show that for all
$ d\in\N $,
$\theta\in \Theta$,
$\varepsilon\in (0,1)$,
 $t\in [0,T)$,
$s\in (t,T]$, $x\in \R^d$ we have that
\begin{align}
\dim(\calD(\calX^{d,\theta,K,\varepsilon,t}_{s}))=K(\dim(\calD(\Phi_{\beta^d_\varepsilon}))-1)+1.
\end{align}
Next, \eqref{k08}, \cref{b04}, the triangle inequality (cf. \cref{b15}), and \eqref{k07}
imply that
  for all
$ d\in\N $,
$\theta\in \Theta$,
$\varepsilon\in (0,1)$,
 $t\in [0,T)$,
$s\in (t,T]$, $x\in \R^d$ we have that
\begin{align} \begin{split} \supnorm{\calD(\calX^{d,\theta,K,\varepsilon,t}_{s})}&=
\supnorm{\operatorname*{\odot}_{k=1}^K\left[
\mathfrak{n}^d_{\dim(\Phi_{\beta^d_\varepsilon})  }\boxplus \calD(\Phi_{\beta^d_\varepsilon})
\boxplus \calD(\Phi_{\sigma^d_\varepsilon,0})
\boxplus
\calD(\Phi_{F^d_\varepsilon,0})\right]
}\\
&\leq\max\left\{2d,
\supnorm{
\mathfrak{n}^d_{\dim(\Phi_{\beta^d_\varepsilon})  }\boxplus \calD(\Phi_{\beta^d_\varepsilon})
\boxplus \calD(\Phi_{\sigma^d_\varepsilon,0})
\boxplus
\calD(\Phi_{F^d_\varepsilon,0})}
\right\}
\\
&\leq \max\left\{2d,
\supnorm{
\mathfrak{n}^d_{\dim(\Phi_{\beta^d_\varepsilon})  }}+ \supnorm{\calD(\Phi_{\beta^d_\varepsilon})}
+
\supnorm{
\calD(\Phi_{\sigma^d_\varepsilon,0})}
+
\supnorm{
\calD(\Phi_{F^d_\varepsilon,0})}
\right\}\\
&=2d+\supnorm{\calD(\Phi_{\beta^d_\varepsilon})}
+
\supnorm{
\calD(\Phi_{\sigma^d_\varepsilon,0})}
+
\supnorm{
\calD(\Phi_{F^d_\varepsilon,0})}.	
\end{split}
\end{align}
The proof of  \cref{l01} is thus completed.
\end{proof}

\subsection{DNN representation of our MLP approximations}
In 
\cref{k04} below we prove that the MLP approximations under consideration are DNNs.
\begin{proposition}
\label{k04}
Assume the setting of \cref{l01}. For every
$d\in \N$, $\varepsilon\in (0,1)$
let $f_\varepsilon\in C(\R,\R)$,
$\Phi_{f_\varepsilon}, \Phi_{g^d_\varepsilon}\in \bfN$
satisfy that
$ \calR(\Phi_{f_\varepsilon})=f_\varepsilon $ and
$ \calR(\Phi_{g^d_\varepsilon})=g^d_\varepsilon $.
For every $d\in \N$, $\varepsilon\in (0,1)$
let 
$U^{d,\theta,K,\varepsilon}_{n,m}\colon [0,T]\times\R^d\times \Omega\to \R$, 
$\theta\in \Theta$,  
$n,m\in \Z$, 
satisfy for all 
$\theta\in \Theta$,  
$n\in \N_0$, $m\in \N$, $t\in [0,T]$, $x\in \R^d$ that
\begin{align} \begin{split} &U^{d,\theta,K,\varepsilon}_{n,m}(t,x)=
\frac{\1_{\N}(n)}{m^n}\sum_{i=1}^{m^n}g^d_\varepsilon\Bigl( X^{ d,(\theta,0,-i), K,\varepsilon,t,x}_{T} \Bigr)\\
&\quad
+\sum_{\ell=0}^{n-1}
\frac{(T-t)}{m^{n-\ell}}
\sum_{i=1}^{m^{n-\ell}}\Bigl(
f_\varepsilon\circ
U^{d,(\theta,\ell,i),K,\varepsilon}_{\ell,m}
-\1_{\N}(\ell)
f_\varepsilon\circ
U^{d,(\theta,-\ell,i),K,\varepsilon}_{\ell-1,m}\Bigr)
\Bigl(\mathfrak{T}_t^{(\theta,\ell,i)}, 
 X^{ d,(\theta,\ell,i), K,\varepsilon,t,x}_{\mathfrak{T}_t^{(\theta,\ell,i)}} 
\Bigr).\end{split}
\end{align}
Let
$(c_{d,\varepsilon})_{d\in\N,\varepsilon\in (0,1)}\subseteq \R $ satisfy 
for all 
$d\in\N$, $\varepsilon\in (0,1)$ 
that
\begin{align}
c_{d,\varepsilon} \geq 
2d+\supnorm{\calD(\Phi_{f_{\varepsilon}})}
+\supnorm{\calD(\Phi_{g^d_{\varepsilon}})}
+\supnorm{\calD(\Phi_{\beta^d_{\varepsilon}})}
+\supnorm{\calD(\Phi_{\sigma^d_{\varepsilon},0})}
+\supnorm{\calD(\Phi_{F^d_{\varepsilon},0})}
.\label{c01}
\end{align}
Then for all
$m\in \N$,
$n\in \N_0$, 
$d\in \N$,
$\varepsilon\in (0,1)$ there exists
$ (\Phi^{d,\theta,K,\varepsilon}_{n,m,t} )_{t\in[0,T],\theta\in \Theta}\subseteq \bfN$ such that the following items are true.
\begin{enumerate}[(i)]
\item\label{h01} We have for all $t_1,t_2\in [0,T]$, $\theta_1,\theta_2\in \Theta $
that $\calD(\Phi^{d,\theta_1,K,\varepsilon}_{n,m,t_1})
=\calD(\Phi^{d,\theta_2,K,\varepsilon}_{n,m,t_2})
$.
\item\label{h02} We have for all
$t\in [0,T]$, $\theta\in \Theta$ that
\begin{align} \begin{split} 
\dim(\calD( 
\Phi^{d,\theta,K,\varepsilon}_{n,m,t}
 ))
&=(n+1)
\left[
K \left( \max\left\{\dim(\calD(\Phi_{\beta^d_\varepsilon  })),
\dim(\calD(\Phi_{\sigma^d_\varepsilon  }))
 \right\} -1\right)+1\right]\\&\quad
+n(\dim(\calD (\Phi_{f_\varepsilon} )  ) -2)
+\dim(\calD (\Phi_{g^d_\varepsilon} )  ) -1.
\end{split}\end{align}
\item\label{h03} We have for all
$t\in [0,T]$, $\theta\in \Theta$ that
$\supnorm{\calD(\Phi^{d,\theta,K,\varepsilon}_{n,m,t})}\leq c_{d,\varepsilon}(3m)^n$.
\item \label{h04} We have for all
$t\in [0,T]$, $\theta\in \Theta$, $x\in \R^d$ that
$U^{d,\theta,K,\varepsilon}_{n,m}(t,x,\omega)=(\calR(\Phi^{d,\theta,K,\varepsilon}_{n,m,t}))(x)$.
\end{enumerate}
\end{proposition}

\begin{proof}[Proof of \cref{k04}]Throughout this proof 
	we use the notation introduced in \cref{m07b}.
First, \cref{l01} and \cref{c01} show that
 there exists
$(\calX^{d,\theta,K,\varepsilon,t}_{s})_{
  d\in\N  ,
 \theta\in \Theta ,
 \varepsilon\in (0,1) ,
  t\in [0,T) ,
 s\in (t,T] 
}\subseteq \bfN$ such that the following items hold.
\begin{enumerate}[(a)]
\item For all $ d\in\N $,
$\theta\in \Theta$,
$\varepsilon\in (0,1)$,
 $t\in [0,T)$,
$s\in (t,T]$, $x\in \R^d$ we have that
$\calR(\calX^{d,\theta,K,\varepsilon,t}_{s})\in C(\R^d,\R^d)$
and
\begin{align}\label{h06}
(\calR(\calX^{d,\theta,K,\varepsilon,t}_{s}))(x)
=X^{d,\theta,K,\varepsilon,t,x}_{s}(\omega).
\end{align}
\item For all $ d\in\N $,
$\theta\in \Theta$,
$\varepsilon\in (0,1)$,
 $t_1\in [0,T)$,
$s_1\in (t_1,T]$, 
 $t_2\in [0,T)$,
$s_2\in (t_2,T]$, 
$x\in \R^d$ we have that
\begin{align}
\label{h07}
\calD(\calX^{d,\theta_1,K,\varepsilon,t_1}_{s_1})=\calD(
\calX^{d,\theta_2,K,\varepsilon,t_2}_{s_2}).
\end{align}
\item For all $ d\in\N $,
$\theta\in \Theta$,
$\varepsilon\in (0,1)$,
 $t\in [0,T)$,
$s\in (t,T]$ we have that
\begin{align}
\label{h08}
\dim(\calD(\calX^{d,\theta,K,\varepsilon,t}_{s}))=K(
\max\{\dim(\calD(\Phi_{\beta^d_\varepsilon})) ,\dim(\calD(\Phi_{\sigma^d_\varepsilon,0})),
\dim(\calD(\Phi_{F^d_\varepsilon,0}))\}
-1)+1
.
\end{align} 
\item For all $ d\in\N $,
$\theta\in \Theta$,
$\varepsilon\in (0,1)$,
 $t\in [0,T)$,
$s\in (t,T]$ we have that
\begin{align}
\begin{split}  \supnorm{\calD(\calX^{d,\theta,K,\varepsilon,t}_{s})}\leq 2d+\supnorm{\calD(\Phi_{\beta^d_\varepsilon})}
+
\supnorm{
\calD(\Phi_{\sigma^d_\varepsilon,0})}
+
\supnorm{
\calD(\Phi_{F^d_\varepsilon,0})}\leq\xeqref{c01} c_{d,\varepsilon}
.\end{split}\label{h09} 
\end{align}
\end{enumerate}
Throughout the rest of this proof let $m\in \N$, 
$d\in \N$,
$\varepsilon\in (0,1)$ be fixed, let $\mathrm{Id}_\R$ be the identity on $\R$, and let 
${L}\in \Z$ satisfy that
\begin{align}
{L}=\dim(\calD(\calX^{d,\theta,K,\varepsilon,t}_{s}))\xeqref{h08}=
K \left( \max\left\{\dim(\calD(\Phi_{\beta^d_\varepsilon  })),
\dim(\calD(\Phi_{\sigma^d_\varepsilon  }))
 \right\} -1\right)+1
 \label{h10}
\end{align}
(cf. \eqref{h08}).
We will prove the result via induction on $n\in \N_0$.
First, the base case $n=0$ is true since the $0$-function can be represented by DNN function of arbitrary length. For the induction step 
$\N_0\ni  n\mapsto n+1\in \N $
let 
$n\in \N_0$ and assume that there exists
$ (\Phi^{d,\theta,K,\varepsilon}_{\ell,m,t} )_{t\in[0,T],\theta\in \Theta}\subseteq \bfN$, $\ell\in [0,n]\cap\Z$, such that the following items are true.
\begin{enumerate}[(a)]
\item We have for all $t_1,t_2\in [0,T]$, $\theta_1,\theta_2\in \Theta $,
$\ell\in [0,n]\cap\Z$
that \begin{align}\label{h01b}
\calD(\Phi^{d,\theta_1,K,\varepsilon}_{\ell,m,t_1})
=\calD(\Phi^{d,\theta_2,K,\varepsilon}_{\ell,m,t_2}).
\end{align}
\item We have for all
$t\in [0,T]$, $\theta\in \Theta$,
$\ell\in [0,n]\cap\Z$ that
\begin{align}\label{h02b}
\dim(\calD( 
\Phi^{d,\theta,K,\varepsilon}_{\ell,m,t}
 ))
=(\ell+1)
{L}
+\ell(\dim(\calD (\Phi_{f_\varepsilon} )  ) -2)
+\dim(\calD (\Phi_{g^d_\varepsilon} )  )-1.
\end{align}
\item We have for all
$t\in [0,T]$, $\theta\in \Theta$, $\ell\in [0,n]\cap\Z$ that
\begin{align}
\label{h03b} \supnorm{\Phi^{d,\theta,K,\varepsilon}_{\ell,m,t}}\leq c_{d,\varepsilon}(3m)^\ell.
\end{align}
\item We have for all
$t\in [0,T]$, $\theta\in \Theta$, $x\in \R^d$, 
$\ell\in [0,n]\cap\Z$ that
\begin{align}\label{h04b}
U^{d,\theta,K,\varepsilon}_{\ell,m}(t,x,\omega)=(\calR(\Phi^{d,\theta,K,\varepsilon}_{\ell,m,t}))(x).
\end{align}
\end{enumerate}
Next,
\cref{b03}, the fact that
$g^d_\varepsilon=\calR(\Phi_{g^d_{\varepsilon}})$,
\eqref{h06}, \eqref{h07}, and \cref{m11b}
show for all
$\theta\in \Theta$, $i\in [1,m^{n+1}]\cap\Z$,
$t\in [0,T]$ that
\begin{align} \label{k09}\begin{split} 
&
g^d_\varepsilon\Bigl( X^{ d,(\theta,0,-i), K,\varepsilon,t,\cdot }_{T} (\omega)\Bigr)
=\mathrm{Id}_\R \left(g^d_\varepsilon\Bigl( X^{ d,(\theta,0,-i), K,\varepsilon,t,\cdot }_{T} (\omega)\Bigr)\right)
\\
&
\in \calR
\left(\left\{
\Phi\in\bfN\colon \calD(\Phi)=
\mathfrak{n}_{(n+1)\bigl(\dim(\calD(\Phi_{f_\varepsilon}))-2+{L}\bigr)+1}
\odot\calD (\Phi_{g^d_\varepsilon})
\odot\calD(\calX^{d,0,K,\varepsilon,0}_{T})
\right\}\right).
\end{split}\end{align}
In addition, 
the definition of $\odot$, 
\eqref{k07}, and \eqref{h10} demonstrate that
\begin{align} \begin{split} 
&\dim \left(\mathfrak{n}_{(n+1)\bigl(\dim(\calD(\Phi_{f_\varepsilon}))-2+{L}\bigr)+1}
\odot\calD (\Phi_{g^d_\varepsilon})
\odot\calD(\calX^{d,0,K,\varepsilon,0}_{T})
\right)\\
&
=
\dim\left(
\mathfrak{n}_{(n+1)\bigl(\dim(\calD(\Phi_{f_\varepsilon}))-2+{L}\bigr)+1}
\right)
+
\dim\left(\calD (\Phi_{g^d_\varepsilon})\right)
+\dim\left(\calD(\calX^{d,0,K,\varepsilon,0}_{T})\right)
-2\\
&=\xeqref{k07}
(n+1)\bigl(\dim(\calD(\Phi_{f_\varepsilon}))-2+{L}\bigr)+1
+
\dim\left(\calD (\Phi_{g^d_\varepsilon})\right)
+\xeqref{h10}{L}-2\\
&=(n+1)\bigl(\dim(\calD(\Phi_{f_\varepsilon}))-2\bigr)+(n+2){L}+
\dim\left(\calD (\Phi_{g^d_\varepsilon})\right)-1.
\end{split}\end{align}
Moreover,
\cref{b03}, the fact that
$f_\varepsilon=\calR(\Phi_{f_{\varepsilon}})$,
\eqref{h01b}, \eqref{h04b},
\eqref{h06}, \eqref{h07}, and \cref{m11b}
ensure for all
$\theta\in \Theta$, $i\in [1,m^{n+1}]\cap\Z$,
$t\in [0,T]$ that
\begin{align} \begin{split} 
&\Bigl(
f_\varepsilon\circ
U^{d,(\theta,n,i),K,\varepsilon}_{n,m}
\Bigr)
\Bigl(\mathfrak{T}_t^{(\theta,\ell,i)}(\omega), 
 X^{ d,(\theta,\ell,i), K,\varepsilon,t,\cdot }_{\mathfrak{T}_t^{(\theta,\ell,i)}(\omega)} (\omega)
\Bigr)\\
&\in 
\calR
\left(\left\{
\Phi\in\bfN\colon \calD(\Phi)=
\calD(\Phi_{f_\varepsilon}) \odot 
\calD(\Phi^{d,0,K,\varepsilon}_{n,m,0})
\odot
\calD(\calX^{d,0,K,\varepsilon,0}_{T})
\right\}\right).\end{split}
\end{align}
In addition, the definition of 
$\odot$, \eqref{h02b}, and
\eqref{h10} show that
\begin{align} \begin{split} 
&
\dim \left(\calD(\Phi_{f_\varepsilon}) \odot 
\calD(\Phi^{d,0,K,\varepsilon}_{n,m,0})
\odot
\calD(\calX^{d,0,K,\varepsilon,0}_{T})
\right)\\
&=
\dim \left(\calD(\Phi_{f_\varepsilon}) 
\right)+
\dim \left( 
\calD(\Phi^{d,0,K,\varepsilon}_{n,m,0})
\right)
+
\dim \left(
\calD(\calX^{d,0,K,\varepsilon,0}_{T})
\right)-2\\
&=
\dim \left(\calD(\Phi_{f_\varepsilon}) 
\right)+
\xeqref{h02b} \Bigl( (n+1)
{L}
+n(\dim(\calD (\Phi_{f_\varepsilon} )  ) -2)
+\dim(\calD (\Phi_{g^d_\varepsilon} )  ) -1
\Bigr)
+\xeqref{h10}
{L}-2\\
&=(n+1)\bigl(\dim(\calD(\Phi_{f_\varepsilon}))-2\bigr)+(n+2){L}+
\dim\left(\calD (\Phi_{g^d_\varepsilon})\right)-1.
\end{split}\end{align}
Furthermore,
 the fact that
$f_\varepsilon=\calR(\Phi_{f_{\varepsilon}})$, \cref{b03},
\eqref{h01b}, \eqref{h04b},
\eqref{h06}, \eqref{h07}, and \cref{m11b}
demonstrate for all
$\theta\in \Theta$, $i\in [1,m^{n+1}]\cap\Z$,
$t\in [0,T]$, $\ell\in [0,n-1]\cap\Z$ that
\begin{align} \begin{split} 
&
\Bigl(
f_\varepsilon\circ
U^{d,(\theta,\ell,i),K,\varepsilon}_{\ell,m}
\Bigr)
\Bigl(\mathfrak{T}_t^{(\theta,\ell,i)}(\omega), 
 X^{ d,(\theta,\ell,i), K,\varepsilon,t,\cdot }_{\mathfrak{T}_t^{(\theta,\ell,i)}(\omega)} ,\omega
\Bigr)\\&=
\Bigl(
f_\varepsilon\circ\mathrm{Id}_\R\circ
U^{d,(\theta,\ell,i),K,\varepsilon}_{\ell,m}
\Bigr)
\Bigl(\mathfrak{T}_t^{(\theta,\ell,i)}(\omega), 
 X^{ d,(\theta,\ell,i), K,\varepsilon,t,\cdot }_{\mathfrak{T}_t^{(\theta,\ell,i)}(\omega)} ,\omega
\Bigr)\\
&\quad\in 
\calR
\left(\left\{
\Phi\in\bfN\colon \calD(\Phi)=
\calD(\Phi_{f_\varepsilon}) \odot 
\mathfrak{n}_{(n-\ell)\bigl(\calD(\Phi_{f_\varepsilon}) -2+{L}\bigr)+1}
\odot 
\calD(\Phi^{d,0,K,\varepsilon}_{\ell,m,0})
\odot
\calD(\calX^{d,0,K,\varepsilon,0}_{T})
\right\}\right).\end{split}
\end{align}
In addition,  the definition of 
$\odot$, \eqref{k07}, \eqref{h02b}, and
\eqref{h10} show for all $\ell\in [0,n-1]\cap\Z$ that
\begin{align} \begin{split} 
&\dim\left(\calD(\Phi_{f_\varepsilon}) \odot 
\mathfrak{n}_{(n-\ell)\bigl(\calD(\Phi_{f_\varepsilon}) -2+{L}\bigr)+1}
\odot 
\calD(\Phi^{d,0,K,\varepsilon}_{\ell,m,0})
\odot
\calD(\calX^{d,0,K,\varepsilon,0}_{T})\right)
\\
&=
\dim \left(\calD(\Phi_{f_\varepsilon})  
\right)+
\dim \left(\mathfrak{n}_{(n-\ell)\bigl(\calD(\Phi_{f_\varepsilon}) -2+{L}\bigr)+1}\right)
 +\dim \left(
\calD(\Phi^{d,0,K,\varepsilon}_{\ell,m,0})
\right)+
\dim \left(
\calD(\calX^{d,0,K,\varepsilon,0}_{T})
\right)-3
\\
&=
\dim \left(\calD(\Phi_{f_\varepsilon})  
\right)+
\xeqref{k07} \Bigl((n-\ell)\bigl(\calD(\Phi_{f_\varepsilon}) -2+{L}\bigr)+1\Bigr)
 \\
&\quad +\xeqref{h02b}\Bigl(
(\ell+1)
{L}
+\ell(\dim(\calD (\Phi_{f_\varepsilon} )  ) -2)
+\dim(\calD (\Phi_{g^d_\varepsilon} )  )-1
\Bigr)+\xeqref{h10}{L}-3\\
&=(n+1)\bigl(\dim(\calD(\Phi_{f_\varepsilon}))-2\bigr)+(n+2){L}+
\dim\left(\calD (\Phi_{g^d_\varepsilon})\right)-1.
\end{split}
\end{align}
Furthermore,
 the fact that
$f_\varepsilon=\calR(\Phi_{f_{\varepsilon}})$, \cref{b03},
\eqref{h01b}, \eqref{h04b},
\eqref{h06}, \eqref{h07}, and \cref{m11b}
imply for all
$\theta\in \Theta$, $i\in [1,m^{n+1}]\cap\Z$,
$t\in [0,T]$, $\ell\in [1,n]\cap\Z$ that
\begin{align} \begin{split} 
&
\Bigl(
f_\varepsilon\circ
U^{d,(\theta,-\ell,i),K,\varepsilon}_{\ell-1,m}
\Bigr)
\Bigl(\mathfrak{T}_t^{(\theta,\ell,i)}(\omega), 
 X^{ d,(\theta,\ell,i), K,\varepsilon,t,\cdot }_{\mathfrak{T}_t^{(\theta,\ell,i)}(\omega)}(\omega) ,\omega
\Bigr)\\&=
\Bigl(
f_\varepsilon\circ\mathrm{Id}_\R\circ
U^{d,(\theta,-\ell,i),K,\varepsilon}_{\ell-1,m}
\Bigr)
\Bigl(\mathfrak{T}_t^{(\theta,\ell,i)}(\omega), 
 X^{ d,(\theta,\ell,i), K,\varepsilon,t,\cdot }_{\mathfrak{T}_t^{(\theta,\ell,i)}(\omega)}(\omega) ,\omega
\Bigr)\\
&\in 
\calR
\left(\left\{
\Phi\in\bfN\colon \calD(\Phi)=
\calD(\Phi_{f_\varepsilon}) \odot 
\mathfrak{n}_{(n-\ell+1)\bigl(\calD(\Phi_{f_\varepsilon}) -2+{L}\bigr)+1}
\odot 
\calD(\Phi^{d,0,K,\varepsilon}_{\ell-1,m,0})
\odot
\calD(\calX^{d,0,K,\varepsilon,0}_{T})
\right\}\right).\end{split}
\end{align}
In addition,  the definition of 
$\odot$, \eqref{k07}, \eqref{h02b}, and
\eqref{h10} show for all $\ell\in [1,n]\cap\Z$ that
\begin{align} \begin{split} 
&
\dim \left(
\calD(\Phi_{f_\varepsilon}) \odot 
\mathfrak{n}_{(n-\ell+1)\bigl(\calD(\Phi_{f_\varepsilon}) -2+{L}\bigr)+1}
\odot 
\calD(\Phi^{d,0,K,\varepsilon}_{\ell-1,m,0})
\odot
\calD(\calX^{d,0,K,\varepsilon,0}_{T})
\right)\\
&=\dim \left(\calD(\Phi_{f_\varepsilon})\right)
+\dim \left(\mathfrak{n}_{(n-\ell+1)\bigl(\calD(\Phi_{f_\varepsilon}) -2+{L}\bigr)+1}\right)\\&\quad 
+\dim \left(\calD(\Phi^{d,0,K,\varepsilon}_{\ell-1,m,0})\right)+\dim \left(\calD(\calX^{d,0,K,\varepsilon,0}_{T})\right)-3\\
&=
\dim \left(\calD(\Phi_{f_\varepsilon})\right)
+\xeqref{k07} \Bigl((n-\ell+1)\bigl(\calD(\Phi_{f_\varepsilon}) -2+{L}\bigr)+1\Bigr)\\
&\quad +\xeqref{h02b}\Bigl(
\ell
{L}
+(\ell-1)(\dim(\calD (\Phi_{f_\varepsilon} )  ) -2)
+\dim(\calD (\Phi_{g^d_\varepsilon} )  )-1
\Bigr)+\xeqref{h10}{L}-3
\\
&=(n+1)\bigl(\dim(\calD(\Phi_{f_\varepsilon}))-2\bigr)+(n+2){L}+
\dim\left(\calD (\Phi_{g^d_\varepsilon})\right)-1.
\end{split}\label{k11}\end{align}
Now,
\eqref{k09}--\eqref{k11} and \cref{b01b}
imply that there exists
$(\Phi^{d,\theta,K,\varepsilon}_{n+1,m,t})_{t\in[0,T], \theta\in \Theta}$ such that
$t\in[0,T]$, $\theta\in \Theta$, $x\in \R^d$
we have that
\begin{align} \begin{split} 
(\calR(\Phi^{d,\theta,K,\varepsilon}_{n+1,m,t}))(x)
&=
\frac{1}{m^{n+1}}\sum_{i=1}^{m^{n+1}}g^d_\varepsilon\Bigl( X^{ d,(\theta,0,-i), K,\varepsilon,t,x}_{T} (\omega)\Bigr)\\
&\quad+
\frac{1}{m}
\sum_{i=1}^{m}
\Bigl(
f_\varepsilon\circ
U^{d,(\theta,n,i),K,\varepsilon}_{n,m}
\Bigr)
\Bigl(\mathfrak{T}_t^{(\theta,\ell,i)}(\omega), 
 X^{ d,(\theta,\ell,i), K,\varepsilon,t,x}_{\mathfrak{T}_t^{(\theta,\ell,i)}(\omega)} (\omega),\omega
\Bigr)
\\
&\quad+
\sum_{\ell=0}^{n-1}
\frac{(T-t)}{m^{n+1-\ell}}
\sum_{i=1}^{m^{n+1-\ell}}\Bigl(
f_\varepsilon\circ
U^{d,(\theta,\ell,i),K,\varepsilon}_{\ell,m}
\Bigr)
\Bigl(\mathfrak{T}_t^{(\theta,\ell,i)}(\omega), 
 X^{ d,(\theta,\ell,i), K,\varepsilon,t,x}_{\mathfrak{T}_t^{(\theta,\ell,i)}(\omega)}(\omega) ,\omega
\Bigr)
\\
&\quad-
\sum_{\ell=1}^{n}
\frac{(T-t)}{m^{n+1-\ell}}
\sum_{i=1}^{m^{n+1-\ell}}\Bigl(
f_\varepsilon\circ
U^{d,(\theta,-\ell,i),K,\varepsilon}_{\ell-1,m}
\Bigr)
\Bigl(\mathfrak{T}_t^{(\theta,\ell,i)}(\omega), 
 X^{ d,(\theta,\ell,i), K,\varepsilon,t,x}_{\mathfrak{T}_t^{(\theta,\ell,i)}(\omega)}(\omega), \omega
\Bigr)\\
&=U^{d,\theta,K,\varepsilon}_{n+1,m}(t,x),
\end{split}\label{k17}
\end{align}
\begin{align}
\calD(\Phi^{d,\theta,K,\varepsilon}_{n+1,m,t})
=(n+1)\bigl(\dim(\calD(\Phi_{f_\varepsilon}))-2\bigr)+(n+2){L}+
\dim\left(\calD (\Phi_{g^d_\varepsilon})\right)-1,
\label{k16}\end{align}
and
\begin{align} \begin{split} 
\calD(\Phi^{d,\theta,K,\varepsilon}_{n+1,m,t})
&=\left[
\operatorname*{\boxplus}_{i=1}^{m^{n+1}}\left[
\mathfrak{n}_{(n+1)\bigl(\dim(\calD(\Phi_{f_\varepsilon}))-2+{L}\bigr)+1}
\odot\calD (\Phi_{g^d_\varepsilon})
\odot\calD(\calX^{d,0,K,\varepsilon,0}_{T})\right]\right]\\
&\quad \boxplus
\left[\operatorname*{\boxplus}_{i=1}^m
\left[
\calD(\Phi_{f_\varepsilon}) \odot 
\calD(\Phi^{d,0,K,\varepsilon}_{n,m,0})
\odot
\calD(\calX^{d,0,K,\varepsilon,0}_{T})
\right]\right]\\
&\quad
\boxplus
\left[
\operatorname*{\boxplus}_{\ell=0}^{n-1}
\operatorname*{\boxplus}_{i=1}^{m^{n+1-\ell}}
\left[\calD(\Phi_{f_\varepsilon}) \odot 
\mathfrak{n}_{(n-\ell)\bigl(\calD(\Phi_{f_\varepsilon}) -2+{L}\bigr)+1}
\odot 
\calD(\Phi^{d,0,K,\varepsilon}_{\ell,m,0})
\odot
\calD(\calX^{d,0,K,\varepsilon,0}_{T})
\right]
\right]\\
&\quad
\boxplus
\left[
\operatorname*{\boxplus}_{\ell=1}^{n}
\operatorname*{\boxplus}_{i=1}^{m^{n+1-\ell}}
\left[\calD(\Phi_{f_\varepsilon}) \odot 
\mathfrak{n}_{(n-\ell+1)\bigl(\calD(\Phi_{f_\varepsilon}) -2+{L}\bigr)+1}
\odot 
\calD(\Phi^{d,0,K,\varepsilon}_{\ell-1,m,0})
\odot
\calD(\calX^{d,0,K,\varepsilon,0}_{T})
\right]
\right].\end{split}\label{k14}
\end{align}
This shows for all
$t_1,t_2\in [0,T]$, $\theta_1,\theta_2\in \Theta$ that
\begin{align}
\calD(\Phi^{d,\theta_1,K,\varepsilon}_{n+1,m,t_1})=
\calD(\Phi^{d,\theta_2,K,\varepsilon}_{n+1,m,t_2}).\label{k15}
\end{align}
Furthermore, \cref{b04}, 
\eqref{k07}, \eqref{c01}, and \eqref{h09} imply that
\begin{align} \begin{split} 
&
\supnorm{\mathfrak{n}_{(n+1)\bigl(\dim(\calD(\Phi_{f_\varepsilon}))-2+{L}\bigr)+1}
\odot\calD (\Phi_{g^d_\varepsilon})
\odot\calD(\calX^{d,0,K,\varepsilon,0}_{T})}
\\
&\leq \max\left\{2d,
\supnorm{\mathfrak{n}_{(n+1)\bigl(\dim(\calD(\Phi_{f_\varepsilon}))-2+{L}\bigr)+1}
},
\supnorm{\calD (\Phi_{g^d_\varepsilon})},
\supnorm{\calD(\calX^{d,0,K,\varepsilon,0}_{T})}
\right\}\leq 
\xeqref{k07}\xeqref{c01}\xeqref{h09}
c_{d,\varepsilon}.\end{split}\label{k12}
\end{align}
Next,
\cref{b04},
\eqref{c01}, \eqref{h03b}, and
\eqref{h09} show that
\begin{align} \begin{split} 
&\supnorm{
\calD(\Phi_{f_\varepsilon}) \odot 
\calD(\Phi^{d,0,K,\varepsilon}_{n,m,0})
\odot
\calD(\calX^{d,0,K,\varepsilon,0}_{T})}\\
&
\leq \max\left\{
2d,
\supnorm{\calD(\Phi_{f_\varepsilon})},
\supnorm{
\calD(\Phi^{d,0,K,\varepsilon}_{n,m,0})},
\supnorm{
\calD(\calX^{d,0,K,\varepsilon,0}_{T})}
\right\}\leq 
\xeqref{c01}\xeqref{h03b}
\xeqref{h09}
c_{d,\varepsilon}(3m)^n.\end{split}
\end{align}
Furthermore, 
\cref{b04},
\eqref{c01}, \eqref{h03b}, and
\eqref{h09} 
demonstrate for all $\ell\in [0,n-1]\cap\Z$
that
\begin{align} \begin{split} 
&
\supnorm{
\calD(\Phi_{f_\varepsilon}) \odot 
\mathfrak{n}_{(n-\ell)\bigl(\calD(\Phi_{f_\varepsilon}) -2+{L}\bigr)+1}
\odot 
\calD(\Phi^{d,0,K,\varepsilon}_{\ell,m,0})
\odot
\calD(\calX^{d,0,K,\varepsilon,0}_{T})}\\
&
\leq \max\left\{2d,
\supnorm{
\calD(\Phi_{f_\varepsilon}) }
,
\supnorm{
\mathfrak{n}_{(n-\ell)\bigl(\calD(\Phi_{f_\varepsilon}) -2+{L}\bigr)+1}}
,
\supnorm{
\calD(\calX^{d,0,K,\varepsilon,0}_{T})}
\right\}\leq 
\xeqref{c01}\xeqref{h03b}
\xeqref{h09}
c_{d,\varepsilon}(3m)^\ell,
\end{split}\end{align}
In addition, 
\cref{b04},
\eqref{c01}, \eqref{h03b}, and
\eqref{h09} 
show for all $\ell\in [1,n]\cap\Z$ that
\begin{align} \begin{split} 
&
\supnorm{
\calD(\Phi_{f_\varepsilon}) \odot 
\mathfrak{n}_{(n-\ell+1)\bigl(\calD(\Phi_{f_\varepsilon}) -2+{L}\bigr)+1}
\odot 
\calD(\Phi^{d,0,K,\varepsilon}_{\ell-1,m,0})
\odot
\calD(\calX^{d,0,K,\varepsilon,0}_{T})}\\
&
\leq 
\max\left\{
2d,
\supnorm{
\calD(\Phi_{f_\varepsilon}) }
,
\supnorm{
\calD(\Phi^{d,0,K,\varepsilon}_{\ell-1,m,0})}
,
\supnorm{
\calD(\calX^{d,0,K,\varepsilon,0}_{T})}
\right\}\leq 
\xeqref{c01}\xeqref{h03b}
\xeqref{h09}
c_{d,\varepsilon}(3m)^{\ell-1}.\end{split}\label{k13}
\end{align}
Now,
\eqref{k14}, the triangle inequality, and  \eqref{k12}--\eqref{k13} imply for all
$t\in[0,T]$, $\theta\in \Theta$ that
\begin{align} \begin{split} 
&\supnorm{\calD(\Phi^{d,\theta,K,\varepsilon}_{n+1,m,t})}\\
&=
\sum_{i=1}^{m^{n+1}}
\supnorm{
\mathfrak{n}_{(n+1)\bigl(\dim(\calD(\Phi_{f_\varepsilon}))-2+{L}\bigr)+1}
\odot\calD (\Phi_{g^d_\varepsilon})
\odot\calD(\calX^{d,0,K,\varepsilon,0}_{T})
}
\\
&\quad +
\sum_{i=1}^m
\supnorm{
\calD(\Phi_{f_\varepsilon}) \odot 
\calD(\Phi^{d,0,K,\varepsilon}_{n,m,0})
\odot
\calD(\calX^{d,0,K,\varepsilon,0}_{T})}
\\
&\quad
+
\sum_{\ell=0}^{n-1}
\sum_{i=1}^{m^{n+1-\ell}}
\supnorm{\calD(\Phi_{f_\varepsilon}) \odot 
\mathfrak{n}_{(n-\ell)\bigl(\calD(\Phi_{f_\varepsilon}) -2+{L}\bigr)+1}
\odot 
\calD(\Phi^{d,0,K,\varepsilon}_{\ell,m,0})
\odot
\calD(\calX^{d,0,K,\varepsilon,0}_{T})}
\\
&\quad
+
\sum_{\ell=1}^{n}
\sum_{i=1}^{m^{n+1-\ell}}
\supnorm{\calD(\Phi_{f_\varepsilon}) \odot 
\mathfrak{n}_{(n-\ell+1)\bigl(\calD(\Phi_{f_\varepsilon}) -2+{L}\bigr)+1}
\odot 
\calD(\Phi^{d,0,K,\varepsilon}_{\ell-1,m,0})
\odot
\calD(\calX^{d,0,K,\varepsilon,0}_{T})}
\end{split}
\end{align}
and
\begin{align}
\begin{split}
&\supnorm{\calD(\Phi^{d,\theta,K,\varepsilon}_{n+1,m,t})}\\
&\leq \left[
\sum_{i=1}^{m^{n+1}}c_{d,\varepsilon} \right]+
\left[ \sum_{i=1}^{m}c_{d,\varepsilon}({3} m)^n\right]+
\left[ \sum_{\ell=0}^{n-1}\sum_{i=1}^{m^{n+1-\ell}}c_{d,\varepsilon}({3} m)^\ell\right]
+\left[ \sum_{\ell=1}^{n}\sum_{i=1}^{m^{n+1-\ell}}c_{d,\varepsilon}({3} m)^{\ell-1}\right]
\\
&= 
m^{n+1}c_{d,\varepsilon}+m c_{d,\varepsilon} (3m)^{n}+\left[\sum_{\ell=0}^{n-1}m^{n+1-\ell}c_{d,\varepsilon}(3m)^\ell\right]+\left[
\sum_{\ell=1}^{n}m^{n+1-\ell}c_{d,\varepsilon}(3m)^{\ell-1}\right]
\\
&= 
m^{n+1}c_{d,\varepsilon}\left[1+3^n+\sum_{\ell=0}^{n-1}3^\ell+\sum_{\ell=1}^{n}3^{\ell-1}\right]=
m^{n+1}c_{d,\varepsilon}\left[1+\sum_{\ell=0}^{n}3^\ell+\sum_{\ell=1}^{n}3^{\ell-1}\right]
\\
&\leq 
 cm^{n+1}\left[1+2\sum_{\ell=0}^{n} {3} ^\ell\right]= cm^{n+1}\left[1+2\frac{{3}^{n+1}-1}{{3}-1}\right]
= c_{d,\varepsilon}({3} m)^{n+1}.
\end{split}
\label{b18}
\end{align}
This, \eqref{k15}, \eqref{k16}, the definition of ${L}$ (see \eqref{h10}),  
and \eqref{k17} completes the induction step. The proof of \cref{k04} is thus completed.
\end{proof}

\section{Proof of \cref{k41}}
\label{s05}


In this section, we show the proof of our main result, \cref{k41}.

\begin{proof}[Proof of \cref{k41}]
First, by \cref{lemma f epsilon} we notice that there exist
$f_\varepsilon\in C(\R,\R)$ and
$\Phi_{f_\varepsilon}\in\bfN$ satisfying
\eqref{k65}--\eqref{growth f epsilon}.
Observe that
\eqref{k59}, \eqref{k60}, \eqref{k65}, and \eqref{k66}
show for all $d\in\N$, $\varepsilon\in (0,1)$ that
\begin{align}2d+
\supnorm{\calD(\Phi_{\beta^d_\varepsilon})}
+
\supnorm{\calD(\Phi_{\sigma^d_\varepsilon,0})}
+
\supnorm{\calD(\Phi_{F^d_\varepsilon,0})}
+\supnorm{\calD(\Phi_{g^d_\varepsilon})}
+\supnorm{\calD(\Phi_{f_\varepsilon})}
\leq bd^c\varepsilon^{-c}\label{k59b}
\end{align} and
\begin{align}
\dim(\calD(\Phi_{\beta^d_\varepsilon}))+ \dim(\calD(\Phi_{\sigma^d_\varepsilon,0}))+
 \dim(\calD(\Phi_{F^d_\varepsilon,0}))
+
\dim(\calD(\Phi_{g^d_\varepsilon}))+\dim(\calD(\Phi_{f_\varepsilon}))\leq bd^c\varepsilon^{-c}.
\label{k60b}
\end{align}
Furthermore, by \cref{x01,k32} we have for all
$d,K\in \N$, $\varepsilon\in(0,1)$,
$t\in[0,T]$, $x\in \R^d$ that 
\begin{align}
\max \left\{
\E \!\left[d^c+\left\lVert X^{d,0,K,\varepsilon,t,x}_{s}\right\rVert^2
\right],
\E \!\left[d^c+\left\lVert X^{d,0,\varepsilon,t,x}_{s}\right\rVert^2
\right]
,
\E \!\left[d^c+\left\lVert X^{d,0,t,x}_{s}\right\rVert^2
\right]\right\}
\leq  (d^c+\lVert x\rVert^2)e^{7c(s-t)}.
\label{k75}
\end{align}
For every $d\in \N$, $\varepsilon\in (0,1)$
let 
$U^{d,\theta,K,\varepsilon}_{n,m}\colon [0,T]\times\R^d\times \Omega\to \R$, 
$\theta\in \Theta$,  
$n,m\in \Z$, 
satisfy for all 
$\theta\in \Theta$,  
$n\in \N_0$, $m\in \N$, $t\in [0,T]$, $x\in \R^d$ that
\begin{align} \begin{split} &U^{d,\theta,K,\varepsilon}_{n,m}(t,x)=
\frac{\1_{\N}(n)}{m^n}\sum_{i=1}^{m^n}g^d_\varepsilon\Bigl( X^{ d,(\theta,0,-i), K,\varepsilon,t,x}_{T} \Bigr)\\
&\quad
+\sum_{\ell=0}^{n-1}
\frac{(T-t)}{m^{n-\ell}}
\sum_{i=1}^{m^{n-\ell}}\Bigl(
f_\varepsilon\circ
U^{d,(\theta,\ell,i),K,\varepsilon}_{\ell,m}
-\1_{\N}(\ell)
f_\varepsilon\circ
U^{d,(\theta,-\ell,i),K,\varepsilon}_{\ell-1,m}\Bigr)
\Bigl(\mathfrak{T}_t^{(\theta,\ell,i)}, 
 X^{ d,(\theta,\ell,i), K,\varepsilon,t,x}_{\mathfrak{T}_t^{(\theta,\ell,i)}} 
\Bigr).\end{split}
\end{align}
For every $ d,K\in\N $,
$\varepsilon\in (0,1)$
let $u^{d,K,\varepsilon},u^{d,\varepsilon},u^d\colon [0,T]\times \R^d\to\R$ 
be measurable functions introduced in 
\eqref{FK u K epsilon}, \eqref{FK u epsilon}, and \eqref{FK u}, respectively.

Next, let $c_1,c_2\in \R$, 
$(\varepsilon_{d,\epsilon})_{d\in\N,\epsilon\in(0,1)}\subseteq\R$,
$(N_{d,\epsilon})_{d\in\N,\epsilon\in(0,1)}\subseteq \N$, 
$(C_\delta)_{\delta\in (0,1)}\subseteq [0,\infty]$ 
satisfy for all $d\in\N$, $\delta,\epsilon\in (0,1)$
that
\begin{align}
c_1=6(c T^{-1})^{\frac{1}{2}}+
12 c^{\frac{3}{2}}(T+2)e^{21cT+5cT^2}T^\frac{1}{2}+23c
e^{24cT+5cT^2}, \quad c_2=2cc_1,\label{x82}
\end{align}
\begin{align}\label{x84}
c_2d^{2c}\lvert \varepsilon_{d,\epsilon}\rvert^\frac{1}{2}
=\frac{\epsilon}{2},\quad
N_{d,\epsilon}=\min\left\{n\in\N\cap[2,\infty)
\colon  c_2d^{2c}\left(\frac{e^{12cTn+\frac{n}{2}}}{n^\frac{n}{2}}
+\frac{1}{n^\frac{n}{2}}\right) \leq \frac{\epsilon}{2}\right\},
\end{align}
and
\begin{align}
C_\delta=\sup_{n\in [2,\infty)}\left[
\left(\frac{e^{(12cT+0.5)(n-1)}}{(n-1)^\frac{n-1}{2}}\right)^{6+\delta}
(3n)^{3n+1}\right].\label{x85}
\end{align}
Then the triangle inequality, \cref{x01,k32,a31} imply for all
$d,K\in \N$, $\theta\in\Theta$, $\varepsilon\in (0,1)$, $n,m\in \N$ that
\begin{align} \begin{split} 
&
\left(
\E\!\left[\left\lvert
{U}^{d,\theta,K,\varepsilon}_{n,m}(t,x)
-u^d(t,x)\right\rvert^2\right]
\right)^\frac{1}{2}\\
&
\leq 
\left(
\E\!\left[\left\lvert
{U}^{d,\theta,K,\varepsilon}_{n,m}(t,x)
-u^{d,K,\varepsilon}(t,x)
\right\rvert^2\right]
\right)^\frac{1}{2}
+\left\lvert
u^{d,K,\varepsilon}(t,x)-u^{d,\varepsilon}(t,x)\right\rvert+
\left\lvert u^{d,\varepsilon}(t,x)
-u^d(t,x)\right\rvert
\\
&
\leq  6
e^\frac{m}{2}m^{-\frac{n}{2}}e^{12 c Tn}
(cd^c T^{-1})^{\frac{1}{2}}\left(d^c+\lVert x\rVert^2\right)^\frac{1}{2}\\
&\quad 
+12 c^{\frac{3}{2}}d^{\frac{c}{2}}(T+2)e^{21cT+5cT^2}
(d^c+\lVert x\rVert^2)^\frac{1}{2}
\frac{T^\frac{1}{2}}{K^\frac{1}{2}}+
 23cd^c\varepsilon^{\frac{1}{2}}
(d^c+\lVert x\rVert^2)
e^{24cT+5cT^2}\\
&
\leq\left[6(cd^c T^{-1})^{\frac{1}{2}}+
12 c^{\frac{3}{2}}d^{\frac{c}{2}}(T+2)e^{21cT+5cT^2}T^\frac{1}{2}
+23cd^ce^{24cT+5cT^2}
\right] (d^c+\lVert x\rVert^2)\left(\varepsilon^\frac{1}{2}+\frac{e^{12cTn+\frac{m}{2}}}{m^\frac{n}{2}}+\frac{1}{K^\frac{1}{2}}\right)\\
&
\leq\left[6(c T^{-1})^{\frac{1}{2}}+
12 c^{\frac{3}{2}}(T+2)e^{21cT+5cT^2}T^\frac{1}{2}+23c
e^{24cT+5cT^2}
\right]d^c (d^c+\lVert x\rVert^2)\left(\varepsilon^\frac{1}{2}+\frac{e^{12cTn+\frac{m}{2}}}{m^\frac{n}{2}}+\frac{1}{K^\frac{1}{2}}\right)\\
&
\leq\xeqref{x82} {c}_1d^{c}
\left(\varepsilon^\frac{1}{2}+\frac{e^{12cTn+\frac{m}{2}}}{m^\frac{n}{2}}+\frac{1}{K^\frac{1}{2}}\right)
(d^c+\lVert x\rVert^2).
\end{split}
\end{align}
Hence, the triangle inequality, \eqref{x61}, and \eqref{x82} demonstrate for all
$d,K\in \N$, $\theta\in\Theta$, $\varepsilon\in (0,1)$, $n,m\in \N$ that
\begin{align} \begin{split} 
&
\left(
\int_{\R^d}
\E\!\left[\left\lvert
{U}^{d,\theta,K,\varepsilon}_{n,m}(t,x)
-u^d(t,x)\right\rvert^2\right]\Gamma^d(dx)\right)^\frac{1}{2}\\&\leq 
c_1d^{c}
\left(\varepsilon^\frac{1}{2}+\frac{e^{12 cTn+\frac{m}{2}}}{m^\frac{n}{2}}+\frac{1}{K^\frac{1}{2}}\right)
\left(
\int_{\R^d}
(d^c+\lVert x\rVert^2)^2\,\Gamma^d(dx)
\right)^\frac{1}{2}\\
&\leq 
{c}_1d^{c}
\left(\varepsilon^\frac{1}{2}+\frac{e^{12 cTn+\frac{m}{2}}}{m^\frac{n}{2}}+\frac{1}{K^\frac{1}{2}}\right)
\left(d^c+ \left(\int_{\R^d}\lVert x\rVert^4\,\Gamma^d(dx)\right)^\frac{1}{2}\right)\\
&\leq 
{c}_1d^{c}
\left(\varepsilon^\frac{1}{2}+\frac{e^{12 cTn+\frac{m}{2}}}{m^\frac{n}{2}}+\frac{1}{K^\frac{1}{2}}\right)\xeqref{x61}2cd^c
=\xeqref{x82}c_2d^{2c}\left(\varepsilon^\frac{1}{2}+\frac{e^{12 cTn+\frac{m}{2}}}{m^\frac{n}{2}}+\frac{1}{K^\frac{1}{2}}\right).
\end{split}
\end{align}
This, Fubini's theorem, and \eqref{x84} show for all $d\in\N$, $\epsilon\in(0,1)$ that
\begin{align} \begin{split} 
&
\left(
\E\!\left[
\int_{\R^d}
\left\lvert
{U}^{d,\theta,n^n,\varepsilon}_{n,n}(t,x)
-u^d(t,x)\right\rvert^2\Gamma^d(dx)\right]\right)^\frac{1}{2}\Bigr|_{\substack{n=N_{d,\epsilon}\\ \varepsilon=\varepsilon_{d,\epsilon} } }\\
&=
\left(
\int_{\R^d}
\E\!\left[\left\lvert
{U}^{d,\theta,n^n,\varepsilon}_{n,n}(t,x)
-u^d(t,x)\right\rvert^2\right]\Gamma^d(dx)\right)^\frac{1}{2}\Bigr|_{\substack{n=N_{d,\epsilon}\\ \varepsilon=\varepsilon_{d,\epsilon} } }\\
&
\leq c_2d^{2c}\left(\varepsilon^\frac{1}{2}+\frac{e^{12 cTn+\frac{n}{2}}}{n^\frac{n}{2}}+\frac{1}{n^\frac{n}{2}}\right)\Bigr|_{\substack{n=N_{d,\epsilon}\\ \varepsilon=\varepsilon_{d,\epsilon} } }\\&\leq \frac{\epsilon}{2}+
\frac{\epsilon}{2}=\epsilon.
\end{split}
\end{align}
Then for all $d\in \N$, $\epsilon\in (0,1)$ there exists $\omega_{d,\epsilon}\in \Omega$ such that 
\begin{align}
\int_{\R^d}
\left\lvert
{U}^{d,\theta,n^n,\varepsilon}_{n,n}(t,x,\omega_{d,\epsilon})
-u^d(t,x)\right\rvert^2\Gamma^d(dx)\Bigr|_{\substack{n=N_{d,\epsilon}\\ \varepsilon=\varepsilon_{d,\epsilon} } }\leq \epsilon^2.\label{x56}
\end{align}
Next, \cref{k04}, \eqref{k60b}, and \eqref{k59b} imply for all
$d\in \N$,
$\epsilon\in (0,1)$ that there exists
$\Psi_{d,\epsilon}\in \bfN$ such that the following items hold.
\begin{enumerate}[(A)]
\item We have for all
$t\in [0,T]$, $\theta\in \Theta$ that
\begin{align} \begin{split} 
\dim(\calD(\Psi_{d,\epsilon} ))&
=(n+1)
\left[
n^n \left( \max\left\{\dim(\calD(\Phi_{\beta^d_\varepsilon  })),
\dim(\calD(\Phi_{\sigma^d_\varepsilon  }))
 \right\} -1\right)+1\right]\\&\quad 
+n(\dim(\calD (\Phi_{f_\varepsilon} )  ) -2)
+\dim(\calD (\Phi_{g^d_\varepsilon} )  ) -1\Bigr|_{n=N_{d,\epsilon},\varepsilon=\varepsilon_{d,\epsilon}}\\
&\leq\xeqref{k60b} 2n n^n b d^c\varepsilon^{-c}+n
b d^c\varepsilon^{-c}+b d^c\varepsilon^{-c}\bigr|_{n=N_{d,\epsilon}}\leq 4 n^n nb d^c\varepsilon^{-c}\bigr|_{n=N_{d,\epsilon},\varepsilon=\varepsilon_{d,\epsilon}}.
\end{split}\label{x88}\end{align}
\item We have for all
$t\in [0,T]$, $\theta\in \Theta$ that
\begin{align} \begin{split} 
\supnorm{\calD(\Psi_{d,\epsilon})}&\leq \Bigl(2d+
\supnorm{\calD(\Phi_{\beta^d_\varepsilon})}
+
\supnorm{\calD(\Phi_{\sigma^d_\varepsilon,0})}\\
&\quad \qquad
+
\supnorm{\calD(\Phi_{F^d_\varepsilon,0})}
+\supnorm{\calD(\Phi_{g^d_\varepsilon})}
+\supnorm{\calD(\Phi_{f_\varepsilon})}\Bigr) (3n)^n\bigr|_{n=N_{d,\epsilon},\varepsilon=\varepsilon_{d,\epsilon}}\\
&\leq\xeqref{k59b} bd^c\varepsilon^{-c}(3n)^n\bigr|_{n=N_{d,\epsilon},\varepsilon=\varepsilon_{d,\epsilon}}.
\end{split}\label{x89}\end{align}
\item We have for all
$t\in [0,T]$, $\theta\in \Theta$, $x\in \R^d$ that
$
U^{d,\theta,n^n,\varepsilon}_{n,n}(t,x,\omega_{d,\epsilon})\Bigr|_{\substack{n=N_{d,\epsilon}\\ \varepsilon=\varepsilon_{d,\epsilon} } }=(\calR(\Psi_{d,\epsilon}))(x).
$
\end{enumerate}
This and \eqref{x56} show for all
$d\in \N$,
$\epsilon\in (0,1)$ that
\begin{align}
\int_{\R^d}
\left\lvert(\calR(\Psi_{d,\epsilon}))(x)
-u^d(t,x)\right\rvert^2\Gamma^d(dx)\leq \epsilon^2.\label{k92b}
\end{align}
Furthermore, the fact that $\forall\, \Phi\in\bfN\colon \calP(\Phi)\leq 2\dim (\calD(\Phi))\supnorm{\calD(\Phi)}^2$, \eqref{x88}, and \eqref{x89} imply for all $d\in \N$,
$\epsilon\in (0,1)$ that
\begin{align} \begin{split} 
&
\calP(\Psi_{d,\epsilon})\leq 2\dim (\calD(\Psi_{d,\epsilon}))\supnorm{\calD(\Psi_{d,\epsilon})}^2\leq 2\cdot\xeqref{x88} 4n^nnb d^c\varepsilon^{-c} \left(\xeqref{x89}bd^c\varepsilon^{-c}(3n)^n\right)^2\bigr|_{n=N_{d,\epsilon}}
\\
&=8 n^n nb^2 d^{3c}\lvert\varepsilon_{d,\epsilon}\rvert^{-3c}(3n)^{2n}\bigr|_{n=N_{d,\epsilon},\varepsilon=\varepsilon_{d,\epsilon}}
\end{split}\label{x92}
\end{align}
Recall that in \eqref{x84} we have for all $d\in \N$, $\epsilon\in (0,1)$ that 
$
c_2d^{2c}\lvert \varepsilon_{d,\epsilon}\rvert^\frac{1}{2}=\frac{\epsilon}{2}$. Hence, for all $d\in \N$, $\epsilon\in (0,1)$ we have that
$ \varepsilon_{d,\epsilon}=\frac{\epsilon^2}{4}\lvert c_2\rvert^{-2}d^{-4c} $. This, \eqref{x92}, and \eqref{x85} show for all $d\in\N$, $\epsilon,\delta\in(0,1)$ that 
\begin{align} \begin{split} 
&\calP(\Psi_{d,\epsilon})
\leq 8n^n n b^2 d^{3c}\lvert\varepsilon_{d,\epsilon}\rvert^{-3c}(3n)^{2n}
\Bigr|_{n=N_{d,\epsilon}}\\&\leq 8 b^2 d^{3c}\lvert\varepsilon_{d,\epsilon}\rvert^{-3c} (3n)^{3n+1}
\Bigr|_{n=N_{d,\epsilon}}\\
&\leq 
8 b^2 d^{3c}\left[\frac{\epsilon^2}{4}\lvert c_2\rvert^{-2}d^{-4c} \right]^{-3c} (3n)^{3n+1}
\Bigr|_{n=N_{d,\epsilon}}\\&\leq 4^{3c+2}b^2d^{3c+12c^2}\lvert c_2\rvert^{6c}\epsilon^{-6c}
(3N_{d,\epsilon})^{3N_{d,\epsilon}+1}\\
&\leq 
4^{3c+2}b^2d^{3c+12c^2}\lvert c_2\rvert^{6c}\epsilon^{-6c-6-\delta}
\epsilon^{6+\delta}
(3N_{d,\epsilon})^{3N_{d,\epsilon}+1}\\
&\leq 4^{3c+2}b^2d^{3c+12c^2}\lvert c_2\rvert^{6c}\epsilon^{-6c-6-\delta}
\left(\xeqref{x84}\frac{4c_2 d^{2c}e^{(12 cT+0.5)(N_{d,\epsilon}-1)}}{(N_{d,\epsilon}-1)^\frac{N_{d,\epsilon}-1}{2}}\right)^{6+\delta}
(3N_{d,\epsilon})^{3N_{d,\epsilon}+1}\\
&\leq 
 4^{3c+8+\delta}b^2d^{3c+12c^2+2c(6+\delta)}\lvert c_2\rvert^{6c+2c(6+\delta)}\epsilon^{-6c-6-\delta}
\left(\frac{e^{(12 cT+0.5)(N_{d,\epsilon}-1)}}{(N_{d,\epsilon}-1)^\frac{N_{d,\epsilon}-1}{2}}\right)^{6+\delta}
(3N_{d,\epsilon})^{3N_{d,\epsilon}+1}\\
&\leq 4^{3c+8+\delta}b^2d^{3c+12c^2+2c(6+\delta)}\lvert c_2\rvert^{6c+2c(6+\delta)}\epsilon^{-6c-6-\delta}C_\delta.
\end{split}\end{align}
Thus, \eqref{k92b}, the fact that $c_2$ does not depend on $d$ (see \eqref{x82}), and the fact that $\forall\,\delta\in (0,1)\colon C_\delta<\infty$ (cf. (171) in \cite{CHW2022})
complete the proof of \cref{k41}.\end{proof}

\section{Conclusion}

In this paper we prove that under moderate conditions 
deep neural networks with ReLU activation function 
can approximate the unique viscosity solution of the semilinear parabolic 
PIDE \eqref{PIDE} without curse of dimensionality
in the sense that the required number of parameters in the deep neural networks
increases at most polynomially in both the dimension $ d $ of PIDE \eqref{PIDE} 
and the reciprocal of the prescribed accuracy~$ \epsilon $. 
For future work, one may also attempt to show that similar constructions of DNNs can be applied
to approximately solve high-dimensional non-linear inverse problems 
(see, e.g. \cite{pineda2023deep}) 
and forward-backward stochastic differential equations
(see, e.g. \cite{bao2021data}) without the curse of dimensionality.

{\bibliographystyle{acm}
\bibliography{References}
}

\end{document}